\documentclass[a4,11pt,usenames, reqno]{amsart}
\usepackage[utf8]{inputenc} 
\usepackage[english]{babel}
\usepackage{mathtools}
\usepackage[dvipsnames]{xcolor}
\usepackage[super]{nth}
\usepackage[b]{esvect}
\usepackage[colorlinks, linkcolor=blue, filecolor=blue,
     citecolor = olive, urlcolor=purple]{hyperref}
\usepackage{amssymb}
\usepackage{amsthm}
\allowdisplaybreaks
\usepackage{mathrsfs}
\usepackage{scalerel}
\usepackage{orcidlink}

\usepackage{dsfont}
\usepackage{amsmath,amssymb,amsthm,amsfonts,enumerate,url,mathbbol,mathrsfs,tikz,graphicx,pifont,doi,comment}
\usetikzlibrary{shapes,arrows,decorations.pathmorphing,backgrounds,fit,positioning,shapes.symbols,chains,graphs,graphs.standard}
\usetikzlibrary{arrows,decorations.pathmorphing,backgrounds,fit,positioning,shapes.symbols,chains}
\usetikzlibrary{decorations.pathmorphing}
\usetikzlibrary{decorations.pathreplacing}
\usepackage[normalem]{ulem}
\usepackage{epigraph}	
\usepackage{tikz,graphicx}
\usetikzlibrary{automata,positioning,fit}
\usetikzlibrary{calc}

\usepackage{enumerate,url,float,lscape}
\usepackage{hyperref}

\usepackage[normalem]{ulem}

\allowdisplaybreaks

\usepackage{aliascnt}

\theoremstyle{plain}
\newtheorem{theorem}{Theorem}[section]

\newaliascnt{corollary}{theorem}
\newaliascnt{proposition}{theorem}
\newaliascnt{lemma}{theorem}

\newtheorem{lemma}[lemma]{Lemma}
\newtheorem{proposition}[proposition]{Proposition}
\newtheorem{corollary}[corollary]{Corollary}

\aliascntresetthe{corollary}
\aliascntresetthe{proposition}
\aliascntresetthe{lemma}

\theoremstyle{definition} 

\newaliascnt{definition}{theorem}
\newaliascnt{assum}{theorem}
\newaliascnt{assums}{theorem}
\newaliascnt{conjecture}{theorem}
\newaliascnt{conv}{theorem}

\newtheorem{definition}[definition]{Definition}
\newtheorem{assum}[assum]{Assumption}

\aliascntresetthe{definition}
\aliascntresetthe{assum}
\aliascntresetthe{assums}
\aliascntresetthe{conjecture}
\aliascntresetthe{conv}

	\theoremstyle{rem}

\newaliascnt{rems}{theorem}
\newaliascnt{rem}{theorem}
\newaliascnt{exa}{theorem}
\newaliascnt{exs}{theorem}

\newtheorem{rem}[rem]{Remark}
\newtheorem{exa}[exa]{Example}

\aliascntresetthe{rems}
\aliascntresetthe{rem}
\aliascntresetthe{exa}
\aliascntresetthe{exs}

\numberwithin{equation}{section}
\numberwithin{lemma}{section}

 \def\mV{\mathsf{V}}
 \def\mE{\mathsf{E}}
 \def\mB{\mathsf{B}}

 \def\mv{\mathsf{v}}
 \def\me{\mathsf{e}}
 \def\mf{\mathsf{f}}
 \def\mw{\mathsf{w}}
 
 \def\mW{\mathsf{W}}
 \def\mf{\mathsf{f}}
 
 \def\erfc{\mathrm{erfc}}

\def\mVD{{\mV_{\mathrm{D}}}}
\def\mVN{{\mV_{\mathrm{N}}}}
\def\mvD{{\mv_{\mathrm{D}}}}
\def\mvN{{\mv_{\mathrm{N}}}}

\newcommand{\1}{\mathbf{1}}

\newcommand{\R}{\mathbb{R}}
\newcommand{\Pstv}{\mathcal{P}_{\mv \rightarrow}(\Graph)}
\newcommand{\Pstvpr}{\mathcal{P}_{\mv' \rightarrow}(\Graph)}
\newcommand{\Pendw}{\mathcal{P}_{\rightarrow \mw}(\Graph)}
\newcommand{\Pendwpr}{\mathcal{P}_{\rightarrow \mw'}(\Graph)}
\newcommand{\DeltaGD}{\Delta^{\Graph;\mVD}}
\newcommand{\ptGD}{p_t^{\Graph;\mVD}}
\newcommand{\QtGD}{{\mathcal Q}_t(\Graph;\mVD)}
\newcommand{\TrtGD}{{\mathrm{Tr}}_t(\Graph;\mVD)}

\newcommand{\C}{\mathbb{C}}
\newcommand{\N}{\mathbb{N}}
\newcommand{\heatcont}{{\mathcal Q}_t}
\newcommand{\heatcontm}{{\mathcal Q}_t^{\mathrm{NT}}}
\newcommand{\heatcontp}{{\mathcal Q}_t^{\mathrm{T}}}

\newcommand{\dd}{\,\mathrm{d}}
\newcommand{\e}{\mathrm{e}}

\DeclareFontEncoding{LS1}{}{}
\DeclareFontSubstitution{LS1}{stix}{m}{n}
\DeclareSymbolFont{stixletters}{LS1}{stix}{m}{it}
\DeclareMathAccent{\cev}{\mathord}{stixletters}{"91}
\DeclareMathAccent{\vec}{\mathord}{stixletters}{"92}

\newcommand{\Graph}{\mathcal{G}} 

\DeclareMathOperator{\dist}{dist}

\newcommand\Pst[1]{{\mathcal{P}_{#1 \rightarrow}(\Graph)}}
\newcommand\Pend[1]{{\mathcal{P}_{\rightarrow #1}(\Graph)}}

\counterwithin{figure}{section}

\title{On the heat content of compact quantum graphs} 

\subjclass[2010]{}


\author[P.~Bifulco]{Patrizio Bifulco \orcidlink{0009-0004-0628-374X}}
\author[D.~Mugnolo]{Delio Mugnolo \orcidlink{0000-0001-9405-0874}}

\address{Lehrgebiet Analysis, Fakult\"at Mathematik und Informatik, Fern\-Universit\"at in Hagen, D-58084 Hagen, Germany}
\email{patrizio.bifulco@fernuni-hagen.de}
\email{delio.mugnolo@fernuni-hagen.de}



\thanks{
The authors were partially supported by the Deutsche Forschungsgemeinschaft (Grant 397230547).  The second author is very grateful to Marvin Plümer for participating in the early discussions from which this article developed. Both authors warmly thank James B. Kennedy (Lisbon), Tommaso Rossi (Sorbonne) and Matthias Täufer (Valenciennes) for interesting discussions and helpful remarks.
}

\begin{document}

\begin{abstract}
We study the heat content for Laplacians on compact, finite metric graphs with Dirichlet conditions imposed at the ``boundary'' (i.e., a given set of vertices). We prove a closed formula of combinatorial flavour, as it is expressed as a sum over all closed orbits hitting the boundary. Our approach delivers a small-time asymptotic expansion that delivers information on crucial geometric quantities of the metric graph, much in the spirit of the celebrated corresponding result for manifolds due to Gilkey--van den Berg; but unlike other known formulae based on different methods, ours holds for all times $t>0$ and it displays stronger decay rate in the short time limit. Furthermore, we prove new surgery principles for the heat content and use them to derive comparison principles for the heat content between metric graphs of different topology.
\end{abstract}

\maketitle
\tableofcontents

\section{Introduction}
Differential operators on a metric graph $\Graph$ have become a popular toy model in mathematical physics and operator theory over the last two decades: self-adjoint operators on $L^2(\Graph)$ usually go under the name of \emph{quantum graphs} and we refer the reader to \cite{BerKuc13, Mug14, Kur23} for an introduction to different topics related to their theory. In particular, much attention has been lately devoted to \emph{spectral geometry}, i.e., to the interplay of a metric graph's topological and metric properties on the one hand, and the eigenvalues of the free Laplacian $\Delta^\Graph$ with standard (i.e., continuity and Kirchhoff) vertex conditions.

It is well-known that compact, finite metric graphs do not admit an intrinsic notion of boundary; it is then natural to interpret as boundary any given set $\mVD$ of vertices where Dirichlet conditions are imposed on the elements of the domain of $\Delta^{\Graph}$: we use the notation $\Delta^{\Graph;\mVD}$ to stress this dependence.

The \emph{heat kernel} of $\Graph$ with respect to $\mVD$ is, by definition, the fundamental solution of the heat equation associated with $\Delta^{\Graph;\mVD}$, i.e., the solution $(0,\infty) \times \Graph \times \Graph \ni (t,x,y) \mapsto  \ptGD(x,y)$ of the initial-boundary value problem
\[
\left\{
\begin{array}{rcll}
\frac{\partial}{\partial t} u(t,x) &=& \Delta^{\Graph;\mVD} u(t,x),\qquad& t>0,\ x\in \Graph,\\
u(t,\mv)&=&0, &t>0,\ \mv\in \mVD,\\
u(0,x)&=&\delta_y(x), &x\in \Graph.
\end{array}
\right.
\]
Several descriptions of the heat kernel are available: in particular, {Mercer's theorem} immediately yields that $p^{\Graph;\mVD}$ can be  described by means of the eigenpairs $(\lambda_k(\Graph;\mVD),\varphi_k)$ of $-\Delta^{\Graph;\mVD}$, i.e.,
\[
p_t^{\Graph;\mVD}(x,y)=\sum_{k=1}^\infty \e^{-t\lambda_k(\Graph;\mVD)}\varphi_k(x)\varphi_k(y),\qquad t>0,\ x,y\in\Graph.
\]
On the other hand, several properties of the heat kernel can be deduced by abstract semigroup theory, in particular using the fact that $p^{\Graph;\mVD}$ is known to be associated with a Dirichlet form \cite{KraMugSik07,Mug07}.
Finally, in the case of metric graphs a different formula of combinatorial flavour was derived by Roth in his pioneering work~\cite{Rot84} using the parametrix method, see Section~\ref{sec:combinatorial-formula} below. It turns out that these three approaches are mutually complementary: for example, the fact that $p_t^{\Graph;\mVD}$ is, for all $t>0$, a strictly positive function is an immediate consequence of the semigroup approach based on the Beurling--Deny theory, whereas Mercer's approach immediately implies that $p_t^{\Graph;\mVD}$ lies in $C(\Graph\times \Graph)$ for all $t>0$.

 Among other things, Roth used his formula for the heat kernel to derive a closed formula for the heat trace of $\Delta^{\Graph;\mVD}$: much in the spirit of the celebrated Gutzwiller trace formula, Roth's trace formula takes the form of an infinite sum over all closed orbits within $\Graph$ but is restricted to the case of no Dirichlet conditions ($\mVD=\emptyset$). A similar trace formula has been derived in \cite{Nic85} for the case of $\mVD\ne\emptyset$, too, but using different methods that circumvent the analysis of the heat kernel; later developments can be found -- among others -- in \cite{KotSmi97, KotSmi99,KurNow05,BolEnd09}.

Like on domains, manifolds and combinatorial graphs, the heat kernel delivers fine information on metric graphs $\Graph$, as is known since~\cite{Rot83}; however, in comparison with more common geometric environments the geometric implications of the heat kernel is less well understood: this is especially true for the \textit{heat content at time $t$}, i.e.,
\[
\heatcont(\Graph;\mVD):=\int_{\Graph}\int_{\Graph} \ptGD(x,y) \dd x \dd y,\qquad t>0,
\]
a mathematical object whose theory is nowadays very rich; it has been known since \cite{BerDav89} that a small-time asymptotics of the natural  counterpart $\mathcal Q_t(\Omega)$ of the heat content for a domain $\Omega$ yields information about the volume and the perimeter of $\Omega$: similar results have been sharpened and extended to many different contexts, including Riemannian manifolds \cite{BerGil94,Sav98} and combinatorial graphs \cite[Section~2.8]{MazSolTol23}.

 It comes to no surprise that the heat content or the heat kernel encode geometric information about the underlying space: for example, it follows from \cite[Theorem 2.1 and Theorem~3.1]{KelLenVog15} -- which hold for a rather general class of self-adjoint operators on metric measure spaces -- that
\begin{equation}\label{eq:kellenvog1}
\lim_{t\to \infty}
\frac{\log \heatcont(\Graph;\mVD)}{t}= -\lambda_1(\Graph;\mVD)
\end{equation}
and
\begin{equation}\label{eq:kellenvog}
\lim_{t\to \infty}\frac{\log \ptGD(x,y)}{t}=-\lambda_1(\Graph;\mVD)\qquad \hbox{holds for all }x,y\in\Graph,
\end{equation}
respectively. 
In the present article we will especially interested in the \textit{short}-time asymptotics of $\QtGD$, which arguably delivers finer information.

The theory of heat content is intertwined with the (much older) theory of torsional rigidity, i.e.,
\[
T(\Graph;\mVD):=\int_0^\infty {\mathcal Q}_t(\Graph;\mVD) \dd t,
\]
which goes back to Saint-Venant and was revived by Pólya in the 1940s~\cite{Pol48,PolWei50,PolSze51}. For example, it was shown in \cite[Corollary~4.1]{ColKagMcd16} that the isospectrality of certain pairs of non-isomorphic quantum graphs discovered in~\cite{BanParSha09}
can be resolved by their heat content, and in turn by the numerical sequence
\[
A_k:=k\int_0^\infty t^{k-1} \heatcont(\Graph;\mVD)\dd t,\qquad k\in\N,
\]
the so-called \textit{moment spectrum} of $\Graph$ \cite[Theorem~2.2]{ColKagMcd16}. A geometric theory of the torsional rigidity of quantum graphs has been developed since~\cite{MugPlu23}: the study of the heat content is made much more difficult by the fact that, unlike for $T(\Graph;\mVD)$, a variational characterisation of $\mathcal Q_t(\Graph;\mVD)$ is currently unavailable and, indeed, it seems unlikely to exist.

We circumvent this impediment and study the heat content of a quantum graph (more precisely, of a free Laplacian with Dirichlet conditions) by using a Roth-type closed formula as extended to a general class of vertex conditions in \cite{KosPotSch07}, see also \cite{CamCor17,BecGreMug21, BorHarJon22} for further extensions of its scope. 
Upper bounds for the heat kernel $p^{\Graph}_t(x,y)$ on $\Graph$ are known to hold -- see, e.g., \cite{Mug07,Hae14,BecGreMug21} --  and may be used, upon integrating them over $x,y$, to derive upper estimates for the heat content, too. In comparison with such approach, a direct application of Roth's formula has, however, the advantage of delivering a precise -- although combinatorially cumbersome -- formula for ${\mathcal Q}_t(\Graph;\mVD)$ at all $t>0$ and a.e.\ $(x,y)\in \Graph\times \Graph$ (more precisely: for all $x,y\in \Graph\setminus\mVD$) in terms of closed orbits hitting $\mVD$, i.e., of closed paths in $\Graph$ (possibly including multiple traversing of the same edge) that start at end at some $\mv\in\mVD$. Upon integrating Roth's formula (see \autoref{lem:path-sum-formula} below) for the heat kernel, and through careful bookkeeping of the contributes of such closed orbits, we are able to derive from the heat kernel formula a significantly more tractable formula for the heat content. The following is the main result of this article.

\renewcommand*{\thetheorem}{\Alph{theorem}}
\begin{theorem}\label{theo:A}
Given a finite metric graph $\Graph$ of volume $|\Graph|$ and upon imposing Dirichlet conditions on a set $\mVD$ consisting of $\#\mVD$ elements, the heat content of $\Delta^{\Graph;\mVD}$ satisfies
        \begin{align}\label{eq:heat-content-formula-intro}
        \begin{aligned}
\heatcont(\Graph;\mVD) &= \vert \Graph \vert - \frac{2\sqrt{t}}{\sqrt{\pi}} \# \mVD + 8\sqrt{t} \sum_{p} \alpha(p) \, H\bigg( \frac{\ell(p)}{2\sqrt{t}} \bigg) 
\end{aligned}
\end{align}
for all $t > 0$, where the summation in \eqref{eq:heat-content-formula-intro} runs over all paths starting and ending at some (possibly different) vertex in $\mVD$.
\end{theorem}

(Here, $\alpha(p)\in [-1,1]$ denotes a \textit{scattering coefficient} associated with each such path and $H:[0,\infty)\to (0,\infty)$ is a monotone, exponentially decaying function.)

Using \autoref{theo:A}, we will also derive two important asymptotic relations between the integrated heat kernel and the geometry of the underlying metric graph. We thus show that the heat content encodes interesting geometric information about the underlying graph $\Graph$: this is reminiscent of comparable information delivered by the torsional rigidity $T(\Graph;\mVD)$ and the ground state energy $\lambda_1(\Graph;\mVD)$, but it turns out that the heat content allows us to recover more  quantities.

\begin{theorem}[Heat content asymptotics]\label{theo:B}
If the minimal edge length of $\Graph$ is $\ell_{\min}$, the heat content satisfies
    \begin{align}\label{eq:small-time-asymp-intro}
    \mathcal{Q}_t(\mathcal{G};\mVD) = \vert \Graph \vert - \frac{2 \sqrt{t}}{\sqrt{\pi}}\# \mVD + \mathcal{O}\Big(\sqrt{t}\,\e^{-\frac{\ell_{\min}^2}{4t}}\Big) \qquad \text{as} \:\: t \rightarrow 0^+.
    \end{align}
\end{theorem}

\begin{theorem}[Caccioppoli-type description of perimeter]\label{theo:C}
If $\mathcal H$ is a closed and connected subset of \(\Graph\setminus\mVD\) whose boundary in $\Graph$ does not contain any vertices of $\Graph$ of degree $\ge 3$, then
    \begin{align}\label{eq:caccioppoli-intro}
        \lim_{t \rightarrow 0^+} \frac{\sqrt{\pi}}{\sqrt{t}} \int_\mathcal{H} \int_{\Graph \setminus \mathcal{H}} \ptGD(x,y) \dd y \dd x = \# \partial \mathcal{H} .
    \end{align}
\end{theorem}

These asymptotic results have classical counterpart on domains that go back to \cite[Section~6]{BerDav89} and \cite[Section~3]{MirPalPar07}, respectively. Let us add that the expansion
\[
\mathcal{Q}_t(\Omega) = \mathrm{vol}(\Omega) - \frac{2\sqrt{t}}{\sqrt{\pi}} \sigma_g(\partial \Omega) + \frac{t}{2} \int_{\partial \Omega} h \dd \sigma_g + o(t) \qquad \hbox{as $t \rightarrow 0^+$},
\]
was proved in \cite{BerGil94} for general open bounded subset $\Omega \subset M$ of a smooth Riemannian manifold $(M,g)$, with Dirichlet conditions at the boundary $\partial \Omega$; and recently extended in~\cite{CapRos24} to general $\mathsf{RCD}(K,N)$-spaces, where $h$ denotes the \emph{mean curvature} of $\partial \Omega$ in $M$ and $\sigma_g$ the \emph{surface measure} of $\partial \Omega$.

We present all necessary preliminary information about metric graphs and heat content in Section~\ref{sec:prelim}: all results in this section can be extended to general symmetric Dirichlet forms on finite measure spaces. 

Our approach is modular: we will first collect in Section~\ref{sec:comb-paths-g} two major combinatorial results that will allow us to conveniently describe a class of relevant paths in $\Graph$ and that do not depend on the structure of the integral kernel of the one-dimensional heat semigroup (i.e., the Gaussian kernel). Then, we show in Section~\ref{sec:combinatorial-formula} how a complicated formula arising from integration of a closed formula for the heat kernel can be simplified, making use of the results in Section~\ref{sec:comb-paths-g} and also of the explicit expression for the Gaussian kernel. This will eventually prove \eqref{eq:heat-content-formula-intro}, see \autoref{thm:heat-content-formula-bif-mug}. 

In Section~\ref{sec:small-time-asymptotics} we use fine asymptotic information about the function $H$ in~\eqref{eq:heat-content-formula-intro}, along with decomposition methods similar to those used in the proof of \autoref{theo:A}, to prove  \eqref{eq:small-time-asymp-intro} and \eqref{eq:caccioppoli-intro}, i.e., \autoref{thm:two-term-small-time-asymptotic-heat-content} and \autoref{thm:subgraph-small-time}, respectively.

Finally, motivated by the surgery principles developed in~\cite{KenKurMal16,BerKenKur19,MugPlu23} to derive spectral and torsional inequalities for quantum graphs, we conclude our paper in Section~\ref{sec:graph-surgery}  by deriving a few comparison principles for the heat content for different metric graphs upon appropriate graph operations. The Faber--Krahn-type inequality for the heat content recently proved in~\cite{BifTau25} in particular implies that for any equilateral metric graph with a Dirichlet condition on at least one vertex, and with at least one vertex of degree larger than 2, there exists a time $t_0$ such that its heat content is, at any time $t<t_0$, strictly smaller than the heat content of the interval of same volume with Dirichlet condition at one endpoint, at the same time $t$. We complement this observation by describing in \autoref{prop:loop-cut} an operation to construct pairs of non-isomorphic metric graphs with same heat content at all times.  Our comparison results also allow for a computation (rather than a mere estimate) of  the heat content  for certain non-trivial classes of equilateral metric graphs, see \autoref{exa:advan-sec-surg}.

Let us finally remark that the above mentioned modular approach may come in handy when deriving similar heat content formulae for Laplace-type operators on $L^2(\Graph)$ under different, non-standard vertex conditions (see \autoref{rem:general-vertex-conditions} below); and/or for operators on $L^2(\Graph)$ whose associated heat kernel $q^{\Graph}_\cdot(\cdot,\cdot)$ has a convenient path decomposition but such that the integral kernel $q^\R_\cdot(\cdot,\cdot)$ of the associated semigroup on $L^2(\R)$ is not the Gaussian kernel. (It was made precise in \cite{BecGreMug21} how a suitable heat kernel decay  -- which holds, for example, for a large class of Schrödinger operators -- is all that is needed to derive a heat kernel formula as a path summation.)

\renewcommand*{\thetheorem}{\arabic{theorem}}

\section{Preliminaries and first consequences}\label{sec:prelim}
\subsection{Notation}
Throughout this article let \(\Graph\) be a metric graph with edge set \(\mE=\mE_\Graph\) and vertex set \(\mV=\mV_\Graph\). We refer to \cite{Mug19} and references therein for a detailed introduction of $\Graph$ as  a metric measure space with respect to the shortest path metric $\dist_\Graph:\Graph\times \Graph\to [0,\infty)$ induced by the edgewise Euclidean distance and the one-dimensional Lebesgue measure $\dd x$. In particular,  two metric graphs $\Graph_1,\Graph_2$ can be considered as equivalent if there exists a \textit{primitive metric graph} $\Graph$ (not containing any vertex of degree 2) such that $\Graph_1,\Graph_2$ are subdivisions of $\Graph$: as will become clear soon, all relevant analytic properties we are interested in are invariant under metric graph subdivision, and for this reason it will be sufficient to formulate all our definitions and results for primitive metric graphs only; even though some of our constructions can be occasionally extended to a primitive metric graph's subdivisions, too, or even require to temporarily pass to an appropriate subdivision, pursue some steps on that level, and then go back to the primitive graph.

 For an edge \(\me\in \mE\), let \(\ell_\me \in (0,\infty)\) denote its \emph{edge length}.  If a vertex $\mv\in \mV$ is endpoint of an edge $\me$, we write $\me\sim \mv$: we denote by $\deg_\Graph(\mv)$ the \emph{degree} of $\mv$, i.e.,
 \[
 \deg_{\Graph}(\mv):=\#\{\me\in \mE:\me\sim \mv\}
 \]
(note that loops are allowed and each loop is counted as \textit{two} incident edges); if the underlying graph is clear from the context, we write $\deg(\mv)$ rather than $\deg_\Graph(\mv)$. It is remarkable that the degree of a vertex is a purely combinatorial quantity (that is, it is not robust under small perturbations of the metric graph) and, as such, it plays a minor role in the spectral theory of metric graphs: we will see in the following (and it was observed already in \cite{Rot84}) that things are rather different when it comes to heat kernels.
 
Throughout this article we impose the following.

\begin{assum}\label{ass:graph}
The metric graph $\Graph$ is compact and finite, i.e., it consists of finitely many edges, each of them having finite length; 
Also,  \(\Graph\) has at least one vertex of degree \(1\), and  \(\mV_{\mathrm{D},\Graph} =: \mV_\mathrm{D}\) is a set
such that
\begin{align}\label{eq:ass-dirichlet-nonempty}
\emptyset \neq \mVD\subset \{\mv\in \mV: \deg_\Graph(\mv)=1\} \subset \mV 
\end{align}
and such that $\Graph \setminus \mVD$ is connected. We assume without loss of generality that $\Graph$ has no vertices of degree 2, upon possibly substituting any two edges sharing an endpoint by one edge of combined length, as this operation does not change the properties of $\Graph$ as a metric measure space.
\end{assum}

Let \(\mV_\mathrm{N}=\mV_{\mathrm{N},\Graph}=:\mV\setminus \mV_{\mathrm{D},\Graph}\) be its complement in \(\mV\). 
We refer to \(\mVD\) as the set of \emph{Dirichlet vertices} of \(\Graph\) and to \(\mV_\mathrm{N}\) as the set of \emph{standard} (or \emph{natural}) \emph{vertices} of \(\Graph\).

\begin{rem}
Observe that because the metric space $\Graph$ is locally compact and separable, connectedness is equivalent to path connectedness, i.e., connectedness of $\Graph \setminus \mVD$ is equivalent to assuming that any two points in $\Graph\setminus\mVD$ can be connected by a path $\gamma\subset \Graph\setminus\mVD$.
\end{rem}

 Two vertices $\mv,\mw \in \mV$ are \emph{adjacent}, denoted with $\mv \sim \mw$, whenever there exists an edge $\me \in \mE$ such that both $\mv$ and $\mw$ are endpoints of $\me$.  
 
 Whenever a vertex $\mv \in \mV$ is an endpoint of an edge $\me \in \mE$ is identified with the left (resp., right) endpoint of $\me$, that is, to $0$ (resp., to $\ell_\me$), we write $\mv \overset{\me}{\rightsquigarrow}$ (resp., $\overset{\me}{\rightsquigarrow} \mv$).
Moreover, the \emph{volume} of $\Graph$ is defined as 
\[
\vert \Graph \vert := \sum_{\me \in \mE} \ell_\me.
\]
Given two edges $\me, \me' \in \mE$ we write $\me \sim \me'$ whenever there exists a vertex $\mv \in \mV$ such that $\me \sim \mv \sim \me'$. 
 
Let \(\DeltaGD\) be the Laplacian on \(\Graph\) with Dirichlet vertex conditions in \(\mVD\) and standard vertex conditions in \(\mV_\mathrm{N}\): more precisely, $\DeltaGD$ is the self-adjoint operator on 
\begin{equation}\label{eq:l2g}
L^2(\Graph) := \bigoplus_{\me \in \mE} L^2(0,\ell_\me)
\end{equation}
 associated with the quadratic form \(a=a_{\Graph}\) defined by
	\[a_{\Graph}(u):=\int_{\Graph}|u'(x)|^2\dd x\]
on the form domain
	\[H^1_0(\Graph;\mVD):=\left\{u=(u_\me)_{\me\in \mE}\in\bigoplus_{\me\in \mE} H^1(0,\ell_\me)~\bigg|~\begin{array}{l} u(\mv)=0\text{ for }\mv\in \mVD,\\ u\text{ is continuous in every }\mv\in \mV_\mathrm{N} \end{array}\right\}.\]
	(Here and in the following we write, in accordance with \eqref{eq:l2g}, 
	\[
	\int_\Graph f(x)\dd x:=\sum_{\me\in \mE}\int_0^{\ell_\me} f_\me(x)\dd x
	\]
	for any $f\in L^1(\Graph):= \bigoplus_{\me \in \mE} L^1(0,\ell_\me)$.)
\begin{rem}\label{rem:nodeg2}
Depending on its parametrization, a loop has an arbitrary number $n\in \N_0$ of vertices of degree 2, and no vertices of any other degree. How is this compatible with \autoref{ass:graph}? Let $\Graph_1,\Graph_2$ be a loop and an interval, respectively, of equal volume. Let $\mv$ be any point of $\Graph_1$ and $\mv_1,\mv_2$ be the endpoints of $\Graph_2$, respectively. Then clearly $H^1_0(\Graph_1;\{\mv\}) \simeq H^1_0(\Graph_2;\{\mv_1,\mv_2\})$, and for this reason the case of graphs with degree-2-vertices only can actually be excluded.
\end{rem}	

It is a well-known fact that $-\DeltaGD$ is given edgewise by the negative second derivative, more precisely,
\begin{align*}
D\big(\DeltaGD \big) &= \Bigg\{ f = (f_\me)_{\me \in \mE} \in \bigoplus_{\me \in \mE} H^1(0,\ell_\me) \: \left\vert \: \begin{array}{l} f_\me' \in H^1(0,\ell_\me), \: f(\mv) = 0 \text{ for all $\mv \in \mVD$} \\ f\text{ is continuous in every }\mv\in \mV \hbox{ and }\\ \sum\limits_{\me \sim \mv} \partial_\me f(\mv) = 0 \text{ for every $\mv \in \mV_\mathrm{N}$}\end{array}\right\}, \\ -\DeltaGD f &= (-f_\me'')_{\me \in \mE},
\end{align*}
where $\partial_\me f(\mv)$ denotes the \emph{inward directed derivative} of $f$ in $\me$ with respect to $\me \sim \mv$: in other words,
\[
 \partial_\mathsf{e}f(\mathsf{v}) := \begin{cases} 
  f_\mathsf{e}'(0), & \text{if }\mathsf{e}\text{ is not a loop and }\, \mathsf{v} \overset{\me}{\rightsquigarrow}, \\
  -f_\mathsf{e}'(\ell_\mathsf{e}), & \text{if }\mathsf{e}\text{ is not a loop and } \, \overset{\me}{\rightsquigarrow} \mathsf{v}, \\
  f_{\me}'(0) - f_{\me}'(\ell_\me), & \text{if }\me\text{ is a loop,}
  \end{cases} \qquad \mv \in \mV, \me \in \mE, \me \sim \mv.
\]

The following is well known, see 
\cite[Lemma~3.1]{KraMugSik07},
\cite[Theorem~3.5]{Mug07}
 and~\cite[Section~1.2]{Hae11}.

\begin{lemma}
The quadratic form $a_{\Graph}$ with domain $H^1_0(\Graph;\mVD)$ is a regular, strongly local Dirichlet form. Furthermore, the associated intrinsic metric agrees with the canonical shortest path metric induced by the edgewise Euclidean distance, i.e., $a_{\Graph}$ is even strictly local.
\end{lemma}

%
%
%

Accordingly, the associated operator $\DeltaGD$ generates a Markovian semigroup $(\e^{t\DeltaGD})_{t \geq 0}$ on $L^2(\Graph)$. Because this semigroup is analytic, it maps $L^2(\Graph)$ to the form domain $H^1_0(\mathcal G;\mVD)$, hence by Sobolev embedding and because of self-adjointness it maps $L^1(\Graph)$ to $L^\infty(\Graph)$ and by the Kantorovich--Vulikh Theorem it consists of kernel operators, i.e., for every $t> 0$ there exists $\ptGD \in L^\infty(\Graph \times \Graph)$
(and Lipschitz continuous, too, 
see~\cite[Theorem~5.3]{KosMugNic22} and \cite[Proposition~4.2] {BifMug23})
 such that for all $f\in L^1(\Graph)$
\[
\big(\e^{t\DeltaGD} f\big)(x)=\int_\Graph \ptGD(x,y)f(y) \dd y \qquad \text{for all $t >0$} \hbox{ and a.e.\ }x\in \Graph.
\]
We next introduce the main mathematical object of this paper.

\begin{definition}\label{def:heat-content}
Let $\Graph$ be a graph as in \autoref{ass:graph}. The quantity
	\[ 
	\QtGD := 
 \big\|\e^{t\DeltaGD}\1_\Graph \big\|_{L^1(\mathcal{G})},\qquad t> 0,
	\]
is called the \textit{heat content of \(\Graph\) (with respect to $\mVD$)}.
\end{definition}
\begin{rem}
At the risk of redundancy, let us stress that assuming that each element of $\mVD$ is a \underline{vertex of degree one} is crucial in our analysis. Indeed, it is of course, true (and well-known) that the Hilbert spaces $L^2(\Graph),L^2(\Graph')$, the form domains $H^1_0(\Graph;\mVD),H^1_0(\Graph';\mVD')$, as well as the forms $a_\Graph,a_{\Graph'}$ agree if $\Graph'$ is the metric graph obtained from $\Graph$ identifying all vertices in $\mVD$ to a unique vertex $\mv_0$, so that $\mVD'=\{\mv_0\}$ is a singleton; and in particular the Laplacians $\DeltaGD$ and $\Delta^{\Graph'; \mVD'}$ agree, so also their associated heat content will agree at each time $t$. However, our formalism is based on a definition of scattering coefficients, see \autoref{defi:scatt-coeff} below, that is tailored for the convention that only vertices of degree one can lie in $\mVD$.
\end{rem}

{ 
\begin{rem}\label{rem:direct-sum}
Under our standing assumptions, $\Graph\setminus \mVD$ is connected. This is not a major restriction: indeed, observe that if 
\[
\Graph\setminus \mVD=\bigsqcup_{i=1}^m \Graph_i
\]
then by definition
\[ 
	\QtGD = 
 \sum_{i=1}^m \big\|\e^{t\Delta^{\Graph_i; \mVD_i}}\1_{\Graph_i} \big\|_{L^1(\Graph_i)}=\sum_{i=1}^m \mathcal Q_t(\Graph_i;\mVD_i)\qquad\hbox{for all } t> 0,
	\]
	where $\mVD_i:=\mVD\cap \Graph_i$, $i=1,\ldots,m$.
	\end{rem}
}

Because the semigroup is positive, hence its heat kernel is for all $t>0$ a positive function, we see that
\[
\QtGD =\int_\Graph \int_\Graph  \ptGD(x,y)\dd x \dd y\qquad \hbox{for all }t> 0.
\]

It is also known that the $\DeltaGD$ has trace-class resolvent, hence pure point spectrum: we denote its eigenvalues by $-\lambda_k(\Graph;\mVD)$, $k\in \N$, (that is, $\lambda_k\ge 0$ for all $k\in \N$) and
the \textit{heat trace of $\mathcal{G}$ (with respect to $\mVD$)} by
\[
\TrtGD := \sum_{k=1}^\infty \e^{-t\lambda_k(\Graph;\mVD)}=\int_\Graph \ptGD(x,x)\dd x,\qquad t>0.
\]
Furthermore, a Poincaré inequality holds: indeed, the smallest eigenvalue $\lambda_1(\Graph;\mVD)$
satisfies
\begin{align}\label{eq:estimate-nicaise}
\frac{\pi^2}{4|\Graph|^2} \leq \lambda_1(\Graph;\mVD)
\end{align}
by \cite[Théorème~3.1]{Nic87}: in particular, the semigroup $(\e^{t\DeltaGD})_{t\ge 0}$ is exponentially stable.
(This holds at first for the Laplacian on $L^2(\Graph)$, and then also for its realization in $L^1(\Graph)$, by~\cite[\S~7.4.6]{Are06} and because $(\e^{t\DeltaGD})_{t\ge 0}$ is known to satisfy Gaussian estimates~\cite[Theorem~4.7]{Mug07}.)

\subsection{Some first observations}
The following are easy consequences of the spectral theorem: we formulate them for the case of Laplacians on compact metric graphs but it is immediate to check that they extend to the general case of metric measure spaces of finite measure such that the heat semigroup admits upper Gaussian estimates.

\begin{lemma}\label{prop:trivial-properties-heat-content}
   Under \autoref{ass:graph}, the following assertions hold:
    \begin{itemize}
        \item[(i)]\label{item:strictmonot} Let $(\varphi_n)_{n \in \mathbb{N}}$ be an orthonormal basis for $L^2(\mathcal{G})$ consisting of real-valued eigenfunctions of $\DeltaGD$. Then:
        \begin{align}\label{eq:heat-cont-eigenfunctions}
        \mathcal{Q}_t(\mathcal{G};\mVD) = \sum_{k=1}^\infty \e^{-t\lambda_k(\Graph;\mVD)} \bigg( \int_\mathcal{G} \varphi_k(x) \dd x \bigg)^2.
        \end{align}
        In particular, the map $t \mapsto \mathcal{Q}_t(\mathcal{G};\mVD)$ is strictly monotonically decreasing, i.e., 
        \[
        \mathcal{Q}_{s}(\mathcal{G};\mVD) < \mathcal{Q}_t(\mathcal{G};\mVD)\qquad \hbox{for any }s>t>0.
        \]
        \item[(ii)]\label{item:twosidedb} There holds 
        \begin{align}\label{eq:heat-content-lambda1}
        \e^{-t\lambda_1(\Graph;\mVD)} \langle \mathbf{1},\varphi_1 \rangle_{L^2(\Graph)}^2< \mathcal{Q}_t(\mathcal{G};\mVD) < \e^{-t\lambda_1(\Graph;\mVD)} \vert \mathcal{G} \vert \quad \text{for all $t>0$,}
        \end{align}
and in particular
        \begin{align}\label{eq:heat-content-trace}
        \frac{\mathcal{Q}_t(\mathcal{G};\mVD)}{\vert \mathcal{G} \vert} < \TrtGD\quad \text{for all $t>0$.}
        \end{align}
        \item[(iii)] 
        The map $(0,\infty) \ni t \mapsto \mathcal{Q}_t(\mathcal{G};\mV_{\mathrm{D}})$ belongs to $C_0^\infty(0,\infty)$ with
        \begin{align}\label{eq:heat-content-derivative}
        \frac{\mathrm{d}^n}{\mathrm{d}t^n} \mathcal{Q}_t(\mathcal{G};\mVD) = \sum_{k=1}^\infty  (-1)^n\lambda_k(\Graph;\mVD)^n \e^{-t\lambda_k(\Graph;\mVD)} \bigg( \int_\mathcal{G} \varphi_k(x) \dd x \bigg)^2.
        \end{align}
        Moreover, the map $(0,\infty) \ni t \mapsto \mathcal{Q}_t(\mathcal{G};\mV_{\mathrm{D}})$ belongs to $L^1(0,\infty)$ and
\begin{equation}\label{eq:heat-torsion}
        \int_0^\infty \mathcal{Q}_t(\mathcal{G};\mVD) \dd t =
         \|(-\DeltaGD)^{-1}\mathbf{1}\|_{L^1(\Graph)}=: T(\mathcal{G};\mVD).
\end{equation}
    \end{itemize}
\end{lemma}

We denote the quantity in \eqref{eq:heat-torsion} by $T(\mathcal{G};\mVD)$: it is the \emph{torsional rigidity} of $\mathcal{G}$ with respect to $\mVD$, see \cite{MugPlu23} and references therein.

\begin{proof}
(i) 
This proof follows almost verbatim \cite{BerDryKap14}: we present it for the sake of self-containedness.

According to Mercer's theorem, and because we can choose an orthonormal basis of $L^2(\Graph)$ consisting of real-valued eigenfunctions of $\DeltaGD$, the heat kernel $\ptGD$ can be written as
\begin{equation}\label{eq:mercer}
\begin{split}
\ptGD(x,y) = \sum_{k=1}^\infty \e^{-t \lambda_k(\Graph;\mVD)} \varphi_k(x) \varphi_k(y) \quad \text{for all $t>0$ and all $x,y \in \Graph$},
\end{split}
\end{equation}
with uniformly convergent right-hand side in $\Graph \times \Graph$. Fubini's theorem then implies
\begin{align}
\begin{aligned}
\heatcont (\Graph;\mVD)&= \int_\Graph \int_\Graph \ptGD(x,y) \dd y \dd x \\&= \sum_{k=1}^\infty \e^{-t\lambda_k(\Graph;\mVD)} \int_\Graph \int_\Graph \varphi_k(x) \varphi_k(y) \dd x \dd y \\&= \sum_{k=1}^\infty \e^{-t\lambda_k(\Graph;\mVD)} \left(\int_\Graph \varphi_k(x) \dd x\right)^2,
\end{aligned} \qquad \text{for all $t>0$,}
\end{align}
which is \eqref{eq:heat-cont-eigenfunctions}.

    (ii) Since $\e^{-t\lambda_k(\Graph;\mVD)} < \e^{-t\lambda_1(\Graph;\mVD)}$ for all $t > 0$ and all $k \ge 2$, it follows due to \eqref{eq:heat-cont-eigenfunctions} and Parseval's identity that
    \begin{align*}
    \heatcont(\Graph;\mVD) < \e^{-t\lambda_1(\Graph;\mVD)}\sum_{k=1}^\infty \langle \mathbf{1}, \varphi_k \rangle_{L^2(\Graph)}^2 = \e^{-t\lambda_1(\Graph;\mVD)} \Vert \mathbf{1} \Vert_{L^2(\Graph)}^2 = \e^{-t\lambda_1(\Graph;\mVD)} \vert \Graph \vert 
    \end{align*}
    for all $t>0$.
    
    (iii)    Let $t_0 > 0$: We can estimate for every $t \in [t_0,\infty)$ that
    \[
    \bigg\vert (-1)^n \lambda_k(\Graph;\mVD)^n \e^{-t\lambda_k(\Graph;\mVD)} \bigg( \int_\mathcal{G} \varphi_k(x) \dd x\bigg)^2 \bigg\vert \leq \lambda_k(\Graph;\mVD)^n \e^{-t_0\lambda_k(\Graph;\mVD)} \Vert \varphi_k \Vert_{L^2(\mathcal{G})}^2 \vert \mathcal{G} \vert
    \]
    due to the Cauchy--Schwarz inequality and $\Vert \varphi_k \Vert_{L^2(\mathcal{G})} = 1$ for every $k \in \mathbb N$. Accordingly, it follows that the series over the derivatives of the coefficients of the expansion \eqref{eq:heat-cont-eigenfunctions} for the heat content $\mathcal{Q}_t(\mathcal{G};\mVD)$ convergences absolutely and uniformly in $t > 0$. Therefore, deriving the right-hand side in \eqref{eq:heat-cont-eigenfunctions} yields \eqref{eq:heat-content-derivative}; whereas \eqref{eq:heat-torsion} is an immediate consequence of the fact that the inverse of $-\DeltaGD$ is the Laplace transform of the semigroup generated by $\DeltaGD$.
    
Because of the Poincaré inequality 
the map $t \mapsto \heatcont(\Graph;\mVD)$ indeed belongs to $L^1(0,\infty)$ according to \cite[Theorem~V.1.8]{EngNag00}.
Moreover, 
by taking Laplace transform one immediately observes
    \begin{align*}
    \int_0^\infty \heatcont(\Graph;\mVD) \dd t = \int_0^\infty \int_\Graph \e^{t\DeltaGD}\mathbf{1}(x) \dd x \dd t = \int_\Graph (-\DeltaGD)^{-1}\mathbf{1}(x) \dd x = T(\Graph;\mVD),
    \end{align*}
    which completes the proof.
\end{proof}

\begin{rem}\label{rem:several-prop}
(1) Also note that by Roth's trace formula, see \cite[Theoreme 1]{Rot84}, there holds
\begin{align}\label{eq:trace-formula}
\TrtGD = \frac{\vert \mathcal{G} \vert}{\sqrt{4 \pi t}}+ \frac{1}{2}( \# \mV_\mathrm{N} - \# \mE) + \sum_{C \in \mathcal{C}(\Graph)} \alpha(C) \ell(\widetilde{C}) \e^{-\frac{\ell(C)^2}{4t}},\qquad \hbox{for all }t>0,
\end{align}
whence by \eqref{eq:heat-content-trace} 
\[
\mathcal{Q}_t(\mathcal{G}; \mVD) \leq \frac{\vert \mathcal{G} \vert^2}{\sqrt{4 \pi t}}+ \frac{\vert \mathcal{G} \vert}{2}( \# \mV_\mathrm{N} - \# \mE) + \vert \mathcal{G} \vert\sum_{C \in \mathcal{C}(\Graph)} \alpha(C) \ell(\widetilde{C}) \e^{-\frac{\ell(C)^2}{4t}},\qquad \hbox{for all }t>0,
\]
where $\mathcal{C}(\Graph)$ denotes the set of \emph{periodic orbits} (or \emph{cycle}) in $\mathcal{G}$ and $\widetilde{C}$ the primitive orbit corresponding to some periodic orbit $C \in \mathcal{C}(\Graph)$.
This should be compared with the main result of the present article, \autoref{thm:heat-content-formula-bif-mug} below.

(2) Integrating over $t$ both sides in the upper estimate in \eqref{eq:heat-content-lambda1} yields the Pólya--Szegő inequality
\[
\lambda_1(\Graph;\mVD)T(\Graph;\mVD) < \vert \Graph \vert,
\]
cf.\ \cite[Proposition~5.1]{MugPlu23}.

(3) \label{prop:semigroup-prop-for-heat-cont}
The proof of \cite[Proposition~8]{Ber06} extends verbatim to our setting and yields the following:
One has $\mathcal{Q}_{2t}(\Graph;\mVD) = \Vert \e^{t\DeltaGD} \mathbf{1} \Vert_{L^2(\Graph)}^2$, $t>0$, and also
\begin{align*}
\mathcal{Q}_{kt}(\Graph;\mVD) \leq \frac{2^{k-2}}{k-1} \Bigg(\prod_{n=1}^{k-2} n^{-\frac{1}{n+1}} \Bigg) \Vert \e^{t\DeltaGD} \mathbf{1} \Vert_{L^k(\Graph)}^k \qquad \text{for all $3 \leq k \in \mathbb{N}$, and all $t>0$.}
\end{align*}
\end{rem}

The following large-time asymptotic for the heat content follows immediately from \autoref{prop:trivial-properties-heat-content}.

\begin{corollary}[Large-time asymptotic for the heat content]\label{cor:large-time-asymp}
Under \autoref{ass:graph} one has that
\[
\heatcont(\Graph; \mVD) = \e^{-\lambda_1(\Graph;\mVD)t} \langle \mathbf{1}, \varphi_1 \rangle_{L^2(\Graph)}^2 + \mathcal{O}\big(\e^{-\lambda_2(\Graph;\mVD)t}\big) \qquad \text{as $t \rightarrow \infty$.}
\]
\end{corollary}
In particular, for a fixed $\mVD$, modifying the metric graph $\Graph$ in such a way that $\lambda_1(\Graph;\mVD)$ increases makes the heat content decrease for large times $t>0$, and vice versa.
\begin{proof}
Using \autoref{prop:trivial-properties-heat-content}.(i), we observe that 
\begin{align*}
\frac{\heatcont(\Graph;\mVD)}{\e^{-\lambda_1(\Graph;\mVD)t}\langle \mathbf{1}, \varphi_1 \rangle_{L^2(\Graph)}^2} = 1 + \sum_{n=2}^\infty \e^{-\left(\lambda_n(\Graph;\mVD)-\lambda_1(\Graph;\mVD)\right)t} \frac{\langle \mathbf{1}, \varphi_n \rangle_{L^2(\Graph)}^2}{\langle \mathbf{1}, \varphi_1 \rangle_{L^2(\Graph)}^2} \rightarrow 1\qquad \hbox{ as }t \rightarrow \infty, 
\end{align*}
since $\langle \mathbf{1}, \varphi_n \rangle_{L^2(\Graph)} \leq \vert \Graph \vert \Vert \varphi_n \Vert_{L^2(\Graph)} = \vert \Graph \vert$ due to Cauchy--Schwarz inequality and therefore the second term on the right-hand side tends to $0$ as $t$ goes to infinity. Moreover, it is an immediate observation that the remainder  term indeed belongs to $\mathcal{O}(\e^{-\lambda_2(\DeltaGD)}t)$ which completes the proof.
\end{proof}



\subsection{Three elementary examples}\label{sec:three-examp}
According to \autoref{prop:trivial-properties-heat-content}.(i), the heat content can be fully determined whenever one has full (analytic) information about the eigenvalues and eigenfunctions. Thus, in the case of intervals and star graphs one can determine the heat content explicitly: \\
(1) Let $\Graph \simeq [0, \ell]$ be the trivial metric graph consisting of just one edge of length $0 < \ell < \infty$,  and let $\emptyset \subsetneq \mVD \subset \{ 0,\ell \}$: We hence look at an interval with mixed Dirichlet/Neumann conditions or purely Dirichlet conditions, respectively, see \autoref{fig:intervals}.
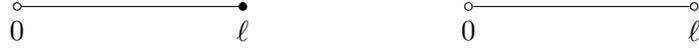
\begin{figure}[h]
        \begin{tikzpicture}[scale=0.6]
      \tikzset{enclosed/.style={draw, circle, inner sep=0pt, minimum size=.1cm, fill=black}, every loop/.style={}}
      \node[enclosed, label = {below: $\ell$}] (Z) at (0,2) {};
      \node[enclosed, label = {below: $0$}, fill=white] (A) at (-5,2) {};
      \node[enclosed, label = {below: $0$}, fill=white] (Z') at (5,2) {};
      \node[enclosed, label = {below: $\ell$}, fill=white] (A') at (10,2) {};
      \draw[-] (Z) edge node[above] {} (A) node[midway, above] (edge1) {};
      \draw[-] (Z') edge node[above] {} (A') node[midway, above] (edge2) {};
     \end{tikzpicture}
     \caption{The graph $\Graph$ with one Dirichlet condition at $0$ on the left and two Dirichlet conditions at $0$ and $\ell$ on the right (Dirichlet conditions are imposed at the white vertices).} \label{fig:intervals}
     \end{figure}
     
It is well-known that in this case the eigenvalues $(\lambda_k([0,\ell];\mVD))_{k \in \mathbb{N}} \subset (0,\infty)$ and corresponding normalized eigenfunctions $(\varphi_k)_{k \in \mathbb{N}} \subset L^2(0,\ell)$ are given by 
\begin{align*}
\lambda_k([0,\ell];\mVD) = \begin{cases} \frac{\pi^2(2k-1)^2}{4\ell^2}, & \text{if $\# \mVD = 1$}, \\ \frac{\pi^2k^2}{\ell^2}, & \text{if $\# \mVD = 2$, } \end{cases} \quad \text{and} \quad \varphi_k(x) = \begin{cases} \sqrt{\frac{2}{\ell}} \sin \big(\frac{\pi(2k-1)x}{2\ell} \big), & \text{if $\# \mVD = 1$}, \\ \sqrt{\frac{2}{\ell}} \sin \big(\frac{\pi k x}{2\ell} \big), & \text{if $\# \mVD = 2$, } \end{cases}
\end{align*}
for $k \in \mathbb{N}$, $x \in [0,\ell]$, and thus, for the heat content, we obtain -- using \eqref{eq:heat-cont-eigenfunctions} -- that
\begin{align}\label{eq:qt-interval-elem}
\mathcal{Q}_t([0,\ell];\mVD) = \frac{8 \ell}{\pi^2} \sum_{k=0}^\infty \e^{-t\big( \frac{\pi (2k+1) \# \mVD}{2\ell} \big)^2} \frac{1}{(2k+1)^2}, \qquad t > 0.
\end{align}
The profile of this function is shown in~\autoref{fig:plotqd-int}: observe that its limit for $t\to 0$ agrees with $\ell$. An educated guess, based on the small-time asymptotic expansion of the heat content on domain, suggests to study the renormalized function $\frac{\sqrt{\pi}}{2\sqrt{t}}\left(\mathcal{Q}_t([0,\ell];\mVD)-\ell\right)$: its profile, plotted in \autoref{fig:plotqd-int-norm}, shows an interesting behaviour as $t\to 0$, again.

\begin{figure}[h]
\includegraphics[width=5cm]{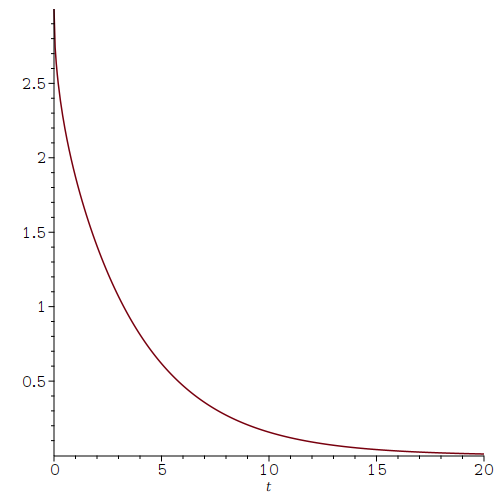} 
\includegraphics[width=5cm]{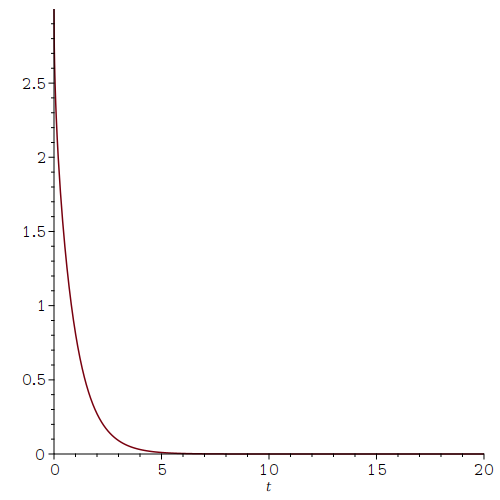} 
\caption{The profiles of $\mathcal{Q}_t([0,\ell];\mVD)$ for $\ell=3$ and $\# \mVD=1$ (left), $\# \mVD=2$ (right)...}\label{fig:plotqd-int}
\end{figure}

\begin{figure}[h]
\includegraphics[width=5cm]{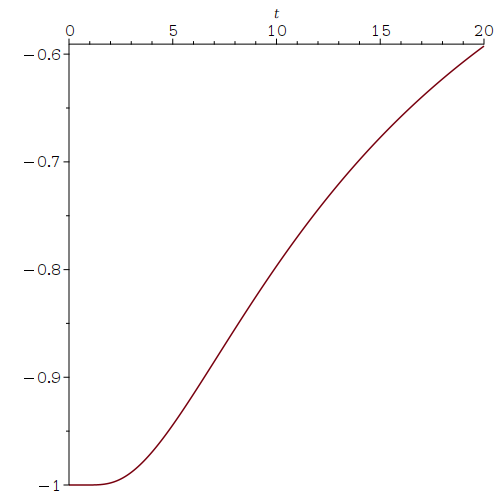} 
\includegraphics[width=5cm]{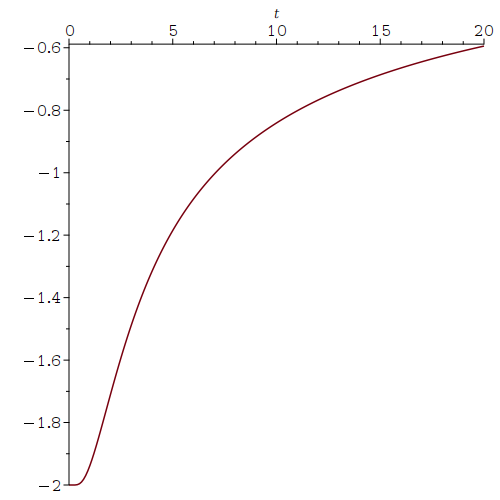} 
\caption{...and the profiles of $\frac{\sqrt{\pi}}{2\sqrt{t}}\left(\mathcal{Q}_t([0,\ell];\mVD)-\ell\right)$, again for $\ell=3$ and $\# \mVD=1$ (left), $\# \mVD=2$ (right)}\label{fig:plotqd-int-norm}
\end{figure}

(2) Let $\Graph \simeq \mathcal{S}_n(\ell) = \mathcal{S}_n$ be the equilateral star graph on $n+1$ vertices with Dirichlet conditions at all outer vertices, i.e., $\# \mVD = n$, and with same edge length $0<\ell<\infty$, see \autoref{fig:star}; for the sake of parametrization, we identify -- without any loss of generality -- the center of $\mathcal{S}_n(\ell)$ with the endpoint $\ell$ of each interval.
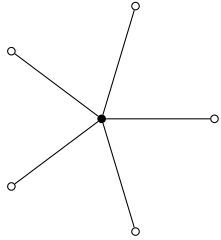
\begin{figure}[h]
\begin{tikzpicture}[scale=0.6]
      \tikzset{enclosed/.style={draw, circle, inner sep=0pt, minimum size=.1cm, fill=black}, every loop/.style={}}

      \node[enclosed] (Z) at (4,0) {};
      \node[enclosed, fill=white] (A) at (6.5,0) {};
      \node[enclosed, fill=white] (B) at (4.75,-2.5) {};
      \node[enclosed, fill=white] (C) at (4.75,2.5) {};
      \node[enclosed, fill=white] (D) at (2,-1.5) {};
      \node[enclosed, fill=white] (E) at (2,1.5) {};

      \draw (Z) edge node[above] {} (A) node[midway, above] (edge1) {};
      \draw (Z) edge node[above] {} (B) node[midway, above] (edge2) {};
      \draw (Z) edge node[above] {} (C) node[midway, above] (edge3) {};
      \draw (Z) edge node[above] {} (D) node[midway, above] (edge4) {};
      \draw (Z) edge node[above] {} (E) node[midway, above] (edge5) {};
     \end{tikzpicture}
     \caption{The star graph $\mathcal{S}_5$ with $5$ edges and Dirichlet conditions imposed at all white vertices.}\label{fig:star}
\end{figure}

We follow the same procedure in (1): A direct computation shows that the {odd} eigenvalues {(which turn out to be simple)} and normalized eigenfunctions of $\DeltaGD$ are given by
\begin{equation*}
\lambda_{2k-1}(\mathcal{S}_n;\mVD) =  \frac{\pi^2{\\
(2k-1)}^2}{4\ell^2}, \quad \text{and} \quad (\varphi_{2k-1})_\me(x) =  \sqrt{\frac{2}{n\ell}} \sin \Big(\frac{\pi {(2k-1)} x}{2\ell} \Big), \quad k \in \mathbb{N}, \, x \in \me \simeq [0,\ell];
\end{equation*}
{whereas the even eigenvalues have multiplicity $n-1$ and their corresponding normalized eigenfunctions can be chose to be antisymmetric and supported on exactly two edges. (In particular, these eigenfunctions are orthogonal to $\mathbf{1}$, thus leading to a vanishing term according to \eqref{eq:mercer}.)}

A tedious but elementary computation eventually yields the explicit expansion 
    \begin{align}\label{ex:heatcontent-stargraphs}
    \begin{aligned}
    \mathcal{Q}_t(\mathcal{S}_n;\mVD)  
    &= \frac{8\vert \mathcal{S}_n \vert}{\pi^2} \sum_{k=0}^\infty \e^{-t \big(\frac{\pi(2k+1)}{2\ell}\big)^2} \frac{1}{(2k+1)^2}
    \end{aligned}
    \end{align}
    of the heat content, see \autoref{fig:plotqd-starn}.

\begin{figure}[h]
\includegraphics[width=5cm]{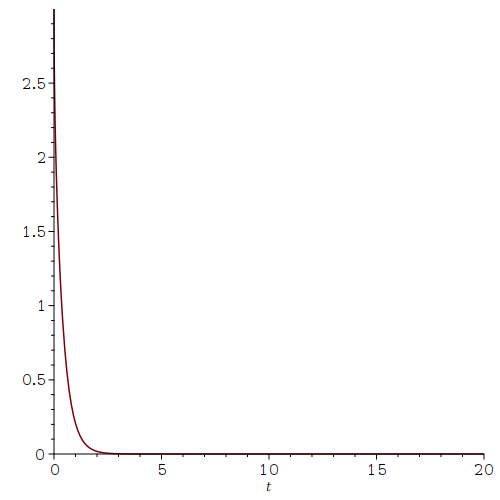} 
\includegraphics[width=5cm]{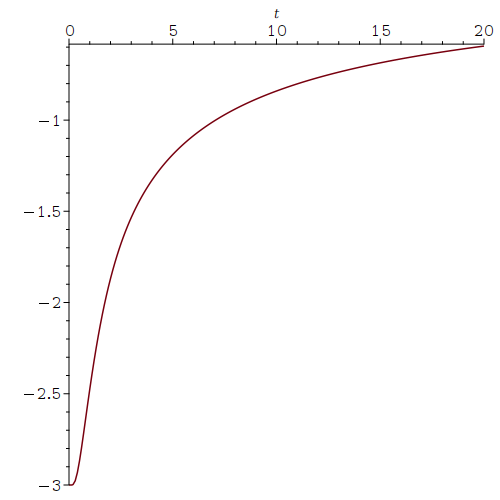} 
\caption{The profiles of $\mathcal{Q}_t(\mathcal S_n;\mVD)$ (left) and $\frac{\sqrt{\pi}}{2\sqrt{t}}\left(\mathcal{Q}_t(\mathcal S_3;\mVD)-3\ell\right)$ (right) for $\ell=1$ and $n=3$}\label{fig:plotqd-starn}
\end{figure}    

\section{Directed paths on metric graphs}\label{sec:comb-paths-g}
The proof of our main result, \autoref{thm:heat-content-formula-bif-mug}, relies upon a careful counting of directed paths between any two points in a metric graph. This section is devoted to the proof of two technical results, \autoref{lem:decomposition-lemma} and \autoref{lem:scattering-coeff-dual} that will crucially allow us to simplify certain terms that appear in the (integrated) path sum formula. We here follow the formalism used in
 \cite{BorHarJon22}.

Let $\Graph$ be a metric graph satisfying \autoref{ass:graph}. An oriented edge $\vec{\me}$ in $\Graph$ is called \emph{bond}: that is, each edge $\me \in \mE$ corresponds to exactly two bonds $\vec{\me},\cev{\me}$ running through $\me$ in opposite directions, and we write $\mathbf{b}(\me):=\{\vec{\me},\cev{\me}\}$.
Moreover, we denote by $\mathsf{B}$ the set of bonds in $\Graph$, i.e., 
\[
\mathsf{B} = \bigcup_{\me \in \mE} \mathbf{b}(\me).
\]

For any bond $\vec \me$, we denote its \emph{initial} vertex with $\partial^-(\vec \me)$ and its \emph{final} vertex with $\partial^+(\vec \me)$. 

\begin{definition}[Directed paths]
Two bonds $\vec{\me}_1$ and $\vec{\me}_2$ are called \emph{consecutive} if
\[
\partial^+(\vec{\me}_1) = \partial^-(\vec{\me}_2).
\]

Let $n \in \mathbb{N}_0$ and $\mv, \mw \in \mV$.
A \emph{directed path $\vec{p}$ from $\mv$ to $\mw$} (which is called in addition \emph{trivial}, if $n=0$ and $\mv = \mw$) is an ordered sequence
\[
\vec{p}=\left(\mv, \vec{\me}_1,\dots,\vec{\me}_n,\mw \right),
\]
such that $\vec{\me}_i,\vec{\me}_{i+1}$ are consecutive bonds for all $i=1,\ldots,n-1$, with $\mv = \partial^-(\vec{\me}_1)$ and $\mw = \partial^+(\vec{\me}_n)$.
Note that the trivial path can be seen as the path staying at a single vertex, avoiding to run through \textit{any} edge
\end{definition}
For any directed path $\vec{p} = (\mv,\vec{\me}_1,\dots,\vec{\me}_n,\mw)$, its \emph{initial} and \emph{final} vertex $\mv,\mw$ will also be denoted with $\mv_-(\vec{p}) =: \mv_-$ and $\mv_+(\vec{p}) =: \mv_+$, respectively. Also,
\begin{align*}
[0,\infty)\ni \ell(\vec{p}) := \begin{cases} \sum\limits_{k=1}^n \ell_{\me_k}, & \text{if $n \in \mathbb{N}$,} \\ 0, & \text{if $n=0$,} \end{cases}
\end{align*}
and
\[
\# \vec{p}  := n \in \mathbb{N}_0
\]
denote its \emph{metric length} and its \emph{combinatorial length}, respectively.
Additionally, for $\me \in \mE$, we denote by $\#_\me \vec{p} := \#\{ i= 1,\dots,n \: : \: \vec{\me_i} \in \mathbf{b}(\me) \}$ the number of bonds in $\vec{p}$ running through the edge $\me$ (in particular, $\# \vec{p}  = \sum_{\me \in \mE} \#_\me \vec{p}$); similarly, for $\mathsf{u} \in \mV$, we define 
\[
\#_{\mathsf{u}} \vec{p} := \# \left\{ i = 1,\dots,n-1 \: : \: \partial^+(\vec{\me}_i) = \mathsf{u} \right\} + \delta_{\mv_-(\vec{p}), \mathsf{u}} + \delta_{\mv_+(\vec{p}), \mathsf{u}}
\]
as the number of bonds having $\mathsf u$ as initial or/and final vertex. Finally, for any subset $\mW \subset \mV$ we let
\begin{align}\label{eq:hitting-number-subset-directed-path}
\#_\mW \hspace{0.02cm} \vec{p} := \sum_{\mathsf{u} \in \mW} \#_{\mathsf{u}} \vec{p}.
\end{align}
 
\begin{definition}\label{defi:scatt-coeff}
The \emph{scattering coefficient} of a nontrivial directed path $\vec{p} = (\mv,\vec{\me}_1,\dots,\vec{\me}_n,\mw)$ is defined by
\[
\alpha_\Graph(\vec{p}) := \left\{
\begin{array}{ll}
1 &\hbox{if }n=0 \hbox{ or }n=1,\\
 \prod\limits_{k=1}^{n-1} \beta(\vec{\me}_{k}, \vec{\me}_{k+1}) ,\quad &\hbox{if }n\ge 2,
\end{array}
\right.
\]
where
\begin{align}\label{eq:scattering-coefficients}
\beta(\vec{\me}_{k}, \vec{\me}_{k+1}) := \begin{cases}
  \frac{2}{\deg_\mathcal{G}(\partial^+(\vec{\me}_k))}-\delta_{{\vec{\me}_k},\vec{\me}_{k+1}}, & \text{if} \:\: \partial^+(\vec{\me}_k) \in \mV \setminus \mVD, \\ -1, & \text{if} \:\: \partial^+(\vec{\me_k}) \in \mVD,
\end{cases}
\end{align}
for $k=1,\dots,n-1$, cf.\ also \autoref{fig:fig-non-triv-dir-path} below.
\end{definition}

{ 
(If the relevant graph is clear from the context, we are going to simplify the notation and write $\alpha(\vec{p}):=\alpha_\Graph(\vec{p})$.)}
\begin{figure}[h]
        \begin{tikzpicture}[scale=0.6]
      \tikzset{enclosed/.style={draw, circle, inner sep=0pt, minimum size=.08cm, fill=black}, every loop/.style={}}

      \node[enclosed] (Z) at (0,2) {};
      \node[enclosed, label = {left: $\mv$}, fill=white] (A) at (-5,2) {};
      \node[enclosed, label = {right: $\mw$}] (B) at (2.25,4.25) {};
      \node[enclosed] (C) at (2.25,-0.25) {};
      
      \draw[-] (Z) edge node[above] {} (A) node[midway, above] (edge1) {};
      \draw[-] (Z) edge node[above] {} (B) node[midway, above] (edge2) {};
      \draw[-] (Z) edge node[above] {} (C) node[midway, above] (edge3) {};
      \draw[->, blue, thick] (A) edge[bend left = 50] node[below] {\textcolor{black}{$\vec{\me}_1$}} (Z) node[midway, above] (bond1) {};
      \draw[->, blue, thick] (Z) edge[bend left = 50] node[below] {\textcolor{black}{$\vec{\me}_2$}} (A) node[midway, above] (bond2) {};
      \draw[->, blue, thick] (A) edge[in=100, out=80, scale=500] node[above] {\textcolor{black}{$\vec{\me}_3$}} (Z) node[midway, above] (bond3) {};
      \draw[->, blue, thick] (Z) edge[bend right = 50] node[below right] {\textcolor{black}{$\vec{\me}_4$}} (B) node[midway, above] (bond3) {};
      
     \end{tikzpicture}
     \caption{A star graph $\mathcal{S}_3$ with $3$ edges and a Dirichlet condition together with a nontrivial directed path $\vec p = (\mv,\vec{\me}_1,\vec{\me}_2,\vec{\me}_3, \vec{\me}_4,\mw)$ (in blue) and corresponding scattering coefficient $\alpha(\vec p) = (\frac{2}{3}-1) \cdot (-1) \cdot \frac{2}{3} = \frac{2}{9}$.} \label{fig:fig-non-triv-dir-path}
     \end{figure}
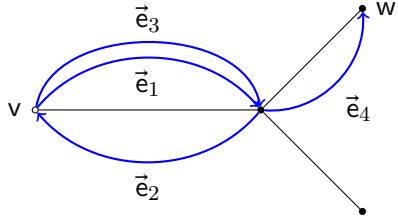

 For $\mv,\mw \in \mV$, we denote by $\mathcal{P}_m(\mv,\mw)$ the set of paths from $\mv$ to $\mw$ of combinatorial lengths $m \in \mathbb{N}$ and define
\begin{equation}\label{eq:def-set-of-paths}
\mathcal{P}_m(\Graph) := \bigcup_{\mv,\mw \in \mV} \mathcal{P}_m(\mv,\mw), \quad \text{and} \quad \mathcal{P}_{\geq n}(\Graph) := \bigcup_{\ell=n}^\infty \mathcal{P}_\ell(\Graph) \qquad \text{for $m,n \in \mathbb{N}_0$}.
\end{equation}
(In particular, $\mathcal{P}_{\geq 0}(\Graph)$ is the set of \emph{all} directed paths in $\Graph$).
  Moreover, we denote by
  \[
  \mathcal{P}(\Graph) := \mathcal{P}_{\geq 1}(\Graph)
  \]
the set of all nontrivial paths in $\Graph$.

Given now a nontrivial directed\ path $\vec{p} \in \mathcal{P}(\Graph)$, we denote by $\vec{\me}_-(\vec{p}) = \vec{\me}_-$ and $\vec{\me}_+(\vec{p}) = \vec{\me}_+$ the \emph{initial} and \emph{final} bond of $\vec p$, respectively; and with $\vec{\me}_-(\vec{p}) = \vec{\me}_-$ and $\vec{\me}_+(\vec{p}) = \vec{\me}_+$ the corresponding \emph{initial} and \emph{final} edge that is run through by $\vec p$. In particular, $\mv_-(\vec{p}) = \partial^-(\vec \me_-(\vec p ))$ and $\mv_+(\vec{p}) = \partial^+(\vec \me_+(\vec p ))$.
 This leads to the following definition, which we illustrate in~\autoref{fig:fig-non-triv-dir-path-minus-plus}.
 
\begin{definition}
Let $\vec{p} =(\mv_-,\vec{\me}_1,\dots,\vec{\me}_n,\mv_+) \in \mathcal{P}(\Graph)$. Then
\begin{itemize}
\item[(i)] $\vec{p}_- \in \mathcal{P}_{\geq 0}(\Graph)$ (resp., $\vec{p}_+ \in \mathcal{P}_{\geq 0}(\Graph)$) denotes the directed path obtained deleting the \emph{first} (resp., the \emph{last}) edge $\vec{\me}_-$ (resp., $\vec{\me}_+$) from the path $\vec{p}$,
\item[(ii)] $\cev{p} \in \mathcal{P}(\Graph)$ denotes the \emph{reversed path} which runs through every edge passed by $\vec p$ in the opposite direction, and opposite order, i.e.,
\[
\cev{p} = (\mv_+,\cev{\me}_n,\dots,\cev{\me}_1,\mv_-).
\]
\item[(iii)] $\vec{p}_\pm \in \mathcal{P}_{\geq 0}(\Graph)$ denotes the directed path defined by
\[
\vec{p}_{\pm} := \begin{cases} (\vec{p}_-)_+ = (\vec{p}_+)_-, & \text{if $\# \vec{p} > 1$,} \\ \cev{p}, & \text{if $\# \vec{p}  = 1.$} \end{cases}
\]
\end{itemize}
\end{definition}

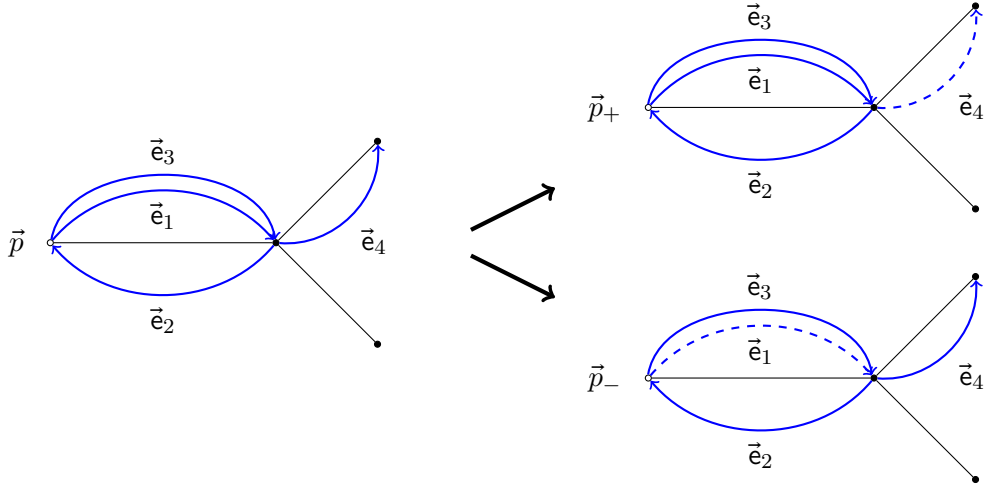
\begin{figure}[h]
        \begin{tikzpicture}[scale=0.6]
      \tikzset{enclosed/.style={draw, circle, inner sep=0pt, minimum size=.08cm, fill=black}, every loop/.style={}}

      \node[enclosed] (Z) at (0,2) {};
      \node[enclosed, label = {left: $\vec p \:\:$}, fill=white] (A) at (-5,2) {};
      \node[enclosed] (B) at (2.25,4.25) {};
      \node[enclosed] (C) at (2.25,-0.25) {};

       \node[enclosed] (Z') at (13.25,5) {};
      \node[enclosed, label = {left: $\vec p_+ \:\:$}, fill=white] (A') at (8.25,5) {};
      \node[enclosed] (B') at (15.5,7.25) {};
      \node[enclosed] (C') at (15.5,2.75) {};

      \node[enclosed] (Z'') at (13.25,-1) {};
      \node[enclosed, label = {left: $\vec p_- \:\:$}, fill=white] (A'') at (8.25,-1) {};
      \node[enclosed] (B'') at (15.5,1.25) {};
      \node[enclosed] (C'') at (15.5,-3.25) {};
      
      \node[enclosed,white] (X1) at (4.25,2.25) {};
      \node[enclosed,white] (Y1) at (6.25,3.25) {};
      \node[enclosed,white] (X2) at (4.25,1.75) {};
      \node[enclosed,white] (Y2) at (6.25,0.75) {};

      \draw[-] (Z) edge node[above] {} (A) node[midway, above] (edge1) {};
      \draw[-] (Z) edge node[above] {} (B) node[midway, above] (edge2) {};
      \draw[-] (Z) edge node[above] {} (C) node[midway, above] (edge3) {};
      \draw[->, blue, thick] (A) edge[bend left = 50] node[below] {\textcolor{black}{$\vec{\me}_1$}} (Z) node[midway, above] (bond1) {};
      \draw[->, blue, thick] (Z) edge[bend left = 50] node[below] {\textcolor{black}{$\vec{\me}_2$}} (A) node[midway, above] (bond2) {};
      \draw[->, blue, thick] (A) edge[in=100, out=80, scale=500] node[above] {\textcolor{black}{$\vec{\me}_3$}} (Z) node[midway, above] (bond3) {};
      \draw[->, blue, thick] (Z) edge[bend right = 50] node[below right] {\textcolor{black}{$\vec{\me}_4$}} (B) node[midway, above] (bond3) {};
      \draw[->, ultra thick] (X1) edge node[below right] {} (Y1) node[midway, above] (bond3) {};
      \draw[->, ultra thick] (X2) edge node[below right] {} (Y2) node[midway, above] (bond3) {};

       \draw[-] (Z') edge node[above] {} (A') node[midway, above] (edge1) {};
      \draw[-] (Z') edge node[above] {} (B') node[midway, above] (edge2) {};
      \draw[-] (Z') edge node[above] {} (C') node[midway, above] (edge3) {};
      \draw[->, blue, thick] (A') edge[bend left = 50] node[below] {\textcolor{black}{$\vec{\me}_1$}} (Z') node[midway, above] (bond1) {};
      \draw[->, blue, thick] (Z') edge[bend left = 50] node[below] {\textcolor{black}{$\vec{\me}_2$}} (A') node[midway, above] (bond2) {};
      \draw[->, blue, thick] (A') edge[in=100, out=80, scale=500] node[above] {\textcolor{black}{$\vec{\me}_3$}} (Z') node[midway, above] (bond3) {};
      \draw[->, blue, thick, dashed] (Z') edge[bend right = 50] node[below right] {\textcolor{black}{$\vec{\me}_4$}} (B') node[midway, above] (bond3) {};

       \draw[-] (Z'') edge node[above] {} (A'') node[midway, above] (edge1) {};
      \draw[-] (Z'') edge node[above] {} (B'') node[midway, above] (edge2) {};
      \draw[-] (Z'') edge node[above] {} (C'') node[midway, above] (edge3) {};
      \draw[->, blue, thick, dashed] (A'') edge[bend left = 50] node[below] {\textcolor{black}{$\vec{\me}_1$}} (Z'') node[midway, above] (bond1) {};
      \draw[->, blue, thick] (Z'') edge[bend left = 50] node[below] {\textcolor{black}{$\vec{\me}_2$}} (A'') node[midway, above] (bond2) {};
      \draw[->, blue, thick] (A'') edge[in=100, out=80, scale=500] node[above] {\textcolor{black}{$\vec{\me}_3$}} (Z'') node[midway, above] (bond3) {};
      \draw[->, blue, thick] (Z'') edge[bend right = 50] node[below right] {\textcolor{black}{$\vec{\me}_4$}} (B'') node[midway, above] (bond3) {};
      
     \end{tikzpicture}
     \caption{A nontrivial directed path $\vec p =\big(\mv_-(\vec p),\vec{\me}_1,\vec{\me}_2,\vec{\me}_3,\vec{\me}_4,\mv_+(\vec p) \big)$ consisting of \emph{four} bonds with $\vec{p}_+ = \big(\mv_-(\vec p_+) = \mv_-(\vec p) ,\vec{\me}_1,\vec{\me}_2,\vec{\me}_3, \mv_+(\vec p_+) = \partial^+(\vec \me_3)\big)$ and $\vec{p}_- = \big(\mv_-(\vec p_+) = \partial^-(\vec \me_2) ,\vec{\me}_2,\vec{\me}_3,\vec{\me}_4, \mv_+(\vec p_-) = \mv_+(\vec p) \big)$ on the right-hand side. Thus, the path $\vec{p}_{\pm} = \big(\mv_-(\vec{p}_\pm) = \partial^-(\vec \me_2), \vec \me_2, \vec \me_3, \mv_+(\vec{p}_\pm) = \partial^+(\vec \me_3) \big)$ consists of  the bonds $\vec{\me}_2,\vec{\me}_3$.} \label{fig:fig-non-triv-dir-path-minus-plus}
     \end{figure}

\begin{rem}\label{rem:alpha-ell-cevvec}
(1) The reason why we defined $\vec{p}_\pm$ as $\cev{p}$ whenever $\vec{p} \in \mathcal{P}_1(\Graph)$ is that formally we delete the first edge (which is at the same time the last edge) of $\vec{p}$ \emph{twice}, but since $\vec{p}_-$ (resp., $\vec{p}_+$) is already the trivial path, we see the next deletion as ``adding" the original edge with reversed direction. 

(2) Given a path $\vec{p} \in \mathcal{P}(\Graph)$, one has in particular that
\[
\ell(\vec{p}_-) = \ell(\vec{p}) - \ell_{{{\me}}_-(\vec{p})} \qquad \text{(resp., $\ell(\vec{p}_+) = \ell(\vec{p}) - \ell_{{{\me}}_+(\vec{p})}$)} 
\]
as well as
\[
\ell(\vec{p}_{\pm}) = \begin{cases} \ell(\vec{p}) - \ell_{{{\me}}_-(\vec{p})} - \ell_{{{\me}}_+(\vec{p})}, & \text{if $\# \vec{p} > 1$,} \\ \ell(\cev{p}) = \ell_{{{\me}}_-(\vv{p})} = \ell_{{{\me}}_+(\vec{p})}, & \text{if $\# \vec{p} = 1$.} \end{cases}
\]

(3) By definition, one has in addition 
\[
\mv_-(\cev{p}) = \mv_+(\vec{p})\quad \text{and} \quad \mv_+(\cev{p}) =\mv_-(\vec{p}).
\]
Furthermore,
\[
\ell(\cev{p}) = \ell(\vec{p}),\quad \alpha(\cev{p}) = \alpha(\vec{p})
\]
hold for any directed path $\vec{p} \in \mathcal{P}(\Graph)$, since \eqref{eq:scattering-coefficients} implies that $\beta(\vec{\me}_k,\vec{\me}_{k+1})=\beta(\cev{\me}_{k+1},\cev{\me}_{k})$ for any two consecutive bonds $\vec{\me}_k,\vec{\me}_{k+1}$; we are therefore occasionally going to write
 \[
\ell(p):=\ell(\cev{p}) = \ell(\vec{p}),\qquad \alpha(p):= \alpha(\cev{p}) = \alpha(\vec{p}).
 \] 
\end{rem}

Let now $\mv,\mw \in \mV$ and $\mW,\mW' \subset \mV$ be two subsets. We let
        \[
\Pstv := \bigcup_{\mw' \in \mV} \mathcal{P}(\mv,\mw'), \qquad \Pendw := \bigcup_{\mv' \in \mV} \mathcal{P}(\mv',\mw),
        \]
i.e., $\Pstv$ and $\Pendw$ are the sets of directed nontrivial paths in $\Graph$ beginning at $\mv$ and ending at $\mw$, respectively; and furthermore
        \[
       \Pst{\mW} := \bigcup_{\mv' \in \mW} \Pstvpr, \qquad \Pend{\mW} := \bigcup_{\mw' \in \mW} \Pendwpr
        \]
        as the set of all paths beginning or ending in the vertex set $\mW\subset \mV$.
We also introduce the set 
\[
\mathcal{P}_{\mW,\mW'}(\Graph) := \Pst{\mW} \cap \Pend{\mW'}\]
 of paths in $\Graph$ starting in $\mW$ and ending in $\mW'$, and we abbreviate 
 \[
\mathcal{P}_{\mW}(\Graph) := \mathcal{P}_{\mW, \mW}(\Graph);
\]
note that for any subset $\mW \subset \mV$, we have that
 $\mathcal{P}(\Graph) = \Pst{\mW} \sqcup \Pst{\mV \setminus \mW} = \Pend{\mW} \sqcup \Pend{\mV \setminus \mW}$.
         
Given now a nontrivial directed path $\vec{p} \in \mathcal{P}(\Graph)$, we call $\vec{q} \in \mathcal{P}(\Graph)$ a \emph{pre-extended} (resp., \emph{post-extended}) path for $\vec{p}$ whenever $\vec{q}_- = \vec{p}$ (resp., $\vec{q}_+ = \vec{p}$). The set of all pre-extended (resp., post-extended) directed paths to $\vec{p}$ will be denoted with $\langle \vec{p} \rangle_-$ (resp., $\langle \vec{p} \rangle_+$).
\begin{exa}
        Let $n \in \mathbb{N}$. Consider $\Graph \simeq \mathcal{S}_n$, a metric star graph on $n$ edges and, hence, $n$ outer vertices $\mW := \{ \mv_1,\dots,\mv_n \}$, and with and central vertex $\mv_0$, see \autoref{fig:p_+-1}. For any path $\vec{p} \in \Pend{\mv_0}$, i.e., a path ending at  $\mv_0$, we have that $\langle \vec{p} \rangle_+ = \{\vec{p}_1,\dots,\vec{p}_n \}$, where each $\vec{p}_i$ follows the same path as $\vec{p}$ with an additional \emph{final} bond $\vec{\me}_i$ with $\partial^-(\vec{\me}_i) = \mv_0$ and $\partial^+(\vec{\me}_i) = \mv_i$ for all $i=1,\dots,n$. 
        \begin{figure}[h]
        \begin{tikzpicture}[scale=0.6]
      \tikzset{enclosed/.style={draw, circle, inner sep=0pt, minimum size=.08cm, fill=black}, every loop/.style={}}

      \node[enclosed, label = {below: $\mv_0$}] (Z) at (0,2) {};
      \node[enclosed, label = {left: $\mv_1$}] (A) at (-3,2) {};
      \node[enclosed, label = {right: $\mv_2$}] (B) at (2.25,4.25) {};
      \node[enclosed, label = {right: $\mv_3$}] (C) at (2.25,-0.25) {};
      \node[enclosed,white] (X1) at (4.25,2.25) {};
      \node[enclosed,white] (Y1) at (6.25,2.25) {};
      \node[enclosed, label = {below: $\mv_0$}] (Z') at (11.25,2) {};
      \node[enclosed, label = {left: $\mv_1$}] (A') at (8.25,2) {};
      \node[enclosed, label = {right: $\mv_2$}] (B') at (13.5,4.25) {};
      \node[enclosed, label = {right: $\mv_3$}] (C') at (13.5,-0.25) {};
      
      \draw[-] (Z) edge node[above] {} (A) node[midway, above] (edge1) {};
      \draw[-] (Z) edge node[above] {} (B) node[midway, above] (edge2) {};
      \draw[-] (Z) edge node[above] {} (C) node[midway, above] (edge3) {};
      \draw[->, blue, thick] (B) edge[bend right = 50] node[above left] {\textcolor{black}{$\vec{\me}_+(\vec{p})$}} (Z) node[midway, above] (bond1) {};
      \draw[->, ultra thick] (X1) edge[] node[above left] {} (Y1) node[midway, above] (bond1) {};
      \draw[-] (Z') edge node[above] {} (A') node[midway, above] (edge1) {};
      \draw[-] (Z') edge node[above] {} (B') node[midway, above] (edge2) {};
      \draw[-] (Z') edge node[above] {} (C') node[midway, above] (edge3) {};
      \draw[->, blue, thick] (B') edge[bend right = 50] node[above left] {\textcolor{black}{$\vec{\me}_+(\vec{p})$}} (Z') node[midway, above] (bond1) {};
       \draw[->, blue, thick, dashed] (Z') edge[bend right = 50] node[right] {\textcolor{black}{$\:\: \vec{\me}_2$}} (B') node[midway, above] (bond1) {};
        \draw[->, blue, thick, dashed] (Z') edge[bend right = 50] node[above] {\textcolor{black}{$\vec{\me}_1$}} (A') node[midway, above] (bond1) {};
         \draw[->, blue, thick, dashed] (Z') edge[bend left = 50] node[right] {\textcolor{black}{$\:\: \vec{\me}_3$}} (C') node[midway, above] (bond1) {};
      
     \end{tikzpicture}
     \caption{The star graph $\mathcal{S}_3$ together with the final bond $\vec{\me}_+(\vec{p})$ of a path $\vec{p}$ ending at $\mv_0$ on the left and the corresponding final bonds $\vec{\me}_i$, $i=1,2,3$, of all paths in $\langle \vec{p} \rangle_+$ on the right.}
     \end{figure}\label{fig:p_+-1}
        
        Likewise, if $\vec{p} \in \Pend{\mW}$, i.e. $\vec{p}$ is a path ending at an outer vertex $\mv_i \in \mW$, $i=1,\dots,n$, then $\langle \vec{p} \rangle_+$ is a singleton consisting of a path $\vec{q}$ following the same path as $\vec{p}$ with an additional final bond starting at $\mv_i$ and ending at the central vertex $\mv_0$,  see \autoref{fig:p_+-2}. 
        \begin{figure}[h]
        \begin{tikzpicture}[scale=0.6]
      \tikzset{enclosed/.style={draw, circle, inner sep=0pt, minimum size=.08cm, fill=black}, every loop/.style={}}

      \node[enclosed, label = {below: $\mv_0$}] (Z) at (0,2) {};
      \node[enclosed, label = {left: $\mv_1$}] (A) at (-3,2) {};
      \node[enclosed, label = {right: $\mv_2$}] (B) at (2.25,4.25) {};
      \node[enclosed, label = {right: $\mv_3$}] (C) at (2.25,-0.25) {};
      \node[enclosed,white] (X1) at (4.25,2.25) {};
      \node[enclosed,white] (Y1) at (6.25,2.25) {};
      \node[enclosed, label = {below: $\mv_0$}] (Z') at (11.25,2) {};
      \node[enclosed, label = {left: $\mv_1$}] (A') at (8.25,2) {};
      \node[enclosed, label = {right: $\mv_2$}] (B') at (13.5,4.25) {};
      \node[enclosed, label = {right: $\mv_3$}] (C') at (13.5,-0.25) {};
      
      \draw[-] (Z) edge node[above] {} (A) node[midway, above] (edge1) {};
      \draw[-] (Z) edge node[above] {} (B) node[midway, above] (edge2) {};
      \draw[-] (Z) edge node[above] {} (C) node[midway, above] (edge3) {};
      \draw[<-, blue, thick] (B) edge[bend left = 50] node[below right] {\textcolor{black}{$\vec{\me}_+(\vec{p})$}} (Z) node[midway, above] (bond1) {};
      \draw[->, ultra thick] (X1) edge[] node[above left] {} (Y1) node[midway, above] (bond1) {};
      \draw[-] (Z') edge node[above] {} (A') node[midway, above] (edge1) {};
      \draw[-] (Z') edge node[above] {} (B') node[midway, above] (edge2) {};
      \draw[-] (Z') edge node[above] {} (C') node[midway, above] (edge3) {};
      \draw[<-, blue, thick] (B') edge[bend left = 50] node[below right] {\textcolor{black}{$\vec{\me}_+(\vec{p})$}} (Z') node[midway, above] (bond1) {};
       \draw[->, blue, thick, dashed] (B') edge[bend right = 50] node[above left] {\textcolor{black}{$\:\: \vec{\me}_+(\vec{q})$}} (Z') node[midway, above] (bond1) {};
      
     \end{tikzpicture}
     \caption{The star graph $\mathcal{S}_3$ together with the final bond $\vec{\me}_+(\vec{p})$ of a path $\vec{p}$ ending at $\mW$ on the left and the corresponding final bond $\vec{\me}_+(\vec{q})$ of the single path $\vec{q} \in \langle \vec{p} \rangle_+$ on the right.}
     \end{figure}\label{fig:p_+-2}
        \end{exa}
        
It turns out that these sets of extended paths are related to the set $\Pst{\mW}, \Pend{\mW}$ of paths beginning or ending at a certain vertex set $\mW\subset \mV$. Here we use the notation introduced in \eqref{eq:def-set-of-paths} with $m=1$.

\begin{lemma}[Decomposition Lemma]\label{lem:decomposition-lemma}
Let $\mW \subset \mV$. Then 
\begin{align}\label{eq:decomposition-lemma-minus-plus}
\begin{aligned}
\Pst{\mW} &= \bigcup_{\vec{q} \in \Pst{\mW}} \langle \vec{q} \rangle_+ \sqcup \Big( \Pst{\mW} \cap \mathcal{P}_1(\Graph) \Big) = \bigcup_{\vec{q} \in \Pst{\mW}} \langle \vec{q} \rangle_+ \sqcup \left\{ \vec{\me} \in \mB \: : \: \partial^-(\vec{\me}) \in \mW \right\}
\end{aligned}
\end{align}
and
\begin{align}\label{eq:decomposition-lemma-plus-minus}
\begin{aligned}
\Pend{\mW} &= \bigcup_{\vec{q} \in \Pend{\mW}} \langle \vec{q} \rangle_- \sqcup \Big( \Pend{\mW} \cap \mathcal{P}_1(\Graph) \Big) = \bigcup_{\vec{q} \in \Pend{\mW}} \langle \vec{q} \rangle_- \sqcup \left\{ \vec{\me} \in \mB \: : \: \partial^+(\vec{\me}) \in \mW \right\}.
\end{aligned}
\end{align}
\end{lemma}
\begin{proof}
The inclusions \glqq $\supset$\grqq\ are immediate in both cases, thus we only have to show the converse inclusions: to this end, let $\vec{p} \in \Pst{\mW}$ (resp., $\vec{p} \in \Pend{\mW}$) and suppose that $\# \vec{p} \geq 2$ (otherwise, $\vec{p} \in \mathcal{P}_1(\Graph)$ and we are done). Putting $\vec{q} := \vec{p}_+$ (resp., $\vec{q} := \vec{p}_-$) it follows that $\vec{q} \in \Pst{\mW}$ (resp., $\vec{q} \in \Pend{\mW}$) as $\mv_-(\vec{q}) = \mv_-(\vec{p}_+) = \mv_-(\vec{p}) \in \mW$ (resp., $\mv_+(\vec{q}) = \mv_+(\vec{p}_-) = \mv_-(\vec{p}) \in \mW$) due to the fact that $\# \vec{p} \geq 2$. Moreover, by definition of a pre- (resp., post-) extended path and the definition of $\vec{q}$, it follows that $\vec{p} \in \langle \vec{q} \rangle_+$ (resp., $\vec{p} \in \langle \vec{q} \rangle_-$), i.e., $\vec{p}$ belongs to the right-hand side of \eqref{eq:decomposition-lemma-minus-plus} (resp., \eqref{eq:decomposition-lemma-plus-minus}).
\end{proof} 
Note that for every $\vec{p} \in \mathcal{P}(\Graph)$ one has that $\# \langle \vec{p} \rangle_- = \deg(\mv_-(\vec{p}))$ and $\# \langle \vec{p} \rangle_+ = \deg(\mv_+(\vec{p}))$. Moreover, it is an immediate observation that 
\begin{align}\label{eq:reverse-plus-minus}
\langle \vec p \rangle_+ = \{ \vec{q} \in \mathcal{P}(\Graph) \: : \: \cev{q} \in \langle \cev{p} \rangle_- \} \qquad \text{for all $\vec p \in \mathcal{P}(\Graph)$.}
\end{align}
We next want to study the relation between the scattering coefficients of a path and its corresponding extended paths.

\begin{lemma}\label{lem:scattering-coeff-dual}
Let $\vec{p} \in \mathcal{P}(\Graph)$. Then one has that
\begin{align}\label{eq:scattering-coeff-dual}
\sum_{\vec{q} \in \langle \vec{p} \rangle_-} \alpha(\vec{q}) = \begin{cases} \alpha(\vec{p}), & \text{if $\vec{p} \notin \Pst{\mVD}$}, \\ -\alpha(\vec{p}), & \text{if $\vec{p} \in \Pst{\mVD}$,} \end{cases} \quad \text{and} \quad  \sum_{\vec{r} \in \langle \vec{p} \rangle_+} \alpha(\vec{r}) = \begin{cases} \alpha(\vec{p}), & \text{if $\vec{p} \notin \Pend{\mVD}$}, \\ -\alpha(\vec{p}), & \text{if $\vec{p} \in \Pend{\mVD}$.} \end{cases}
\end{align}
\end{lemma}
\begin{proof}
Let $\vec{p} \in \mathcal{P}(\Graph)$ and let $m:= \deg(\mv_-(\vec{p}))$. As $\# \langle \vec{p} \rangle_- = m$, we can write
\begin{align*}
\langle \vec{p} \rangle_- = \{\vec{q}_1,\dots,\vec{q}_{m} \}, \quad \text{and} \quad \left\{ \vec{\me} \in \mB \: : \: \partial^-(\vec{\me}) = \mv_-(\vec{p}) \right\} = \{ \vec{\me}_1,\dots,\vec{\me}_m \},
\end{align*}
i.e., such that $\vec{\me}_-(\vec{q}_j) = \vec{\me}_j$ for all $j=1,\dots,m$ with $\vec{\me}_k \neq \vec{\me}_-(\vec{p})$ for $k=1,\dots,m-1$ and $\vec{\me}_m = \vec{\me}_-(\vec{p})$, in other words, $\vec{q}_1,\dots,\vec{q}_{m-1}$ are the corresponding pre-extended paths to $\vec{p}$ starting with a transfer through $\mv_-(\vec{p})$, and $\vec{q}_m$ is the corresponding pre-extended path to $\vec{p}$ starting with a reflection at $\mv_-(\vec{p})$. If $\vec{p}$ does not belong to $\Pst{\mVD}$, i.e.\ $\mv_-(\vec{p}) \in \mV_\mathrm{N}$, for each $\vec{q}_i$ we have that  $\alpha(\vec{q}_i) = (\frac{2}{m} - \delta_{i,m}) \alpha(\vec{p})$ for all $i=1,\dots,m$ and thus
\begin{align*}
\sum_{\vec{q} \in \langle \vec{p} \rangle_-} \alpha(\vec{q}) = \sum_{i=1}^m \alpha(\vec{q}_i) = \sum_{i=1}^{m-1} \frac{2}{m} \alpha(\vec{p}) + \bigg(\frac{2}{m} - 1 \bigg) \alpha(\vec{p}) = \bigg( \sum_{i=1}^m \frac{2}{m} - 1 \bigg) \alpha(\vec{p}) = \alpha(\vec{p}).
\end{align*}
In the other case, where $\vec{p} \in \Pst{\mVD}$, it follows that $m=1$ and that $\alpha(\vec{q}_1) = (-1) \cdot \alpha(\vec{p})$ and thus
\begin{align*}
\sum_{\vec{q} \in \langle \vec{p} \rangle_-} \alpha(\vec{q}) = \alpha(\vec{q}_1) = -\alpha(\vec{p}):
\end{align*}
this completely shows the left identity of \eqref{eq:scattering-coeff-dual}.

As $\vec{p} \in \Pend{\mVD}$ if and only if $\cev{p} \in \Pst{\mVD}$ for any path $\vec{p} \in \mathcal{P}(\Graph)$, we obtain using \eqref{eq:reverse-plus-minus} as well as the first identity of \eqref{eq:scattering-coeff-dual} that
\begin{align*}
\sum_{\vec{r} \in \langle \vec{p} \rangle_+} \alpha(\vec{r}) =  \sum_{\vec{r} \in \langle \cev{p} \rangle_-} \alpha(\vec{r}) = \begin{cases} \alpha(\cev{p}), & \text{if $\cev{p} \notin \Pst{\mVD}$} \\ -\alpha(\cev{p}), & \text{if $\cev{p} \in \Pst{\mVD}$} \end{cases} = \begin{cases} \alpha(\vec{p}), & \text{if $\vec{p} \notin \Pend{\mVD}$}, \\ -\alpha(\vec{p}), & \text{if $\vec{p} \in \Pend{\mVD}$,} \end{cases}
\end{align*}
completing the proof.
\end{proof}

  
  \section{The heat content formula}\label{sec:combinatorial-formula}

This section is devoted to formulate and prove the following.

\numberwithin{theorem}{section}
\begin{theorem}[Heat content formula]\label{thm:heat-content-formula-bif-mug}
The heat content satisfies
        \begin{align}\label{eq:heat-content-formula-bif-mug}
        \begin{aligned}
\heatcont(\Graph;\mVD) &= \vert \Graph \vert - \frac{2\sqrt{t}}{\sqrt{\pi}} \# \mVD + 8\sqrt{t} \sum_{p} \alpha(p) \, H\bigg( \frac{\ell(p)}{2\sqrt{t}} \bigg) 
\end{aligned}
\end{align}
for all $t > 0$, where the summation in \eqref{eq:heat-content-formula-bif-mug} runs over all undirected paths starting and ending at $\mVD$.
\end{theorem}
The proof of \autoref{thm:heat-content-formula-bif-mug} will require a number of auxiliary results and will be completed at the end of this section.

To begin with, let us recall the \emph{path sum formula} first obtained in \cite{Rot84} for $\mVD=\emptyset$ and, in the general case of possibly nonempty $\mVD$, in \cite[Section~3.4]{KosPotSch07}.

\begin{proposition}[Path sum formula]\label{lem:path-sum-formula}
Let $\Graph$ be as in \autoref{ass:graph} and let $\me,\mf \in \mE$. Then, the heat kernel $\ptGD(\cdot,\cdot)$ associated with the Dirichlet Laplacian $\DeltaGD$ is given by 
\begin{align}
\ptGD(x ,y) 
&= \frac{1}{\sqrt{4\pi t}} \sum_{\vec{p} \in \mathcal{P}_{\geq 0}(x,y)} \alpha(\vec{p}) \e^{-\frac{\ell(\vec p)^2}{4t}} \label{eq:heat-kernel-KPS-2} \\
&= \frac{1}{\sqrt{4\pi t}} \, \delta_{\me,\mf} \, \e^{\frac{-\dist_\Graph( x, y)^2}{4t}} + \frac{1}{\sqrt{4 \pi t}}\sum_{\vec{p} \in \mathcal{P}_{\geq 2}(x,y)} \alpha(\vec{p}) \e^{-\frac{\ell(\vec p)^2}{4t}} \label{eq:heat-kernel-KPS-2b}
\end{align}
for all $t>0$ and all points $x,y\in \Graph$ that lie in the interior of edges $\me,\mf$, respectively, with uniformly convergent right-hand side on
 $\Graph \times \Graph$ for fixed $t>0$. 
\end{proposition}

\begin{rem}
(1) 
Inserting artificial vertices of degree 2 does not change $\DeltaGD$ or its heat kernel. Therefore, the subclasses $\mathcal P_{\ge 0}(x,y)$ and $\mathcal P_{\ge 2}(x,y)$ of directed paths between $x,y$ are well defined.

(2) Let $y \in \mVN$. As observed in \cite[Proposition~2.1]{BorHarJon22}, \eqref{eq:heat-kernel-KPS-2b} also holds whenever $x\in \mVN$. However, \eqref{eq:heat-kernel-KPS-2b} may be wrong whenever $x\in \mVD$, as can be seen in the case of an interval $\Graph \simeq [0,\ell]$ with mixed boundary conditions (Dirichlet at $0$, Neumann at $\ell$): indeed, $\ptGD(x,y)=0$ since $\ptGD(\cdot,y)$ lies in the domain of $\DeltaGD$, while for the right-hand side of \eqref{eq:heat-kernel-KPS-2} there holds
\[
\frac{1}{\sqrt{4\pi t}} \sum_{n=0}^\infty (-1)^{n}\e^{-\frac{(2n+1)^2\ell^2}{4t}} \geq \frac{1}{\sqrt{4\pi t}} \bigg( \e^{-\frac{\ell^2}{4t}} - \e^{-\frac{9\ell^2}{4t}} \bigg) > 0 = \ptGD(x,y).
\]
\end{rem}
The first addend of \eqref{eq:heat-kernel-KPS-2b} does not depend on the scattering coefficients and thus not encode the topology of the graph $\mathcal{G}$. This motivates us to introduce the following notions as both terms appearing in \eqref{eq:heat-kernel-KPS-2b} belong to $L^\infty(\Graph \times \Graph)$ for all $t>0$.

\begin{definition}\label{def:trivial-nontrivial-heat-content}
The \emph{non-topological part} of the heat content is
    \[
    \heatcontm(\mathcal{G}) := \frac{1}{\sqrt{4\pi t}} \sum_{\me,\mf \in \mE} \int_0^{\ell_\me} \int_0^{\ell_\mf} \delta_{\me,\mf} \, \e^{\frac{-\dist_\Graph( x, y)^2}{4t}} \dd y \dd x, \qquad t>0,
    \]
    whereas the \emph{topological part} is
    \[
    \heatcontp(\mathcal{G};\mVD) := \frac{1}{\sqrt{4 \pi t}} \int_\Graph \int_\Graph \sum_{\vec{p} \in \mathcal{P}_{\geq 2}(x,y)} \alpha(\vec{p}) \e^{-\frac{\ell(\vec p)^2}{4t}}  \dd y \dd x, \qquad t >0.
    \]
\end{definition}
By construction we have that $\mathcal{Q}_t(\mathcal{G};\mVD) = \heatcontm(\mathcal{G}) + \heatcontp(\mathcal{G};\mVD)$,  and it will be beneficial for us to study both parts separately: as the name however suggests, the non-topological part is just given by integration on each edge, separately: a first immediate observation is that the non-topological part of the heat content can be written as
\begin{align}\label{eq:non-triv-heat-cont}
    \heatcontm(\mathcal{G}) = \frac{1}{\sqrt{4 \pi t}}\sum_{\me \in \mE} \int_0^{\ell_\me} \int_0^{\ell_\me} \e^{-\frac{\vert x - y \vert^2}{4t}} \dd y \dd x \qquad \text{for all $t>0$.}
    \end{align}
    Let us present a further simple representation based on the \emph{complementary error function} $\mathrm{erfc} \in C^\infty(\mathbb{R})$ defined by 
\[
\mathrm{erfc}(x) := \frac{2}{\sqrt{\pi}}\int_x^\infty \e^{-s^2} \dd s, \qquad x \in \mathbb{R},
\]
cf.\ also \cite{KosPotSch07, AbrSte72} (note that $\frac{\mathrm{d}}{\mathrm{d}x} \erfc(x) = -\frac{2}{\sqrt{\pi}} \e^{-x^2}$ for any $x \in \mathbb{R}$).
    \begin{lemma}\label{lem:trivial-heat-content}
There holds
    \begin{align}\label{eq:non-topological-heat-cont}
    \heatcontm(\mathcal{G}) = \vert \Graph \vert + \frac{2}{\sqrt{\pi}} \sqrt{t} \bigg( \sum_{\me \in \mE} \e^{-\frac{\ell_\me^2}{4t}} - \# \mE \bigg) - \sum_{\me \in \mE} \ell_\me \mathrm{erfc}\bigg( \frac{\ell_\me}{2\sqrt{t}} \bigg) \qquad  \text{for all $t>0$}.
    \end{align}
\end{lemma}
\begin{proof}
First, for any $\me \in \mE$ we have
    \begin{align}\label{eq:triv-heat-content-leibniz}
    \int_0^{\ell_\me} \int_0^{\ell_\me} \e^{-\frac{(x-y)^2}{4t}} \dd y \dd x &= 4t\Big( \e^{-\frac{\ell_\me^2}{4t}} - 1 \Big) + 2 \sqrt{\pi t} \ell_\me \bigg(1  -  \mathrm{erfc} \bigg( \frac{\ell_\me}{2\sqrt{t}} \bigg)\bigg).
    \end{align}
    Indeed, defining $f(\ell,x):= \int_0^\ell \mathrm{e}^{-(x-y)^2} \dd y$ for $\ell \geq 0$ and $x \in [0,\ell]$, we observe according to Leibniz integral rule that
    \begin{align}\label{eq:leibniz-non-topological}
    \begin{aligned}
    \frac{\mathrm{d}}{\mathrm{d}\ell} \int_0^\ell \int_0^\ell \mathrm{e}^{-(x-y)^2} \dd y \dd x &= f(\ell,\ell) + \int_0^\ell \frac{\mathrm{d}}{\mathrm{d}\ell} f(\ell,x) \mathrm{d}x 
    \\&= \int_0^\ell \mathrm{e}^{-(\ell - y)^2} \dd y + \int_0^\ell \mathrm{e}^{-(x-\ell)^2} \dd x = \sqrt{\pi} (1-\erfc (\ell))
    \end{aligned}
    \end{align}
    and integrating both sides of \eqref{eq:leibniz-non-topological} yields
    \begin{align*}
    \begin{aligned}
        \int_0^\ell \int_0^\ell \mathrm{e}^{-(x-y)^2} \dd y \dd x &= \sqrt{\pi} \int_0^\ell (1-\erfc (s)) \dd s 
        \\&= \sqrt{\pi}\ell(1-\erfc (\ell)) + \mathrm{e}^{-\ell^2} - 1,
        \end{aligned} \qquad \text{for $\ell \in \mathbb{R}_+$,}
    \end{align*}
    which eventually yields \eqref{eq:triv-heat-content-leibniz} through substitution. This shows the claim.
\end{proof}
The expression in \eqref{eq:non-topological-heat-cont} can also be rewritten as
\begin{align}\label{eq:non-topological-heat-cont-2}
\begin{aligned}
\heatcontm(\Graph) &= \vert \Graph \vert - \frac{2\sqrt{t}}{\sqrt{\pi}}\# \mE + 2\sqrt{t} \sum_{\me \in \mE} \Bigg( \frac{1}{\sqrt{\pi}} \e^{-\frac{\ell_\me^2}{4t}} - \frac{\ell_\me}{2\sqrt{t}} \erfc \bigg(\frac{\ell_\me}{2\sqrt{t}} \bigg) \Bigg) \\&= \vert \Graph \vert - \frac{2\sqrt{t}}{\sqrt{\pi}}\# \mE + 2\sqrt{t} \sum_{\me \in \mE} H \bigg( \frac{\ell_\me}{2\sqrt{t}} \bigg) \
\end{aligned}, \qquad t > 0,
\end{align}
where 
\begin{align}\label{eq:anti-primitive-erfc}
        H(x) := \frac{1}{\sqrt{\pi}} - \int_0^x \mathrm{erfc}(s)\dd s =  \frac{1}{\sqrt{\pi}}\e^{-x^2} - x\mathrm{erfc}(x), \qquad x \in \mathbb{R}.
        \end{align}

Let us summarize a few basic properties of $H$, which we plot in \autoref{fig:func-H}.

\begin{lemma}\label{lem:properties-H}
The function $H$ belongs to $C^\infty(\mathbb{R})$ and is strictly convex.
Furthermore, $H(\mathbb{R}) \subset (0,\infty)$.
    \end{lemma}
    
\begin{proof}
A direct computation shows that
\[
H'(x) = -\erfc(x),\quad H''(x) = \frac{2\e^{-x^2}}{\sqrt{\pi}}\qquad\hbox{ for every }x \in \mathbb{R},
\]
hence $H\in C^\infty(\mathbb{R})$ and $H$ is strictly decreasing and convex since $H' < 0$ and $H''>0$. Moreover, 
\begin{align}\label{eq:estimate-erfc}
        \mathrm{erfc}(x) = \frac{2}{\sqrt{\pi}} \int_x^\infty \e^{-t^2} \dd t < \frac{1}{\sqrt{\pi}x} \int_x^\infty 2t\e^{-t^2} \dd t = \frac{\e^{-x^2}}{\sqrt{\pi}x}\qquad\hbox{for all }x \in \mathbb{R}_+,
    \end{align}
    whence $H(\mathbb{R}_+) \subset (0,\infty)$. Moreover, it is an immediate observation that $H(\mathbb{R} \setminus \mathbb{R}_+) \subset (0,\infty)$.
\end{proof}

\begin{figure}[h]
\includegraphics[width=5cm]{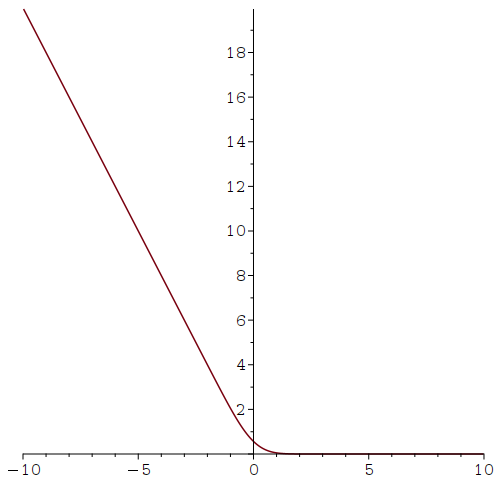} 
\caption{The profile of the function $H$ in~\eqref{eq:anti-primitive-erfc}}\label{fig:func-H}
\end{figure}

Showing that the topological part of the heat content, too, can be represented in terms of $H$ will be the next major step on the way to prove \autoref{thm:heat-content-formula-bif-mug}, the main result of this paper.

\begin{lemma}\label{lem:representation-topological-heat-cont}
    For $\Graph$ and $\heatcontp(\Graph;\mVD)$ as in \autoref{def:trivial-nontrivial-heat-content} one has
    \begin{align}
        \heatcontp(\Graph;\mVD) = \frac{1}{\sqrt{4\pi t}} \sum_{\vec{p} \in \mathcal{P}_{\geq 2}(\Graph)} \alpha(\vec{p}) \int_0^{\ell_{\me_-(\vec{p})}} \int_0^{\ell_{\me_+(\vec{p})}} \e^{-\frac{(x+\ell(\vec{p}_\pm) + y)^2}{4t}} \dd y \dd x \qquad \text{for all $t>0$.}
    \end{align}
\end{lemma}
Here $\mathcal{P}_{\geq 2}(\Graph)$ is defined as in~\eqref{eq:def-set-of-paths}.

\begin{proof}
Let $x,y \in \Graph$. Given a path $\vec{p} =(x,\vec{\me}_1,\dots,\vec{\me}_n,y) \in \mathcal{P}_{\geq 2}(x,y)$ with $n \geq 2$, we can find bonds $\vec{\me}, \vec{\mf} \in \mB$ such that $x$ lies on the corresponding edge $\me \in \mE$ and $y$ lies on the corresponding edge $\mf \in \mE$ and such that $\vec{\me}$ and $\vec{\mf}$ are directed into the same direction as $\vec{\me}_1$ and $\vec{\me}_n$ in the sense that $\partial^+(\vec \me) = \partial^+(\vec \me_1)$ and $\partial^-(\vec \mf) = \partial^-(\vec \me_n)$ (see \autoref{figure:from-paths-between-points-to-paths-between-edges}).
    \begin{figure}[h]
        \begin{tikzpicture}[scale=0.9]
      \tikzset{enclosed/.style={draw, circle, inner sep=0pt, minimum size=.08cm, fill=black}, every loop/.style={}}

      \node[enclosed, label={below: $\partial^-(\vec{\me}_n) \qquad$}] (C) at (4,5) {};
      \node[enclosed] (A) at (-1,5) {};;
      \node[enclosed] (B) at (1,5) {};;
      \node[enclosed, label={above: \qquad $y \in \mf$}] (D') at (6,5) {};
      \node[enclosed] (D) at (8,5) {};
      \node[enclosed,white] (D1) at (8.5,4.5) {};
      \node[enclosed,white] (D2) at (8.5,5) {};
      \node[enclosed,white] (D3) at (8.5,5.5) {};
      \node[enclosed, label={below: \qquad $\partial^+(\vec{\me}_1)$}] (F) at (-4,5) {};
      \node[enclosed] (G) at (-8,5) {};
      \node[enclosed, label={above: $\me \ni x \qquad$}] (G') at (-6,5) {};
      \node[enclosed, white] (G1) at (-8.5,4.5) {};
      \node[enclosed, white] (G2) at (-8.5,5) {};
      \node[enclosed, white] (G3) at (-8.5,5.5) {};
      
      \draw[dotted] (A) edge node[below] {} (B) node[midway, above] (edge1) {};
      \draw[->, thick, blue] (B) edge[bend left = 20] node[above] {\color{black} $\vec{\me}_{n-1}$} (C) node[midway, above] (edge2) {};
       \draw[-] (B) edge node[above] {} (C) node[midway, above] (edge2) {};
      \draw[<-, thick, blue] (A) edge[bend right= 20] node[above] {\color{black} $\vec{\me}_2$} (F) node[midway, above] (edge4) {};
      \draw[-] (A) edge node[above] {} (F) node[midway, above] (edge4) {};
      \draw[-] (G') edge node[above] {} (G) node[midway, above] (edge4) {};
      \draw[->, thick, blue] (G') edge[bend left] node[above] {\color{black} $\vec{\me_1}$} (F) node[midway, above] (edge4) {};
      \draw[->, thick, blue] (G) edge[bend right] node[below] {\color{black} $\vec{\me}$} (F) node[midway, above] (edge4) {};
      \draw[-] (G') edge node[above] {} (F) node[midway, above] (edge4) {};
      \draw[dotted] (G) edge node[above] {} (G1) node[midway, above] (edge4) {};
      \draw[dotted] (G) edge node[above] {} (G2) node[midway, above] (edge4) {};
      \draw[dotted] (G) edge node[above] {} (G3) node[midway, above] (edge4) {};
      \draw[dotted] (D) edge node[above] {} (D1) node[midway, above] (edge4) {};
      \draw[dotted] (D) edge node[above] {} (D2) node[midway, above] (edge4) {};
      \draw[dotted] (D) edge node[above] {} (D3) node[midway, above] (edge4) {};
      \draw[->, blue, thick] (C) edge[bend left] node[above] {\color{black} $\vec{\me}_n$} (D') node[midway, above] (edge4) {};
      \draw[-] (C) edge node[above] {} (D') node[midway, above] (edge4) {};
      \draw[->, thick, blue] (C) edge[bend right] node[below] {\color{black} $\vec{\mf}$} (D) node[midway, above] (edge4) {};
      \draw[-] (D') edge node[above] {} (D) node[midway, above] (edge4) {};
     \end{tikzpicture}
     \vspace{-3cm}
     \caption{The path $\vec{p} = (x,\vec{\me}_1,\dots,\vec{\me}_n,y)$ and edges $\me \ni x$ and $\mf \ni y$ such that $\partial^+(\vec{\me}) = \partial^+(\vec{\me}_1)$ and $\partial^-(\vec{\mf}) = \partial^-(\vec{\me}_n)$.}\label{figure:from-paths-between-points-to-paths-between-edges}
    \end{figure}
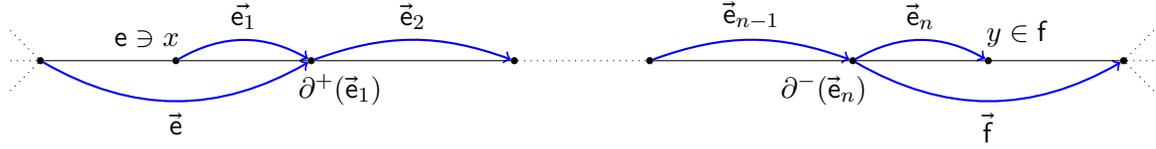

Thus, we have a one-to-one correspondence between
 $\mathcal{P}_{\geq 2}(x,y)$ and
\begin{equation}\label{eq:p2-bije}
\bigcup_{\vec{\me} \in \mathbf{b}(\me), \vec{\mf} \in \mathbf{b}(\mf)} \left\{ \vec{p} \in \mathcal{P}_{\geq 2}(\Graph) \: : \: \vec{\me}_-(\vec{p}) = \vec{\me}, \: \vec{\me}_+(\vec{p}) = \vec{\mf} \right\},
\end{equation}
where every path $\vec{p} =(x,\vec{\me}_1,\dots,\vec{\me}_n,y) \in \mathcal{P}_{\geq 2}(x,y)$ corresponds to a path 
\[
\vec{p}({\vec\me,\vec\mf}) := \big(\partial^-(\vec{\me}), \vec{\me},\vec{\me}_1,\dots,\vec{\me}_n,\vec{f}, \partial^-(\vec{\mf}) \big) \in \mathcal{P}(\Graph)
\]
with $\vec{\me} \in \mathbf{b}(\me)$ and $\vec{\mf} \in \mathbf{b}(\mf)$ such that $\partial^+(\vec \me) = \partial^+( \vec{\me}_1)$ and $\partial^-(\vec \mf) = \partial^-(\vec \me_n)$, as well as
\begin{align}\label{eq:scattering-coefficient-one-to-one-corresp}
\alpha(\vec p) = \alpha\big(\vec{p}({\vec\me,\vec\mf}) \big),
\end{align}
and
\begin{align}\label{eq:length-one-to-one-corresp}
\begin{aligned}
\ell(\vec p) &= \dist (x,\partial^+({\vec \me}_1)) + \ell\big(\vec{p}({\vec\me,\vec\mf})_\pm\big) + \dist(\partial^-({\vec{\me}}_{n}),y)\\
&= \dist(x,\partial^+({\vec \me})) + \ell(\vec{p}_\pm ) + \dist(\partial^-({\vec \mf}),y),
\end{aligned}
\end{align}
as $\vec{p}({\vec\me,\vec\mf})_\pm = \vec{p}_\pm$ by construction. Moreover, one has
\begin{align}
    \dist(x,\partial^+(\vec{\me})) = \begin{cases}
        \ell_\me-x, & \text{if $\overset{\me}{\rightsquigarrow} \partial^+(\vec{\me}_1) $,} \\ x , & \text{if $\partial^+(\vec{\me}_1) \overset{\me}{\rightsquigarrow}$,} 
    \end{cases} \quad \text{and} \quad \dist(\partial^-(\vec{\mf}),y) = \begin{cases}
        \ell_\mf - y, & \text{if $\overset{\mf}{\rightsquigarrow} \partial^-(\vec{\me}_n) $,} \\ y , & \text{if $\partial^-(\vec{\me}_n) \overset{\mf}{\rightsquigarrow}$.} 
    \end{cases}
\end{align}
By appropriate substitution, one sees that
\begin{align*}
    &\int_0^{\alpha} \int_0^{\beta} \e^{-\frac{(x + \ell + y)^2}{4t}} \dd y \dd x = \int_0^{\alpha} \int_0^{\beta} \e^{-\frac{(\alpha - x + \ell + y)^2}{4t}} \dd y \dd x \\& \quad\qquad= \int_0^{\alpha} \int_0^{\beta} \e^{-\frac{(x + \ell + \beta - y)^2}{4t}} \dd y \dd x = \int_0^{\alpha} \int_0^{\beta} \e^{-\frac{(\alpha - x + \ell + \beta - y)^2}{4t}} \dd y \dd x
\end{align*}
for all $\alpha,\beta,\ell \in \mathbb{R}_+$: therefore,
\begin{align}\label{eq:easier-integral}
    \int_0^{\ell_\me} \int_0^{\ell_\mf} \e^{-\frac{(\dist(x,\partial^+(\vec{\me})) + \ell(\vec{p}_\pm) + \dist(\partial^-(\vec{\mf}),y))^2}{4t}} \dd y \dd x = \int_0^{\ell_\me} \int_0^{\ell_\mf} \e^{-\frac{(x + \ell(\vec{p}_\pm) + y)^2}{4t}} \dd y \dd x,
\end{align}
for all $\vec \me,\vec \mf \in \mB$, all $\vec p \in \mathcal{P}_{\geq 2}(\partial^-(\vec \me), \partial^+(\vec \mf))$, and all $t > 0$. Combining now \eqref{eq:scattering-coefficient-one-to-one-corresp}, \eqref{eq:length-one-to-one-corresp} and \eqref{eq:easier-integral}, this eventually leads to
    \begin{align*}
        \heatcontp(\Graph;\mVD) &= \frac{1}{\sqrt{4\pi t}} \sum_{\me,\mf \in \mE} \int_0^{\ell_\me} \int_0^{\ell_\mf} \sum_{\vec{p} \in \mathcal{P}_{\geq 2}(x,y)} \alpha(\vec p) \e^{-\frac{\ell(\vec{p})^2}{4t}} \dd y \dd x \\&= \frac{1}{\sqrt{4\pi t}} \sum_{\me,\mf \in \mE} \int_0^{\ell_\me} \int_0^{\ell_\mf} \sum_{\vec{\me} \in \mathbf{b}(\me), \, \vec{\mf} \in \mathbf{b}(\mf)}\sum_{\stackrel{\vec{p} \in \mathcal{P}_{\geq 2}(\Graph)}{\vec{\me}_-(\vec{p}) = \vec{\me}, \, \vec{\me}_+(\vec{p}) = \vec{\mf}}} \alpha(\vec p) \e^{-\frac{(\dist(x,\partial^+(\vec{\me})) + \ell(\vec{p}_\pm)+\dist (\partial^-(\vec{\mf}),y))^2}{4t}} \dd y \dd x \\&=  \frac{1}{\sqrt{4\pi t}} \sum_{\me,\mf \in \mE} \sum_{\vec{\me} \in \mathbf{b}(\me), \, \vec{\mf} \in \mathbf{b}(\mf)}\sum_{\stackrel{\vec{p} \in \mathcal{P}_{\geq 2}(\Graph)}{\vec{\me}_-(\vec{p}) = \vec{\me}, \, \vec{\me}_+(\vec{p}) = \vec{\mf}}} \alpha(\vec p) \int_0^{\ell_\me} \int_0^{\ell_\mf} \e^{-\frac{(x + \ell(\vec{p}_\pm)+y)^2}{4t}} \dd y \dd x 
        \\&=  \frac{1}{\sqrt{4\pi t}} \sum_{\vec \me, \vec \mf \in \mB} \sum_{\stackrel{\vec{p} \in \mathcal{P}_{\geq 2}(\Graph)}{\vec{\me}_-(\vec{p}) = \vec{\me}, \, \vec{\me}_+(\vec{p}) = \vec{\mf}}} \alpha(\vec p) \int_0^{\ell_\me} \int_0^{\ell_\mf} \e^{-\frac{(x + \ell(\vec{p}_\pm)+y)^2}{4t}} \dd y \dd x
         \\&=  \frac{1}{\sqrt{4\pi t}} \sum_{\vec{p} \in \mathcal{P}_{\geq 2}(\Graph)} \alpha(\vec p) \int_0^{\ell_{\me_-(\vec{p})}} \int_0^{\ell_{\me_+(\vec{p})}} \e^{-\frac{(x + \ell(\vec{p}_\pm)+y)^2}{4t}} \dd y \dd x:
    \end{align*}
this completes the proof.
\end{proof}

Using Leibniz' integral rule to compute the terms of the form $ \int_0^\alpha \int_0^\beta \e^{-\frac{(x+\ell+y)^2}{4t}} \dd y \dd x $ for $\alpha,\beta,\ell \in \mathbb{R}_+$, 
in a similar way to the proof of \autoref{lem:trivial-heat-content}, we reach at the following first combinatorial expression for the heat content which also involves the function $H$.

\begin{proposition}\label{thm:first-roth-like-expansion}
Under \autoref{ass:graph} the following assertions hold.
    
(i) For $t>0$ each of the series
\[
\sum_{\vec{p} \in \mathcal{P}_{\geq 2}(\mathcal{G})} \alpha(\vec{p}) 
 H\bigg( \frac{\ell(\vec{p})}{2\sqrt{t}} \bigg) , \:\:
\sum_{\vec{p} \in \mathcal{P}_{\geq 2}(\mathcal{G})} \alpha(\vec{p}) 
H\bigg( \frac{\ell(\vec{p}_-)}{2\sqrt{t}} \bigg),\:\:
\sum_{\vec{p} \in \mathcal{P}_{\geq 2}(\mathcal{G})} \alpha(\vec{p})
 H\bigg( \frac{\ell(\vec{p}_+)}{2\sqrt{t}} \bigg) ,\:\:
  \sum_{\vec{p} \in \mathcal{P}_{\geq 2}(\mathcal{G})} \alpha(\vec{p})
H\bigg( \frac{\ell(\vec{p}_\pm)}{2\sqrt{t}} \bigg)
\]
converges.
    
(ii) The topological part of the heat content satisfies
\begin{align}\label{eq:first-roth-like-expansion}
        \heatcontp(\mathcal{G};\mVD) = \sqrt{t} \sum_{\vec{p} \in \mathcal{P}_{\geq 2}(\mathcal{G})} \alpha(\vec{p}) \left( H\bigg( \frac{\ell(\vec{p})}{2\sqrt{t}} \bigg) - H\bigg( \frac{\ell(\vec{p}_-)}{2\sqrt{t}} \bigg) - H\bigg( \frac{\ell(\vec{p}_+)}{2\sqrt{t}} \bigg) + H\bigg( \frac{\ell(\vec{p}_\pm)}{2\sqrt{t}} \bigg) \right)
    \end{align}
for all $t>0$.
\end{proposition}
In particular, \autoref{thm:first-roth-like-expansion} shows that the heat content itself satisfies
    \begin{align}\label{eq:heat-cont-erf} 
    \begin{aligned}
        \mathcal{Q}_t(\mathcal{G};\mVD) &= \vert \Graph \vert + \frac{2\sqrt{t}}{\sqrt{\pi}} \# \mE - 2\sqrt{t} \sum_{\me \in \mE} H\bigg( \frac{\ell_\me}{2\sqrt{t}} \bigg) \\& \qquad +\sqrt{t} \sum_{\vec{p} \in \mathcal{P}(\mathcal{G})} \alpha(\vec{p}) \left( H\bigg( \frac{\ell(\vec{p})}{2\sqrt{t}} \bigg) - H\bigg( \frac{\ell(\vec{p}_-)}{2\sqrt{t}} \bigg) - H\bigg( \frac{\ell(\vec{p}_+)}{2\sqrt{t}} \bigg) + H\bigg( \frac{\ell(\vec{p}_\pm)}{2\sqrt{t}} \bigg) \right)
\\&= \vert \Graph \vert + \frac{2\sqrt{t}}{\sqrt{\pi}} \# \mE - 2\sqrt{t} \sum_{\me \in \mE} H\bigg( \frac{\ell_\me}{2\sqrt{t}} \bigg) \\& \qquad +\sqrt{t} \sum_{\vec{p} \in \mathcal{P}(\mathcal{G})} \alpha(\vec{p}) \left( H\bigg( \frac{\ell(\vec p)}{2\sqrt{t}} \bigg) - 2 H\bigg( \frac{\ell(\vec{p}_-)}{2\sqrt{t}} \bigg) + H\bigg( \frac{\ell(\vec{p}_\pm)}{2\sqrt{t}} \bigg) \right) 
    \end{aligned}
    \end{align}
   for every $t > 0$.
   
\begin{proof}[Proof of \autoref{thm:first-roth-like-expansion}]
(i) follows immediately from \eqref{eq:estimate-erfc}. 
    
(ii) Since for any path $\vec{p} \in \mathcal{P}(\Graph)$ one has $\ell(\vec{p}_-) = \ell(\cev{p}_+)$ and $\alpha(\vec{p}) = \alpha(\cev{p})$, and the map $\vec{p} \mapsto \cev{p}$ is a bijection on $\mathcal{P}(\Graph)$, it follows that
    \[
    \sum_{\vec{p} \in \mathcal{P}(\Graph)} \alpha(\vec{p}) H\bigg( \frac{\ell(\vec{p}_-)}{2\sqrt{t}} \bigg) = \sum_{\vec{p} \in \mathcal{P}(\Graph)} \alpha(\cev{p}) H\bigg( \frac{\ell(\cev{p}_+)}{2\sqrt{t}} \bigg) = \sum_{\vec{p} \in \mathcal{P}(\Graph)} \alpha(\vec{p}) H\bigg( \frac{\ell(\vec{p}_+)}{2\sqrt{t}} \bigg),
    \]
    thus, implying the second identity in \eqref{eq:heat-cont-erf}. 
    
 Using Leibniz integral rule once again, for $\alpha, \beta, \ell \in \mathbb{R}_+$, we deduce that
    \begin{align}\label{eq:roth-second-auxiliary-eq}
    \begin{aligned}
        &\int_0^\alpha \int_0^\beta \e^{-\frac{(x+\ell+y)^2}{4t}} \dd y \dd x \\& \quad\qquad = \sqrt{4\pi} t\Bigg( H\bigg(\frac{\ell}{2\sqrt{t}} \bigg) - H\bigg(\frac{\alpha + \ell}{2\sqrt{t}} \bigg) - H\bigg(\frac{\beta + \ell}{2\sqrt{t}} \bigg) + H\bigg(\frac{\alpha+\ell+\beta}{2\sqrt{t}} \bigg) \Bigg)
        \end{aligned} \quad \text{for all $t>0$.}
    \end{align}
     Indeed, differentiating the left-hand side of \eqref{eq:roth-second-auxiliary-eq} with respect to $\beta$, we observe at first that
    \begin{align*}
        \frac{\mathrm{d}}{\mathrm{d}\beta} \int_0^\alpha \int_0^\beta \e^{-\frac{(x+\ell+y)^2}{4t}} \dd y \dd x &= \int_0^\alpha \e^{-\frac{(x+\ell+\beta)^2}{4t}} \dd x = 2 \sqrt{t} \int_{\frac{\ell+\beta}{2 \sqrt{t}}}^{\frac{\alpha + \ell + \beta}{2\sqrt{t}}} \e^{-x^2} \dd x \\&= \sqrt{t \pi} \bigg( \erfc  \bigg(\frac{\ell+\beta}{2\sqrt{t}}\bigg) - \erfc  \bigg( \frac{\alpha + \ell + \beta}{2\sqrt{t}} \bigg) \bigg)
    \end{align*}
    which yields, since $H'(x) = -\erfc (x)$ for all $x \in \mathbb{R}$,
    \begin{align*}
        \int_0^\alpha \int_0^\beta \mathrm{e}^{-\frac{(x+\ell+y)^2}{4t}} \dd y \dd x &= \sqrt{t \pi}\bigg(\int_0^\beta \erfc \bigg(\frac{\ell + s}{2\sqrt{t}} \bigg) \dd s - \int_0^\beta \erfc \bigg(\frac{\alpha + \ell + s}{2\sqrt{t}} \bigg) \dd s \bigg) \\&= \sqrt{4\pi}t \Bigg(  \int_{\frac{\ell}{2\sqrt{t}}}^{\frac{\ell + \beta}{2\sqrt{t}}} \erfc (s) \dd s - \int_{\frac{\alpha + \ell}{2\sqrt{t}}}^{\frac{\alpha + \ell + \beta}{2\sqrt{t}}} \erfc (s) \dd s \Bigg) \\&= \sqrt{4\pi}t \Bigg( H\bigg(\frac{\ell}{2\sqrt{t}} \bigg) - H\bigg(\frac{\ell+\beta}{2\sqrt{t}} \bigg) + H\bigg(\frac{\alpha + \ell + \beta}{2\sqrt{t}} \bigg) - H\bigg(\frac{\alpha + \ell}{2\sqrt{t}} \bigg) \Bigg)
    \end{align*}
    implying \eqref{eq:roth-second-auxiliary-eq} and, according to \autoref{lem:representation-topological-heat-cont}, the expansion in \eqref{eq:first-roth-like-expansion}.
    
Now as $\ell(\vec p_-) = \ell (\vec p_+) = 0$, $\ell (\vec p_\pm) = \ell(\cev p) = \ell(\vec p)$ and $\alpha(\vec{p}) = 1$ whenever $\# \vec{p} = 1$, it follows that
    \begin{align*}
   &\sum_{\vec p \in \mathcal{P}_{1}(\mathcal{G})} \alpha(\vec{p}) \left( H\bigg( \frac{\ell (\vec p)}{2\sqrt{t}} \bigg) - H\bigg( \frac{\ell (\vec p_-)}{2\sqrt{t}} \bigg) - H\bigg( \frac{\ell (\vec p_+)}{2\sqrt{t}} \bigg) + H\bigg( \frac{\ell (\vec p_\pm)}{2\sqrt{t}} \bigg) \right) \\& \qquad\quad = 2\sum_{\vec p \in \mathcal{P}_{1}(\mathcal{G})}  \left( H\bigg( \frac{\ell(\vec p)}{2\sqrt{t}} \bigg) - H( 0 ) \right) = -\frac{4\sqrt{t}}{\sqrt{\pi}} \# \mE + 4\sqrt{t} \sum_{\me \in \mE} H \bigg(\frac{\ell_\me}{2\sqrt{t}} \bigg)
    \end{align*}
    for all $t>0$, as $\# \mathcal{P}_1(\Graph) = 2\# \mE$, yielding the first identity in \eqref{eq:heat-cont-erf} due to \eqref{eq:first-roth-like-expansion}.
\end{proof}

All addends appearing in \eqref{eq:heat-cont-erf} are convergent for fixed $t>0$, for $\Graph$ and $\mVD \subset \mV$, and we can thus introduce quantities
        \begin{align}\label{eq:addend-1-heat-content-formula}
        \widetilde{\mathcal{Q}_{t}}(\Graph;\mVD) := \sum_{\vec{p} \in \mathcal{P}(\Graph)} \alpha(\vec{p}) \bigg( H\bigg(\frac{\ell(\vec p)}{2\sqrt{t}} \bigg) - H\bigg(\frac{\ell(\vec{p}_-)}{2\sqrt{t}} \bigg) \bigg), 
        \end{align}
     and
     \begin{align}\label{eq:addend-2-heat-content-formula}
     \widetilde{\widetilde{\mathcal{Q}_{t}}}(\Graph;\mVD) := \sum_{\vec p \in \mathcal{P}(\Graph)} \alpha(\vec p) \bigg( H\bigg(\frac{\ell(\vec p_+)}{2\sqrt{t}} \bigg) - H\bigg(\frac{\ell(\vec  p_\pm)}{2\sqrt{t}} \bigg) \bigg),
     \end{align}
for $t>0$ and study them separately, using  \autoref{lem:decomposition-lemma} and \autoref{lem:scattering-coeff-dual}. 

       \begin{lemma}\label{lem:q1tq2t}
       The expressions in \eqref{eq:addend-1-heat-content-formula} and \eqref{eq:addend-2-heat-content-formula} can be rewritten as
       \begin{align}\label{eq:modified-addend-1-heat-content-formula}
       \widetilde{\mathcal{Q}_{t}}(\Graph;\mVD) = 2  \sum_{\vec{p} \in \Pst{\mVD}} \alpha(\vec{p}) H\bigg( \frac{\ell(\vec{p})}{2\sqrt{t}} \bigg) - \frac{2\# \mE}{\sqrt{\pi}}
       \end{align}
       and
       \begin{align}\label{eq:modified-addend-2-heat-content-formula}
       \widetilde{\widetilde{\mathcal{Q}_{t}}}(\Graph;\mVD) =  2 \sum_{\vec p \in \Pst{\mVD}} \alpha(\vec p) H\bigg( \frac{\ell(\vec{p}_+)}{2\sqrt{t}} \bigg) - 2\sum_{\me \in \mE} H\bigg( \frac{\ell_\me}{2\sqrt{t}} \bigg),
       \end{align}
       respectively, for all $t>0$.
       \end{lemma} 
       Note that, again due to the symmetry $\vec{p} \mapsto \cev{p}$, one can replace the set $\Pst{\mVD}$ in \eqref{eq:modified-addend-1-heat-content-formula} and \eqref{eq:modified-addend-2-heat-content-formula} by $\Pend{\mVD}$ and vice versa.
       \begin{proof}
       We first determine the identity in \eqref{eq:modified-addend-1-heat-content-formula}: first, as $\mathcal{P}(\Graph) = \Pend{\mVD} \cup \Pend{\mVN}$, we can write
       \begin{align}\label{eq:decomposition-q-1-t}
       \begin{aligned}
       \widetilde{\mathcal{Q}_{t}}(\Graph;\mVD) &= \sum_{\vec{p} \in \Pend{\mVD}} \alpha(\vec{p}) \bigg( H\bigg(\frac{\ell(\vec p)}{2\sqrt{t}} \bigg) - H\bigg(\frac{\ell(\vec  p_-)}{2\sqrt{t}} \bigg) \bigg) \\& \quad\qquad + \sum_{\vec{p} \in \Pend{\mVN}} \alpha(\vec{p}) \bigg( H\bigg(\frac{\ell(\vec p)}{2\sqrt{t}} \bigg) - H\bigg(\frac{\ell(\vec  p_-)}{2\sqrt{t}} \bigg) \bigg) \\& =: \widetilde{\mathcal{Q}_{t}^\mVD}(\Graph;\mVD) + \widetilde{\mathcal{Q}_{t}^\mVN}(\Graph;\mVD),
       \end{aligned} \qquad \text{for all $t>0$.}
       \end{align}
       According to 
\autoref{lem:decomposition-lemma}, for $\mW \in \{ \mVD, \mVN \}$ one has that
\begin{align}\label{eq:application-decomposition-lemma}
\begin{aligned}
\sum_{\vec{p} \in \Pend{\mW}} \alpha(\vec{p})H\bigg( \frac{\ell(\vec{p}_-)}{2\sqrt{t}} \bigg) &= \sum_{\vec{p} \in \Pend{\mW}} \sum_{\vec{q} \in \langle \vec{p} \rangle_-} \alpha(\vec{q}) H \bigg( \frac{\ell(\vec{q}_-)}{2\sqrt{t}} \bigg) + \sum_{\vec{p} \in \Pend{\mW} \cap \mathcal{P}_1(\Graph)} \alpha(\vec{p}) H \bigg( \frac{\ell(\vec{p}_-)}{2\sqrt{t}} \bigg) \\&=  \sum_{\vec{p} \in \Pend{\mW}} \sum_{\vec{q} \in \langle \vec{p} \rangle_-} \alpha(\vec{q}) H \bigg( \frac{\ell(\vec{p})}{2\sqrt{t}} \bigg) + \sum_{\vec{p} \in \Pend{\mW} \cap \mathcal{P}_1(\Graph)} \frac{1}{\sqrt{\pi}}
\end{aligned}
\end{align}       
  for all $t>0$, because $\vec{q}_- = \vec{p}$ for all $\vec{q} \in \langle \vec{p} \rangle_-$ by definition, and $\ell(\vec{p}_-)=0$ as well as $\alpha(\vec{p}) =1$ for all $\vec{p} \in \mathcal{P}_1(\Graph)$. Thus, due to the convergence in \autoref{thm:first-roth-like-expansion} for any $t>0$, we can decompose $\widetilde{\mathcal{Q}_{t}^\mW}(\Graph;\mVD)$, for $\mW \in \{ \mVD,\mVN \}$ as 
        \begin{align}\label{eq:help-q-1-t}
        \begin{aligned}
        \widetilde{\mathcal{Q}_{t}^\mW}(\Graph;\mVD) &= \sum_{\vec{p} \in \Pend{\mW}} \alpha(\vec{p}) H\bigg(\frac{\ell(\vec p)}{2\sqrt{t}} \bigg) -  \sum_{\vec{p} \in \Pend{\mW}} \alpha(\vec{p}) H\bigg(\frac{\ell(\vec  p_-)}{2\sqrt{t}} \bigg) \\&= \sum_{\vec{p} \in \Pend{\mW}} \Bigg( \alpha(\vec{p}) - \sum_{\vec{q} \in \langle \vec{p} \rangle_-} \alpha(\vec{q}) \Bigg) H\bigg(\frac{\ell(\vec{p})}{2\sqrt{t}} \bigg) - \sum_{\vec{p} \in \Pend{\mW} \cap \mathcal{P}_1(\Graph)} \frac{1}{\sqrt{\pi}},
        \end{aligned}
\end{align}               
     for every $t>0$. Now, by \autoref{lem:scattering-coeff-dual}, we have that 
     \begin{align}\label{eq:scattering-dual-terms}
     \alpha(\vec{p}) - \sum_{\vec{q} \in \langle \vec{p} \rangle_-} \alpha(\vec{q}) = \begin{cases} 2\alpha(\vec{p}), & \text{if $\mW = \mVD$}, \\ 0, & \text{if $\mW = \mVN$,}
     \end{cases} \qquad \text{for all $\vec{p} \in \Pst{\mW}$},
\end{align}       
therefore, we obtain for $\widetilde{\mathcal{Q}_{t}^\mVD}(\Graph;\mVD)$ and $\widetilde{\mathcal{Q}_{t}^\mVN}(\Graph;\mVD)$ that
\begin{align*}
\widetilde{\mathcal{Q}_{t}^\mVD}(\Graph;\mVD) &= 2\sum_{ \vec{p} \in \Pend{\mVD} \cap \Pst{\mVD}} \alpha(\vec{p})H\bigg(\frac{\ell(\vec p)}{2\sqrt{t}} \bigg) - \sum_{\vec{p} \in \Pend{\mVD} \cap \mathcal{P}_1(\Graph)} \frac{1}{\sqrt{\pi}} \\&= 2\sum_{\vec{p} \in \mathcal{P}_{\mVD}(\Graph)} \alpha(\vec{p})H\bigg(\frac{\ell(\vec p)}{2\sqrt{t}} \bigg) - \sum_{\vec{p} \in \Pend{\mVD} \cap \mathcal{P}_1(\Graph)} \frac{1}{\sqrt{\pi}}
\end{align*}
and
\begin{align*}
\widetilde{\mathcal{Q}_{t}^\mVN}(\Graph;\mVD) &= 2\sum_{\vec{p} \in \Pend{\mVN} \cap \Pst{\mVD}} \alpha(\vec{p})H\bigg(\frac{\ell(\vec p)}{2\sqrt{t}} \bigg) - \sum_{\vec{p} \in \Pend{\mVN} \cap \mathcal{P}_1(\Graph)} \frac{1}{\sqrt{\pi}} \\&= 2\sum_{\vec{p} \in \mathcal{P}_{\mVD,\mVN}(\Graph)} \alpha(\vec{p})H\bigg(\frac{\ell(\vec p)}{2\sqrt{t}} \bigg) - \sum_{\vec{p} \in \Pend{\mVN} \cap \mathcal{P}_1(\Graph)} \frac{1}{\sqrt{\pi}}
\end{align*}
        for all $t>0$, yielding \eqref{eq:modified-addend-1-heat-content-formula} by \eqref{eq:decomposition-q-1-t}, since 
        \begin{align}\label{eq:decomposition-path-sets}
        \begin{aligned}
        \mathcal{P}_{\mVD}(\Graph) \sqcup \mathcal{P}_{\mVD,\mVN}(\Graph) &= \Big( \Pend{\mVD} \cap \Pst{\mVD} \Big) \sqcup \Big( \Pend{\mVN} \cap \Pst{\mVD} \Big) \\&= \underbrace{\big( \Pend{\mVD} \sqcup \Pend{\mVN} \big)}_{= \mathcal{P}(\Graph)} \cap \: \Pst{\mVD} = \Pst{\mVD};
        \end{aligned}
        \end{align}
        likewise $\big(\Pend{\mVD} \cap \mathcal{P}_1(\Graph)\big) \sqcup \big(\Pend{\mVN} \cap \mathcal{P}_1(\Graph)\big) = \mathcal{P}_1(\Graph)$, and consequently,
        \[
        \sum_{\vec{p} \in \mathcal{P}_1(\Graph)} \frac{1}{\sqrt{\pi}} = \sum_{\vec{\me} \in \mB} \frac{1}{\sqrt{\pi}} = \frac{\# \mB}{\sqrt{\pi}} = \frac{2\# \mE}{\sqrt{\pi}}.
        \]
        
        We determine \eqref{eq:modified-addend-2-heat-content-formula} in a similar manner: again, we decompose 
        \begin{align}\label{eq:decomposition-q-2-t}
       \begin{aligned}
       \widetilde{\widetilde{\mathcal{Q}_{t}}}(\Graph;\mVD) &= \sum_{\vec{p} \in \Pend{\mVD}} \alpha(\vec{p}) \bigg( H\bigg(\frac{\ell(\vec p_+)}{2\sqrt{t}} \bigg) - H\bigg(\frac{\ell(\vec  p_\pm)}{2\sqrt{t}} \bigg) \bigg) \\& \quad\qquad + \sum_{\vec{p} \in \Pend{\mVN}} \alpha(\vec{p}) \bigg( H\bigg(\frac{\ell(\vec p_+)}{2\sqrt{t}} \bigg) - H\bigg(\frac{\ell(\vec  p_\pm)}{2\sqrt{t}} \bigg) \bigg) \\& =: \widetilde{\widetilde{\mathcal{Q}_{t}^\mVD}}(\Graph;\mVD) + \widetilde{\widetilde{\mathcal{Q}_{t}^\mVN}}(\Graph;\mVD),
       \end{aligned} \qquad \text{for all $t>0$.}
       \end{align}
        and, using again \autoref{lem:decomposition-lemma},  write
        \begin{align}\label{eq:application-decomposition-lemma-2}
\begin{aligned}
\sum_{\vec{p} \in \Pend{\mW}} \alpha(\vec{p})H\bigg( \frac{\ell(\vec{p}_\pm)}{2\sqrt{t}} \bigg) &= \sum_{\vec{p} \in \Pend{\mW}} \sum_{\vec{q} \in \langle \vec{p} \rangle_-} \alpha(\vec{q}) H \bigg( \frac{\ell(\vec{q}_\pm)}{2\sqrt{t}} \bigg) + \sum_{\vec{p} \in \Pend{\mW} \cap \mathcal{P}_1(\Graph)} \alpha(\vec{p}) H \bigg( \frac{\ell(\vec{p}_\pm)}{2\sqrt{t}} \bigg) \\&=  \sum_{\vec{p} \in \Pend{\mW}} \sum_{\vec{q} \in \langle \vec{p} \rangle_-} \alpha(\vec{q}) H \bigg( \frac{\ell(\vec{p}_+)}{2\sqrt{t}} \bigg) + \sum_{\vec{p} \in \Pend{\mW} \cap \mathcal{P}_1(\Graph)} H\bigg( \frac{\ell(\vec{p})}{2\sqrt{t}} \bigg)
\end{aligned}
\end{align} 
for $\mW \in \{ \mVD, \mVN \}$ and all $t>0$, as $\vec{q}_\pm = (\vec{q}_-)_+ = \vec{p}_+$ for any $\vec{q} \in \langle \vec{p} \rangle_-$, and $\vec{p}_\pm = \cev{p}$ for all $\vec{p} \in \mathcal{P}_1(\Graph)$, thus $\ell(\vec{p}_\pm) = \ell(\cev{p}) = \ell(\vec{p})$ and $\alpha(\vec{p}_\pm) = \alpha(\cev{p}) = \alpha(\vec{p})= 1$ by definition.
Hence, like in \eqref{eq:help-q-1-t}, this implies 
\begin{align}\label{eq:help-q-2-t}
        \begin{aligned}
        \widetilde{\widetilde{\mathcal{Q}_{t}^\mW}}(\Graph;\mVD) &= \sum_{\vec{p} \in \Pend{\mW}} \alpha(\vec{p}) H\bigg(\frac{\ell(\vec p_+)}{2\sqrt{t}} \bigg) -  \sum_{\vec{p} \in \Pend{\mW}} \alpha(\vec{p}) H\bigg(\frac{\ell(\vec  p_\pm)}{2\sqrt{t}} \bigg) \\&= \sum_{\vec{p} \in \Pend{\mW}} \Bigg( \alpha(\vec{p}) - \sum_{\vec{q} \in \langle \vec{p} \rangle_-} \alpha(\vec{q}) \Bigg) H\bigg(\frac{\ell(\vec{p}_+)}{2\sqrt{t}} \bigg) - \sum_{\vec{p} \in \Pend{\mW} \cap \mathcal{P}_1(\Graph)} H\bigg( \frac{\ell(\vec{p})}{2\sqrt{t}} \bigg),
        \end{aligned}
\end{align} 
for $\mW \in \{ \mVD,\mVN \}$ and all $t>0$ and again by \eqref{eq:scattering-dual-terms} this leads to the expressions
\begin{align*}
\widetilde{\widetilde{\mathcal{Q}_{t}^\mVD}}(\Graph;\mVD) &= 2\sum_{\vec{p} \in \Pend{\mVD} \cap \Pst{\mVD}} \alpha(\vec{p})H\bigg(\frac{\ell(\vec p_+)}{2\sqrt{t}} \bigg) - \sum_{\vec{p} \in \Pend{\mVD} \cap \mathcal{P}_1(\Graph)} H\bigg(\frac{\ell(\vec{p})}{2\sqrt{t}} \bigg) \\&= 2\sum_{\vec{p} \in \mathcal{P}_{\mVD}(\Graph)} \alpha(\vec{p})H\bigg(\frac{\ell(\vec p_+)}{2\sqrt{t}} \bigg) - \sum_{\vec{p} \in \Pend{\mVD} \cap \mathcal{P}_1(\Graph)} H\bigg(\frac{\ell(\vec{p})}{2\sqrt{t}} \bigg)
\end{align*}
and
\begin{align*}
\widetilde{\widetilde{\mathcal{Q}_{t}^\mVN}}(\Graph;\mVD) &= 2\sum_{\vec{p} \in \Pend{\mVN} \cap \Pst{\mVD}} \alpha(\vec{p})H\bigg(\frac{\ell(\vec p_+)}{2\sqrt{t}} \bigg) - \sum_{\vec{p} \in \Pend{\mVN} \cap \mathcal{P}_1(\Graph)}  H\bigg(\frac{\ell(\vec{p})}{2\sqrt{t}} \bigg) \\&= 2\sum_{\vec{p} \in \mathcal{P}_{\mVD,\mVN}(\Graph)} \alpha(\vec{p})H\bigg(\frac{\ell(\vec p_+)}{2\sqrt{t}} \bigg) - \sum_{\vec{p} \in \Pend{\mVN} \cap \mathcal{P}_1(\Graph)}  H\bigg(\frac{\ell(\vec{p})}{2\sqrt{t}} \bigg)
\end{align*}
        for all $t>0$. These expressions for $\widetilde{\widetilde{\mathcal{Q}_{t}^\mVD}}(\Graph;\mVD)$ and $\widetilde{\widetilde{\mathcal{Q}_{t}^\mVN}}(\Graph;\mVD)$, respectively, then imply \eqref{eq:modified-addend-2-heat-content-formula} by \eqref{eq:decomposition-q-2-t}, using the same decompositions as in \eqref{eq:decomposition-path-sets} and afterwards, noting that
        \[
        \sum_{\vec{p} \in \mathcal{P}_1(\Graph)} H \bigg(\frac{\ell(\vec{p})}{2\sqrt{t}} \bigg) = \sum_{\vec{\me} \in \mB} H \bigg(\frac{\ell_\me}{2\sqrt{t}} \bigg) = 2\sum_{\me \in \mE} H \bigg(\frac{\ell_\me}{2\sqrt{t}} \bigg);
        \]
        this finishes the proof.      
      \end{proof}

      \begin{proof}[Proof of \autoref{thm:heat-content-formula-bif-mug}]
{We are going to show that
\begin{align}\label{eq:heat-content-formula-bif-mug-2}
\heatcont(\Graph;\mVD) &= \vert \Graph \vert - \frac{2\sqrt{t}}{\sqrt{\pi}} \# \mVD + 4\sqrt{t} \sum_{\vec{p} \in \mathcal{P}_\mVD(\Graph)} \alpha(\vec p) \, H\bigg( \frac{\ell(\vec p)}{2\sqrt{t}} \bigg),
\end{align}
which is equivalent to \eqref{eq:heat-content-formula-bif-mug}.
}            
            Plugging \eqref{eq:addend-1-heat-content-formula} and \eqref{eq:addend-2-heat-content-formula} in \eqref{eq:heat-cont-erf}, and then applying \autoref{lem:q1tq2t}, we reach at
      \begin{align}\label{eq:decomposition-pre-final-heat-content-formula}
      \begin{aligned}
      \heatcont(\Graph;\mVD) &= \vert \Graph \vert +\frac{2\sqrt{t}}{\sqrt{\pi}} \# \mE - 2\sqrt{t}\sum_{\me \in \mE} H \bigg( \frac{\ell_\me}{2\sqrt{t}} \bigg) + \sqrt{t} \Big(\widetilde{\mathcal{Q}_{t}}(\Graph;\mVD) - \widetilde{\widetilde{\mathcal{Q}_{t}}}(\Graph;\mVD) \Big) \\&= \vert \Graph \vert + 2\sqrt{t} \sum_{\vec{p} \in \Pst{\mVD}} \alpha(\vec p)\Bigg( H\bigg( \frac{\ell(\vec{p})}{2\sqrt{t}} \bigg) - H\bigg( \frac{\ell(\vec{p}_+)}{2\sqrt{t}} \bigg) \Bigg),
      \end{aligned}
      \end{align}
      for $t>0$. Now,
by \autoref{lem:decomposition-lemma}, we can write
    \begin{align*}
    \sum_{\vec{p} \in \Pst{\mVD}} \alpha(\vec p)H\bigg( \frac{\ell(\vec{p}_+)}{2\sqrt{t}} \bigg) &= \sum_{\vec{p} \in \Pst{\mVD}} \sum_{\vec{q} \in \langle \vec{p} \rangle_+} \alpha(\vec{q}) H\bigg( \frac{\ell(\vec{q}_+)}{2\sqrt{t}} \bigg) + \sum_{\vec{p} \in \Pst{\mVD} \cap \mathcal{P}_1(\Graph)} \alpha(\vec{p}) H\bigg( \frac{\ell(\vec{p}_+)}{2\sqrt{t}} \bigg) \\&= \sum_{\vec{p} \in \Pst{\mVD}} \sum_{\vec{q} \in \langle \vec{p} \rangle_+} \alpha(\vec{q}) H\bigg( \frac{\ell(\vec{p})}{2\sqrt{t}} \bigg) + \sum_{\vec{p} \in \Pst{\mVD} \cap \mathcal{P}_1(\Graph)} \frac{1}{\sqrt{\pi}},
    \end{align*}
    for every $t>0$, since $\vec{q}_+ = \vec{p}$ for any $\vec{q} \in \langle \vec{p} \rangle_+$, and $\ell(\vec{p}_+) = 0$ as well as $\alpha(\vec{p})=1$ for every $\vec{p} \in \mathcal{P}_1(\Graph)$; also, $H(0)=\frac{1}{\sqrt{\pi}}$ follows immediately from \eqref{eq:anti-primitive-erfc}.  Moreover, by \autoref{lem:scattering-coeff-dual}
\begin{align}\label{eq:scattering-dual-terms-2}
     \alpha(\vec{p}) - \sum_{\vec{q} \in \langle \vec{p} \rangle_+} \alpha(\vec{q}) = \begin{cases} 2\alpha(\vec{p}), & \text{if $\vec{p} \in \Pend{\mVD}$}, \\ 0, & \text{if $\vec{p} \in \Pend{\mVN}$,}
     \end{cases} \qquad \text{for all $\vec{p} \in \Pst{\mVD}$}.
\end{align}      
Using \autoref{thm:first-roth-like-expansion}.(i), this eventually leads to
    \begin{align}\label{eq:respresentation-rest-heat-content-formula}
    \begin{aligned}
    &\sum_{\vec{p} \in \Pst{\mVD}} \alpha(\vec p) \Bigg( H\bigg( \frac{\ell(\vec{p})}{2\sqrt{t}} \bigg) - H\bigg( \frac{\ell(\vec{p}_+)}{2\sqrt{t}} \bigg) \Bigg) \\
\\& \qquad\quad = \sum_{\vec{p} \in \Pst{\mVD}} \Bigg( \alpha(\vec{p}) - \sum_{\vec{q} \in \langle \vec{p} \rangle_+} \alpha(\vec{q}) \Bigg) H\bigg( \frac{\ell(\vec{p})}{2\sqrt{t}}\bigg) -\sum_{\vec{p} \in \Pst{\mVD} \cap \mathcal{P}_1(\Graph)} \frac{1}{\sqrt{\pi}} \\
& \qquad\quad 
= 2\sum_{\vec{p} \in \Pst{\mVD} \cap \Pend{\mVD}} \alpha(\vec{p})H\bigg( \frac{\ell(\vec{p})}{2\sqrt{t}}\bigg) - \sum_{\vec{\me} \in \mB, \: \partial^-(\vec{\me}) \in \mVD} \frac{1}{\sqrt{\pi}} \\
&\qquad\quad
    = 2\sum_{\vec{p} \in \mathcal{P}_\mVD(\Graph)} \alpha(\vec{p})H\bigg( \frac{\ell(\vec{p})}{2\sqrt{t}}\bigg) - \frac{\# \mVD}{\sqrt{\pi}}.
    \end{aligned}
    \end{align}
    Now plugging \eqref{eq:respresentation-rest-heat-content-formula} into \eqref{eq:decomposition-pre-final-heat-content-formula} finally yields \eqref{eq:heat-content-formula-bif-mug-2}, and thus the claimed formula in \eqref{eq:heat-content-formula-bif-mug}.
\end{proof}       
\begin{exa}[Lasso graph]\label{ex:lasso}
Let us examine \eqref{eq:heat-content-formula-bif-mug} where $\Graph$ is a \emph{lasso graph} (see \autoref{fig:lasso}) with a single Dirichlet vertex $\mVD = \{ \mv_\mathrm{D} \}$ connected to an edge $\me_1$ of length $\ell_1 >0$ with a standard vertex $\mv_\mathrm{N}$ at the other end which is connected to a loop $\me_2$ of length $\ell_2 > 0$:
        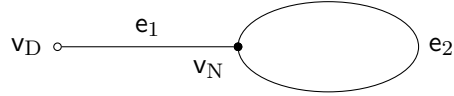
\begin{figure}[h]
        \begin{tikzpicture}[scale=0.6]
      \tikzset{enclosed/.style={draw, circle, inner sep=0pt, minimum size=.1cm, fill=black}, every loop/.style={}}
      \node[enclosed, label = {below left: $\mv_\mathrm{N}$}] (Z) at (0,2) {};
      \node[enclosed, label = {left: $\mv_\mathrm{D}$}, fill=white] (A) at (-4,2) {};
      \draw[-] (Z) edge node[above] {$\me_1$} (A) node[midway, above] (edge1) {};
      \draw[-] (Z) arc [start angle=-180, end angle=180,
                  x radius=2cm, 
                  y radius=10mm] node[right] [pos=0.5] {$\me_2$};
      \node[enclosed] (Z') at (0,2) {};
     \end{tikzpicture}
     \vspace{-0.7cm}
     \caption{A lasso graph with a Dirichlet condition at $\mv_\mathrm{D}$.} \label{fig:lasso}
     \end{figure}
     
Given $\vec p \in \mathcal{P}(\{ \mvD \})$, i.e., a directed path starting and ending at $\mvD$, we can count the number
\begin{itemize}
\item 
$R_{\vec p}(\mv_\mathrm{D})$ of times $\vec{p}$ touches $\mvD$;
\item
$R_{\vec p}^{(1)}(\mv_\mathrm{N})$ of times $\vec{p}$ arrives along $\me_1$, touches $\mvN$, and is reflected back into $\me_1$;
\item
$R_{\vec p}^{(2)}(\mv_\mathrm{N})$ of times $\vec{p}$ arrives along $\me_2$, touches $\mvN$, and is reflected back into $\me_2$;
\item $T_{\vec p}^{(1)}(\mv_\mathrm{N})$ of times $\vec{p}$ arrives along $\me_2$ and goes into $\me_1$ by transferring through $\mvN$;
\item $T_{\vec p}^{(2)}(\mv_\mathrm{N})$ of times $\vec{p}$ arrives along $\me_2$ and goes into $\me_2$ by transferring through $\mvN$.
\end{itemize}
 respectively.
        This leads to a length of
        \[
        \ell(\vec p) = \big(R_{\vec{p}}(\mv_\mathrm{D}) + R_{\vec{p}}^{(1)}(\mv_\mathrm{N}) +T_{\vec p}^{(1)}(\mv_\mathrm{N})+ 1 \big)\ell_1 + \big(R_{\vec{p}}^{(2)}(\mv_\mathrm{N}) + T_{\vec{p}}^{(2)}(\mv_\mathrm{N}) \big) \ell_2
        \]
        and a scattering coefficient
        \begin{align*}
        \alpha(\vec p) &= (-1)^{R_{\vec p}(\mv_\mathrm{D})}\bigg(-\frac{1}{3}\bigg)^{R_{\vec p}^{(1)}(\mv_\mathrm{N}) + R_{\vec p}^{(2)}(\mv_\mathrm{N})} \bigg(\frac{2}{3}\bigg)^{T_{\vec p}^{(1)}(\mv_\mathrm{N}) + T_{\vec p}^{(2)}(\mv_\mathrm{N})} 
        \end{align*}
        thus, letting $L_1(\vec{p}) := R_{\vec{p}}(\mv_\mathrm{D}) + R_{\vec{p}}^{(1)}(\mv_\mathrm{N}) +T_{\vec p}^{(1)}(\mv_\mathrm{N})$ and $L_2(\vec{p}) := R_{\vec{p}}^{(2)}(\mv_\mathrm{N}) + T_{\vec{p}}^{(2)}(\mv_\mathrm{N})$ the  heat content formula \eqref{eq:heat-content-formula-bif-mug} -- which is based on path enumeration -- can be reformulated in the following purely analytic way
        \begin{align*}
&\heatcont(\mathcal{\Graph};\{ \mv_\mathrm{D} \}) = \ell_1 + \ell_2 - \frac{2\sqrt{t}}{\sqrt{\pi}} \\
& \:\, + 4\sqrt{t} \sum_{\vec p \in \mathcal{P}(\{\mv_\mathrm{D} \})} (-1)^{R_{\vec p}(\mv_\mathrm{D})}\bigg(-\frac{1}{3}\bigg)^{R_{\vec p}^{(1)}(\mv_\mathrm{N}) + R_{\vec p}^{(2)}(\mv_\mathrm{N})} \bigg(\frac{2}{3}\bigg)^{T_{\vec p}^{(1)}(\mv_\mathrm{N}) + T_{\vec p}^{(2)}(\mv_\mathrm{N})} H\bigg( \frac{L_1(\vec{p})\ell_1 + L_2(\vec{p}) \ell_2}{2\sqrt{t}} \bigg) \\
& \:\,= \ell_1+\ell_2 - \frac{2\sqrt{t}}{\sqrt{\pi}} + 8\sqrt{t} \sum_{m \in \mathbb{N}_0} (-1)^m \sum_{\substack{n,k,\ell,j \in \mathbb{N}_0 \\ n+k \leq m+1}} \bigg( -\frac{1}{3} \bigg)^{n + \ell}\bigg(\frac{2}{3} \bigg)^{k+j} H \bigg( \frac{(m+n+\ell+1)\ell_1 + (k+j) \, \ell_2}{2\sqrt{t}} \bigg),
        \end{align*}
using the fact that each \textit{un}directed path is counted twice to account for the orientation of the elements of $\mathcal P_\mVD(\Graph)$.
 
(Note that each path in $\mathcal{P}(\{ \mvD \})$ (which, by definition, is necessarily nontrivial) either stops at $\mvD$ or reflects at $\mvD$ after hitting $\mvN$ at least once; hence $R_{\vec{p}}^{\ell_1}(\mvN) + T_{\vec{p}}^{\ell_1}(\mvN) \leq R_{\vec{p}}(\mvD) + 1$.)
\end{exa}
\begin{rem}\label{rem:general-vertex-conditions}
There is no particular reason for restricting to the case of standard (continuity and Kirchhoff) conditions in $\mV\setminus \mVD$ and/or to Dirichlet conditions at the ``boundary'' of $\Graph$. The theory of heat content as presented above is mainly based on two ingredients: 1) the fact that the Gaussian kernels is the fundamental solution of the heat equation on $\R$ and 2) the fact that the heat kernel for the Laplacian with standard vertex conditions is positive. By~\cite[Theorem~6.85]{Mug14}, the heat kernel's positivity still holds if, more generally, the boundary conditions are given by 
\[
\Gamma_\circ u\in Y \quad \hbox{and}\quad \Gamma^\circ u+S\Gamma_\circ u\in Y^\perp
\]
for a sublattice $Y$ of $\C^{2\ell_\me}$ and a linear operator $S$ on $Y$ such that the orthogonal projector onto $Y$ is a positive operator and $(\e^{-tS})_{t\ge 0}$ is a positive semigroup, see also \cite{CarMug09,Kur19}; here 
\[
\Gamma_\circ u:=\begin{pmatrix}
u(0)\\ u(\ell)
\end{pmatrix}^\top\quad\hbox{and}\quad \Gamma^\circ u:=\begin{pmatrix}
-u'(0)\\ u'(\ell)\end{pmatrix}.
\]
By general Perron--Frobenius theory, the generators of such semigroups have a dominant eigenvalue associated with a positive eigenfunction. In order to avoid trivialities, it is sufficient to define the heat content by evaluating the heat kernel at a positive function that does \textit{not} belong to the generator's null space. In particular, apart from the case of the Laplacian with natural conditions at each vertex, $\mathbf 1$ is not in the null space of a Laplacian realization, and hence not a fixed vector of the semigroup, for any positive realization, since it does not satisfies \textit{weighted} continuity conditions; and in this case we can start our machinery again, upon suitably adapting the scattering coefficients in~\autoref{defi:scatt-coeff}, see e.g.~\cite{KosPotSch07}. Observe that in this case \autoref{lem:path-sum-formula} remains valid by \cite[Corollary~3.4]{KosPotSch07}.
\end{rem}

\section{Small-time asymptotics for the heat content}\label{sec:small-time-asymptotics}
Using the heat content formula  in \autoref{thm:heat-content-formula-bif-mug}, it is possible to deduce 
the following \emph{small-time asymptotic} for the heat content.

\begin{theorem}[Small-time asymptotics for the heat content]\label{thm:two-term-small-time-asymptotic-heat-content}
The heat content satisfies
    \begin{align}\label{eq:small-time-asymp-estimate}
    \left\vert \heatcont(\Graph;\mVD) - \vert \Graph \vert + \frac{2\sqrt{t}}{\sqrt{\pi}} \# \mVD \right\vert \leq \frac{8\sqrt{t}}{\sqrt{\pi}} \frac{
       \e^{-\frac{\ell_{\min}^2}{4t}}}{1-d_{\max} \e^{-\frac{\ell_{\min}^2}{2t}}}  \qquad \text{for all $0 < t \leq t_0 < \frac{\ell_{\min}^2}{2\log d_{\max}}$.}
    \end{align}
    \begin{align}\label{eq:small-time-asymp}
    \mathcal{Q}_t(\mathcal{G};\mVD) = \vert \Graph \vert - \frac{2 \sqrt{t}}{\sqrt{\pi}}\# \mVD + \mathcal{O}\Big(\sqrt{t}\,\e^{-\frac{\ell_{\min}^2}{4t}}\Big) \qquad \text{as} \:\: t \rightarrow 0^+.
    \end{align}
\end{theorem}
Here we have denoted by $d_{\max} := \max_{\mv \in \mV} \deg(\mv)$ the \emph{maximal degree} and by $\ell_{\min} := \min_{\me \in \mE} \ell_\me$ the \emph{minimal edge length} of $\Graph$. 

Also, $\frac{\ell_{\min}^2}{2\log d_{\max}}:=+\infty$ whenever $\Graph$ is a path graph (recall that, by \autoref{ass:graph}, we are always assuming $\Graph$ to have no vertices of degree 2) and, hence, $d_{\max} =1$; in other words, in the case of intervals the estimate in \eqref{eq:small-time-asymp-estimate} for the heat content in the previous theorem holds for all $t > 0$ (as we, indeed, already know from \eqref{eq:qt-interval-elem}).
\begin{proof}
Given $\Graph$ and $\mVD \subset \mV$, we denote by
\begin{align}\label{eq:ltgvd}
\mathfrak{L}_t(\Graph;\mVD) := 4\sqrt{t}  \sum_{\vec p \in \mathcal{P}_\mVD(\Graph)} \alpha(\vec p) H \bigg( \frac{\ell(\vec p)}{2\sqrt{t}} \bigg), \qquad t>0,
\end{align}
the remainder term in~\eqref{eq:heat-content-formula-bif-mug}.
By \autoref{lem:properties-H} and the definition of $H$ we find
\begin{align}\label{eq:first-estimate-r-t}
|\mathfrak{L}_t(\Graph;\mVD)| \leq \frac{4\sqrt{t}}{\sqrt{\pi}} \sum_{\vec p \in \mathcal{P}_\mVD(\Graph)}| \alpha(\vec p) | \mathrm{e}^{-\frac{\ell(\vec p)^2}{4t}} \qquad \text{for all $t>0$.}
\end{align}
Given a path $\vec p \in \mathcal{P}_\mVD(\Graph)$ of combinatorial length $n\ge 2$, 
one can trivially estimate $|\alpha(\vec p)|\le 1$; moreover, there are at most $(d_{\max})^{n-1}$ such paths, since there are at most $(d_{\max})^{n-1}$ possibilities to reach $\mv_+(\vec{p})$  from $\mv_-(\vec{p})$ in $n$ steps, because at any vertex $\mv' \in \mV$ that is traversed by $\vec{p}$ there are at most $\deg(\mv')\le d_{\max}$
possible choices to reach the next vertex from $\mv'$; in other words,
\begin{align}\label{eq:estimate-comb-n-dmax}
\# (\mathcal{P}_n(\Graph) \cap \mathcal{P}_\mVD(\Graph) ) \leq 2(d_{\max})^{n-1} \quad \text{for every $n \in \mathbb{N}$}
\end{align}
where the factor $2$ appears because each \textit{un}directed path is counted twice to account for the orientation of the elements of $\mathcal P_\mVD(\Graph)$.
Additionally, the length of these paths can be estimated from below by $\ell(\vec p) \geq n\ell_{\min}$. Noting that $\mathcal{P}_\mVD(\Graph) = \bigcup_{n \in \mathbb{N}} \mathcal{P}_n(\Graph) \cap \mathcal{P}_\mVD(\Graph)$, this yields the estimate
\begin{align}\label{eq:help-estimate-small-time-asymp}
\sum_{\vec p \in \mathcal{P}_\mVD(\Graph)} |\alpha(\vec p)| \e^{-\frac{\ell(\vec p)^2}{4t}} = \sum_{n = 1}^\infty \sum_{\vec p \in \mathcal{P}_n(\Graph) \cap \mathcal{P}_\mVD(\Graph)}| \alpha(\vec p)| \mathrm{e}^{-\frac{\ell(\vec p)^2}{4t}} \leq \frac{2}{d_{\max}}\sum_{n=1}^\infty \big(d_{\max} \big)^{n} \mathrm{e}^{-\frac{n^2\ell_{\min}^2}{4t}}.
\end{align}
    Now, we use
        \begin{equation}\label{eq:sim-borharjon}
        2n-1\le n^2
\qquad\hbox{for all }n \in \mathbb{N},
    \end{equation}
and the fact that $d_{\max} \e^{-\frac{\ell_{\min}^2}{2t}} < 1$ holds if and only if  $0 < t < \frac{\ell_{\min}^2}{2 \log d_{\max}}$: combining \eqref{eq:first-estimate-r-t} with \eqref{eq:help-estimate-small-time-asymp}, considering \eqref{eq:sim-borharjon}, and using the geometric series afterwards, we reach at the following estimate for $\mathfrak{L}_t(\Graph;\mVD)$:
    \begin{align}\label{eq:final-estimate-long-term-significant-part}
    \begin{aligned}
        \vert \mathfrak{L}_t(\Graph;\mVD) \vert &\leq \frac{8\sqrt{t}}{\sqrt{\pi}}\frac{1}{d_{\max}} \, \e^{\frac{\ell_{\min}^2}{4t}}\sum_{n=1}^\infty \big(d_{\max}\big)^n \e^{-\frac{n\ell_{\min}^2}{2t}} \\&= \frac{8\sqrt{t}}{\sqrt{\pi}} \frac{1}{d_{\max}} \, \e^{\frac{\ell_{\min}^2}{4t}}\sum_{n=1}^\infty \Big( d_{\max} \e^{-\frac{\ell_{\min}^2}{2t}} \Big)^n \\&= \frac{8\sqrt{t}}{\sqrt{\pi}}\frac{1}{d_{\max}} \,\e^{\frac{\ell_{\min}^2}{4t}} \frac{d_{\max} \e^{-\frac{\ell_{\min}^2}{2t}}}{1-d_{\max} \e^{-\frac{\ell_{\min}^2}{2t}}} \\&= \frac{8\sqrt{t}}{\sqrt{\pi}} \frac{
       \e^{-\frac{\ell_{\min}^2}{4t}}}{1-d_{\max} \e^{-\frac{\ell_{\min}^2}{2t}}} \in \mathcal{O} \Big(
         \sqrt{t}\,\e^{-\frac{\ell_{\min}^2}{4t}} \Big),
         \end{aligned}
    \end{align}
    for all $0 < t \leq t_0 < \frac{\ell_{\min}^2}{2 \log d_{\max}}$, eventually yielding \eqref{eq:small-time-asymp}.
\end{proof}
\begin{rem}
(i) 
A similar decay-rate
\begin{align}\label{eq:small-time-borthwick}
\ptGD(x,x) =
\begin{cases} \frac{1}{\sqrt{4\pi t}} \Bigg( 1+ \bigg( \frac{2}{\deg(\mv_x)} -1\bigg) \e^{-\frac{\dist (x,\mV)^2}{t}} \Bigg) + \frac{1}{\sqrt{4\pi t}} \mathcal{O}\Big( \# \mE \, \e^{-\frac{\ell_{\min}^2}{4t}} \Big) \\
\frac{1}{\sqrt{4\pi t}}  + \frac{1}{\sqrt{4\pi t}} \mathcal{O}\Big( \# \mE \, \e^{-\frac{\ell_{\min}^2}{4t}} \Big) \\
\end{cases}
\quad \hbox{as }t\to 0+
\end{align}
 for the heat kernel $\ptGD$ on the diagonal  was proved in \cite[Proposition~3.1]{BorHarJon22} in the case where $\mVD = \emptyset$, for all $0 < t \leq t_0 < \frac{\ell_{\min}^2}{2 \log \# \mE}$, where the second expression holds for each $x\in \Graph$ that is the midpoint of an edge and the first expression for all other $x\in \Graph$. Here $\dist (x,\mV) := \min_{\mv \in \mV} \dist(x,\mv)$ denotes the minimal distance of $x \in \Graph$ to the vertex set $\mV$ and $\mv_x \in \mV$ the corresponding vertex which has minimal distance to $x$, i.e., the vertex in $\mV$ such that $\dist_\Graph (x,\mV) = \dist_\Graph (x,\mv_x)$. But -- by the same argumentation used in the proof of \autoref{thm:two-term-small-time-asymptotic-heat-content} -- one may sharpen \eqref{eq:small-time-borthwick} replacing $\# \mE$ by $d_{\max}$. 
 
(ii) In the case of $\mVD \neq \emptyset$, it is possible to adapt the methods used in \cite{BorHarJon22}, to deduce
\begin{align*}
\ptGD(x,x) = \begin{cases} \frac{1}{\sqrt{4\pi t}} \Bigg( 1+ \bigg( \frac{2}{\deg(\mv_x)} \cdot \mathbf{1}_{\mV_\mathrm{N}}(\mv_x) -1\bigg) \e^{-\frac{\dist (x,\mV)^2}{t}} \Bigg) + \frac{1}{\sqrt{4\pi t}} \mathcal{O}\Big( d_{\max} \e^{-\frac{\ell_{\min}^2}{4t}} \Big) \\
\frac{1}{\sqrt{4\pi t}}  + \frac{1}{\sqrt{4\pi t}} \mathcal{O}\Big( d_{\max} \e^{-\frac{\ell_{\min}^2}{4t}} \Big) \end{cases}
\end{align*}
for $x \in \Graph \setminus \mVD$ and all $0 < t \leq t_0 < \frac{\ell_{\min}^2}{2\log d_{\max}}$. Observe that such an asymptotic analysis based on the methods in \cite{BorHarJon22} delivers an exponential decay, thus improving the polynomial decay that was proved in \cite[Proposition~8.1]{BolEggRue15} for general Schrödinger operators with general self-adjoint vertex conditions, including — as in the setting of our paper — mixed (standard and Dirichlet) vertex conditions.
\end{rem}
As the remainder term $\mathfrak{L}_t(\Graph;\mVD)$ in particular belongs to $\sqrt{t}\, \mathcal{O}\big(\e^{-\frac{\ell_{\min}^2}{4t}}\big)$ for sufficiently small $t > 0$ and $\mathcal{Q}_t(\Graph;\mVD) \rightarrow 0$ as $t \rightarrow 0^+$, the heat content $\mathcal{Q}_t(\mathcal{G})$ cannot have nontrivial polynomial terms of order $\alpha \in \mathbb{R} \setminus \{0,\frac{1}{2}\}$ for small $t > 0$. Thus, we reach at the small-time expansion for the heat content which represents a counterpart of \cite[Theorem~1.1]{BerGil94} to quantum graphs. 
\begin{corollary}[Asymptotic small-time expansion]\label{cor:vandenberg-gilkey-small-time}
The map $(0,\infty) \ni t \mapsto \mathcal{Q}_t(\Graph;\mVD)$ admits an asymptotic expansion near $0$, more precisely,
\begin{equation}\label{eq:asympt-small-time-exp}
\heatcont(\Graph;\mVD)\sim \sum_{n=0}^\infty \beta_n t^\frac{n}{2},\qquad \hbox{as }t\rightarrow 0^+
\end{equation}
for constants $\beta_n = \beta_n(\Graph;\mVD) \in \R$, $n\in \mathbb{N}_0$ (depending on the graph $\Graph$) given by 
\begin{align*}
\beta_k(\Graph;\mVD) = \begin{cases} \vert \Graph \vert, & \text{if $k=0$,} \\ -\frac{2\# \mVD}{\sqrt{\pi}}, & \text{if $k=1$,} \\ 0, & \text{else,}
\end{cases} \qquad k \in \mathbb{N}_0.
\end{align*}
\end{corollary}

\begin{proof}
This is an immediate consequence of \autoref{thm:heat-content-formula-bif-mug} and \autoref{thm:two-term-small-time-asymptotic-heat-content}.
\end{proof}

\begin{rem}
The asymptotic expansion~\eqref{eq:asympt-small-time-exp} has been known to hold, among others, for planar domains with polygonal boundary~\cite[Theorem~1]{BerSri90} and compact Riemannian manifolds with smooth boundary since~\cite[Theorem~1.2]{BerGil94} -- but seems to be unknown for general metric measure spaces and, hitherto, even for metric graphs. Quite recently,  a similar small-time asymptotic for the heat content on $\mathsf{RCD}(K,N)$-spaces has been proved in \cite[Theorem~1.1]{CapRos24}: it is noteworthy that the remainder term derived there is only of \emph{polynomial} type $\mathcal{O}\Big(t^{\frac{2(1+\rho)-1}{2(1+\rho)}} \Big)$ (for certain $\rho > 0$). More nontrivial terms than in~\eqref{eq:asympt-small-time-exp} may, therefore, generally exist in this small-time expansion. 
\end{rem}

We conclude this section presenting a result that relates the heat flow out of a subgraph of $\Graph$ with its topological boundary, in the spirit of a Caccioppoli-type heat kernel-based description of the perimeter of a set, as studied in the Euclidean case, cf.~\cite[Theorem~3.3]{MirPalPar07}.

\begin{theorem}\label{thm:subgraph-small-time}
    Let $\Graph$ be as in \autoref{ass:graph} and $\mathcal H$ be a closed and connected subset of \(\Graph\setminus\mVD\) whose boundary $\partial \mathcal{H}$ in $\Graph$ does not contain any vertices of $\Graph$ of degree $\ge 3$. Then
    \begin{align}
        \lim_{t \rightarrow 0^+} \frac{\sqrt{\pi}}{\sqrt{t}} \int_\mathcal{H} \int_{\Graph \setminus \mathcal{H}} \ptGD(x,y) \dd y \dd x = \# \partial \mathcal{H} .
    \end{align}
\end{theorem}

\begin{proof}
First, we can assume without loss of generality -- upon possibly subdividing edges -- that each edge in $\Graph$ either has empty intersection with $\mathcal H$ or $\Graph \setminus \mathcal{H}$, or else it contains precisely one element of the topological boundary of $\mathcal H$ {in its interior}: in the latter case, since $\partial \mathcal{H}$ does not contain any vertices of $\Graph$, the Lebesgue measure $\ell_\me^\mathcal{H}:=\vert \me\cap \mathcal H\vert$ of $\me \cap \mathcal{H}$ is positive; moreover, we let
\[
\mE_{\Graph}^{\partial \mathcal{H}} := \{ \me \in \mE_\Graph \: : \: \me \cap \partial \mathcal{H} \neq \emptyset\}. 
\]
Like in the proof of \autoref{thm:heat-content-formula-bif-mug}, one can decompose $\int_{\mathcal{H}} \int_{\mathcal{G} \setminus \mathcal{H}} \ptGD(x,y) \dd y \dd x$ -- using \eqref{eq:heat-kernel-KPS-2b} -- as
\begin{align*}
\begin{aligned}
    \int_{\mathcal{H}} \int_{\mathcal{G} \setminus \mathcal{H}} \ptGD(x,y) \dd y \dd x &= \frac{1}{\sqrt{4\pi t}} \int_{\mathcal{H}} \int_{\Graph \setminus \mathcal{H}} \delta_{\me,\me'} \e^{-\frac{\dist(x_\me,y_{\me'})^2}{4t}} \dd x_\me \dd y_{\me'} \\& \qquad\quad+ \frac{1}{\sqrt{4\pi t}} \int_{\mathcal{H}} \int_{\Graph \setminus \mathcal{H}} \sum_{\vec p \in \mathcal{P}_{\geq 2}(x,y)} \alpha(\vec p) \e^{-\frac{\ell(\vec p)^2}{4t}} \dd x \dd y \\&=: \widetilde{\mathcal{Q}_{t}}(\Graph; \mathcal H) + \widetilde{\widetilde{\mathcal{Q}_{t}}}(\Graph;\mathcal H).
    \end{aligned} \qquad \text{for $t>0$.}
\end{align*}
We determine on $\widetilde{\mathcal{Q}_{t}}(\Graph;\mathcal{H})$ first: indeed, since for any edge $\me \in \mE_\Graph$, $\vert \me \cap (\Graph \setminus \mathcal{H})\vert > 0$ and $ \vert \me \cap \mathcal{H} \vert > 0$ is by construction of the subgraph $\mathcal{H}$ only possible if and only if $\me \in \mE_{\mathcal{H}}^{\partial \mathcal{H}}$,
one observes
\begin{align}\label{eq:boundary-q1}
\widetilde{\mathcal{Q}_{t}}(\Graph;\mathcal{H}) = \frac{1}{\sqrt{4\pi t}} \sum_{\me \in \mE_{\Graph}^{\partial \mathcal{H}}} \int_{\ell_{\me}^\mathcal{H}}^{\ell_\me} \int_0^{\ell_\me^{\mathcal{H}}} \e^{-\frac{(x-y)^2}{4t}} \dd y \dd x, \qquad t > 0.
\end{align}
Moreover, analogous to \eqref{eq:triv-heat-content-leibniz} and \eqref{eq:roth-second-auxiliary-eq}, the Leibniz integral rule yields
\begin{align*}
\frac{1}{\sqrt{4\pi t}}\int_\alpha^\beta \int_0^\alpha \e^{-\frac{(x-y)^2}{4t}} \dd y \dd x = \sqrt{t} \Bigg( H\bigg( \frac{\beta}{2\sqrt{t}} \bigg) - H\bigg( \frac{\alpha}{2\sqrt{t}}\bigg) - H\bigg( \frac{\beta-\alpha}{2\sqrt{t}}\bigg) + H(0) \Bigg) 
\end{align*}
for all $\beta \geq \alpha \in \mathbb{R}_+$ and all $t > 0$: together with \eqref{eq:boundary-q1}, this implies that
\begin{align*}
    \widetilde{\mathcal{Q}_{t}}(\Graph;\mathcal{H}) &= \sqrt{t} \sum_{\me \in \mE_{\Graph}^{\partial \mathcal{H}}} \Bigg( H\bigg( \frac{\ell_\me}{2\sqrt{t}} \bigg) - H\bigg( \frac{\ell_\me^\mathcal{H}}{2\sqrt{t}}\bigg) - H\bigg( \frac{\ell_\me-\ell_\me^\mathcal{H}}{2\sqrt{t}}\bigg) + H(0) \Bigg)  \\&= \frac{\sqrt{t}}{\sqrt{\pi}} \sum_{\me \in \mE_{\Graph}^{\partial \mathcal{H}}} 1 + \sqrt{t} \sum_{\me \in \mE_{\Graph}^{\partial \mathcal{H}}} \Bigg( H\bigg( \frac{\ell_\me}{2\sqrt{t}} \bigg) - H\bigg( \frac{\ell_\me^\mathcal{H}}{2\sqrt{t}}\bigg) - H\bigg( \frac{\ell_\me-\ell_\me^\mathcal{H}}{2\sqrt{t}}\bigg) \Bigg) \\&= \frac{\sqrt{t}}{\sqrt{\pi}} \# \partial \mathcal{H} + \sqrt{t} \sum_{\me \in \mE_{\Graph}^{\partial \mathcal{H}}} \Bigg( H\bigg( \frac{\ell_\me}{2\sqrt{t}} \bigg) - H\bigg( \frac{\ell_\me^\mathcal{H}}{2\sqrt{t}}\bigg) - H\bigg( \frac{\ell_\me-\ell_\me^\mathcal{H}}{2\sqrt{t}}\bigg) \Bigg), 
\end{align*}
since $\# \mE_{\Graph}^{\partial \mathcal{H}} = \# \partial \mathcal{H}$ by construction (recall that we suppose each edge that has nonempty intersection with both $\mathcal{H}$ and $\Graph \setminus \mathcal{H}$ to contain exactly one element of $\partial \mathcal{H}$). This yields
\begin{align}\label{eq:small-time-boundary-q1}
    \frac{\sqrt{\pi}}{\sqrt{t}} \widetilde{\mathcal{Q}_{t}}(\Graph;\mathcal{H}) &= \# \partial \mathcal{H} + \sqrt{\pi} \sum_{\me \in \mE_{\Graph}^{\partial \mathcal{H}}} \Bigg( H\bigg( \frac{\ell_\me}{2\sqrt{t}} \bigg) - H\bigg( \frac{\ell_\me^\mathcal{H}}{2\sqrt{t}}\bigg) - H\bigg( \frac{\ell_\me-\ell_\me^\mathcal{H}}{2\sqrt{t}}\bigg) \Bigg) \qquad \text{for $t>0$.}
\end{align}
Using now the estimate in \eqref{eq:estimate-erfc}, we observe that the second addend in the right-hand side of \eqref{eq:small-time-boundary-q1} can be estimated from above by $3\sqrt{\pi} (\# \partial \mathcal{H}) \e^{-\frac{\dist(\partial \mathcal{H};\mV)}{4t}} \in \mathcal{O}\big(\e^{-\frac{\mathrm{dist}(\partial \mathcal{H};\mV)}{4t}}\big)$ as $t \rightarrow 0^+$, since 
\[
\min\left\{\ell_\me,\ell_\me^\mathcal{H}, \ell_\me - \ell_\me^\mathcal{H} \right\} \geq \min_{\mv \in \mV, y \in \partial \mathcal H} \dist(\mv,y) = \dist(\partial \mathcal{H} ;\mV)>0 \qquad \text{for all $\me \in \mE$.}
\]
This implies 
\begin{align}\label{eq:trivial-boundary-part-caccioppoli}
\frac{\sqrt{\pi}}{\sqrt{t}} \widetilde{\mathcal{Q}_{t}}(\Graph;\mathcal{H}) = \# \partial \mathcal{H} + 
\mathcal{O}\Big( \e^{-\frac{\dist(\partial \mathcal{H};\mV)}{4t}} \Big) \qquad \hbox{as }t \rightarrow 0^+.
\end{align}
We next evaluate on $\widetilde{\widetilde{\mathcal{Q}_{t}}}(\Graph;\mathcal H)$: let $x \in \Graph \setminus \mathcal{H}$ and $y \in \mathcal{H}$. Any directed path $\vec{p} \in \mathcal{P}_{  \geq 2}(x,y)$ then has to traverse at least one vertex in $\mV$ but also a point lying on the boundary $\partial \mathcal{H}$ as it connects points from $\Graph \setminus \mathcal{H}$ and $\mathcal{H}$, respectively. Therefore
 one estimates
\begin{align*}
\ell(\vec{p}) \geq \dist(\partial \mathcal{H}; \mV) > 0
\end{align*}
and therefore, we observe at first
\begin{align}\label{eq:first-estimate-second-term-cacciopoli}
    \widetilde{\widetilde{\mathcal{Q}_{t}}}(\Graph;\mathcal H) \leq \frac{\vert \mathcal{H} \vert \vert \mathcal{G} \setminus \mathcal{H} \vert}{\sqrt{4\pi t}}  \sum_{n=0}^\infty \big(d_{\max}\big)^{n +1} \e^{-\frac{n^2\ell_{\min}^2{ + \dist(\partial \mathcal{H};\mV)^2}}{4t}};
\end{align}
{ note that, because we are summing over directed paths with combinatorial distance of at least $2$ connecting $\mathcal{H}$ and $\Graph \setminus \mathcal{H}$, every such path has a length of at least $n\ell_{\min} + \dist(\partial \mathcal{H};\mV)$, $n \in \mathbb{N}_0$. Moreover, as every such path hits atleast one vertex of $\Graph$, we get an additional factor of $d_{\max}$ on the right-hand side of \eqref{eq:first-estimate-second-term-cacciopoli}.}
Letting now $0 < t < \frac{\ell_{\min}}{2\log d_{\max}}$ one can deduce  { similarly} to the proof of \autoref{thm:two-term-small-time-asymptotic-heat-content} that
\begin{align*}
    \sqrt{\frac{\pi}{t}}\widetilde{\widetilde{\mathcal{Q}_{t}}}(\Graph;\mathcal H) &= { \frac{\vert \mathcal{H} \vert \vert \mathcal{G} \setminus \mathcal{H} \vert}{\sqrt{4\pi t}} d_{\max} \bigg( \e^{-\frac{\dist(\partial \mathcal{H};\mV)^2}{4t}} + \sum_{n=1}^\infty (d_{\max})^n \e^{-\frac{n^2\ell_{\min}^2 + \dist(\partial \mathcal{H};\mV)^2}{4t}} \bigg)} \\&\lesssim \frac{1}{2{ \sqrt{t}}} { \Bigg(\e^{-\frac{\dist(\partial \mathcal{H};\mV)^2}{4t}} + }\frac{d_{\max} \e^{-\frac{\ell_{\min}^2}{4t}}}{1-d_{\max} \e^{-\frac{\ell_{\min}^2}{2t}}} \Bigg) 
\end{align*}
for all $0 < t < \frac{\ell_{\min}}{2\log d_{\max}}$. Now the claim follows letting $t \rightarrow 0^+$ and combining this with \eqref{eq:trivial-boundary-part-caccioppoli}.
\end{proof}

\section{Heat content of subgraphs and further surgery principles}\label{sec:graph-surgery}

\subsection{Domination results}
Let \(\mathcal H\) be closed
 subset of \(\Graph\setminus \mVD\): we can think of \(\mathcal H\) as compact finite metric graph in its own right, with its vertex set given by \(\mV_\mathcal H=\partial \mathcal H\cup (\mV_\Graph\cap \mathcal H)\), see \autoref{fig:subgraph}. Following the usual notation, let \(\Delta^{\mathcal H; \mV_{\mathrm{D}, \mathcal{H}}}\) be the Laplacian on \(\mathcal H\) with Dirichlet conditions in the set 
 \[
 \mV_{\mathrm{D},\mathcal H}:=\partial \mathcal H\cup (\mV_{\mathrm{D}, \Graph}\cap \mathcal H),
 \]  
 and Kirchhoff and continuity conditions in the set 
 \[
 \mV_{\mathrm{N}, \mathcal H}=\mV_{\mathrm{N},\Graph}\cap \mathcal H;
 \] 
 we write $\mathcal{H} \subset \mathcal{G}$.

\begin{figure}[h]
      \centering
\begin{tikzpicture}[scale=0.6]
      \tikzset{enclosed/.style={draw, circle, inner sep=0pt, minimum size=.08cm, fill=black}}

      \node[enclosed] (A) at (0,2) {};
      \node[enclosed, blue] (B) at (1,4) {};
      \node[enclosed, blue] (C) at (3,3) {};
      \node[enclosed, white, label={below: $\Graph$}] (C'') at (3,2) {};
      \node[enclosed, blue] (D) at (4.5,6) {};
      \node[enclosed, blue] (E) at (7.5,6) {};
      \node[enclosed, blue] (G) at (-0.5,4) {};
      \node[enclosed] (H) at (-0.5,6) {};
      \node[enclosed] (J) at (7,2) {};

      \draw (B) edge[blue] node[below] {} (C) node[midway, above] (edge2) {};
      \draw (C) edge[blue] node[right] {\color{black} $\mathcal{H}$} (D) node[midway, above] (edge3) {};
      \draw (D) edge[blue] node[above] {} (E) node[midway, above] (edge3) {};
      \draw (B) edge[blue] node[below] {} (G) node[midway, above] (edge3) {};

      \node[enclosed, blue, label={right:$\partial \mathcal{H}$}] (A') at (7.25,4) {};
      \draw (A') edge[blue] node[below] {} (E) node[midway, above] (edge3) {};
      \draw (A') edge node[below] {} (J) node[midway, above] (edge3) {};
      \node[enclosed, blue,label={left:$\partial \mathcal{H}$}] (B') at (-0.5,4.5) {};
      \draw (H) edge node[below] {} (B') node[midway, above] (edge3) {};
      \draw (G) edge[blue] node[below] {} (B') node[midway, above] (edge3) {};
      \node[enclosed, blue, label={above:$\partial \mathcal{H}$}] (C') at (2.2,6) {};
      \draw (H) edge node[below] {} (C') node[midway, above] (edge3) {};
      \draw (C') edge[blue] node[below] {} (D) node[midway, above] (edge3) {};
      \node[enclosed, blue, label={right:{\color{black} $\partial \mathcal{H}$}}] (D') at (0.25,2.5) {};
      \draw (D) edge[blue] node[below] {} (B) node[midway, above] (edge3) {};
       \draw (E) edge[blue] node[below] {} (C) node[midway, above] (edge3) {};
      \draw (A) edge node[below] {} (D') node[midway, above] (edge3) {};
      \draw (D') edge[blue] node[below] {} (B) node[midway, above] (edge3) {};
     \end{tikzpicture}
     \vspace{-0.9cm}
     \caption{A graph $\mathcal{G}$ with a closed subgraph $\mathcal{H} \subset \Graph$ (blue).}\label{fig:subgraph}
    \end{figure}
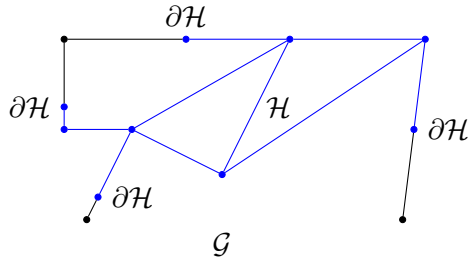
To begin with, we observe that $\Delta^{\mathcal H; \mV_{\mathrm{D}, \mathcal{H}}}$ generates a positive strongly continuous semigroup \break $(\e^{t\Delta^{\mathcal H; \mV_{\mathrm{D}, \mathcal{H}}}})_{t \geq 0}$, which is dominated by the original semigroup $(\e^{t\DeltaGD})_{t \geq 0}$ as the following proposition states.

\begin{proposition}\label{prop:domination-subgraphs}
Let $\Graph$ be a graph as in \autoref{ass:graph} and $\mathcal{H}$ be a closed subset of $\Graph \setminus \mVD$. Then
\begin{align}\label{eq:domination-heat-kernel}
	p^{\mathcal H; \mV_{\mathrm{D}, \mathcal{H}}}_t(x,y)\leq \ptGD(x,y)\qquad\hbox{for all \(t>0\) and \(x,y\in \mathcal H\)},
	\end{align}
and in particular
\begin{equation}\label{eq:monotonicity-heat-kernel-and-heat-content}
\heatcont(\mathcal{H};\mV_{\mathrm{D},\mathcal{H}}) \leq  \heatcont(\Graph;\mVD) \qquad \text{for all $t>0$.}
\end{equation}
\end{proposition}
\begin{proof}
By choice of the sets of Dirichlet and Neumann vertices, the inclusion
	\[H^1_0(\mathcal H;\mV_{\mathrm{D}, \mathcal H})\hookrightarrow H^1_0(\Graph;\mVD), \quad g\mapsto \bar{g}:=\begin{cases} g, & \text{on}~ \mathcal H,\\ 0, &\text{on}~\Graph\setminus \mathcal H, \end{cases}\]
is well-defined, more precisely -- \(H^1_0(\mathcal H;\mV_{\mathrm{D},\mathcal H})\) is an ideal in \(H^1_0(\Graph;\mVD)\) in the sense of~\cite[Definition~2.19]{Ouh05} -- with \(a_\Graph(\bar{g})=a_\mathcal H(g)\) and \(\|\bar{g}\|_{L^2(\Graph)}=\|g\|_{L^2(\mathcal H)}\) for \(g\in H^1_0(\mathcal H;\mV_{\mathrm{D},\mathcal H})\). Thus, by~\cite[Corollary~2.22]{Ouh05}, cf.\ also~\cite[Corollary~B.3]{StoVoi96}, this implies that $(\e^{t\DeltaGD})_{t\ge 0}$ dominates $(\e^{t\Delta^{\mathcal H; \mV_{\mathcal{H},\mathrm{D}}}})_{t\ge 0}$ and hence, by the fundamental theorem of calculus of variations, pointwise domination of heat kernel \(p_t^{\mathcal H;\mV_{\mathcal{H},\mathrm{D}}}(\cdot,\cdot) \) by \(\ptGD(\cdot, \cdot)\), i.e., \eqref{eq:domination-heat-kernel}. \eqref{eq:monotonicity-heat-kernel-and-heat-content} then readily follows. 
\end{proof}

The following can be proved in an analogous way to \autoref{prop:domination-subgraphs}, using the fact that 
$H_0^1(\mathcal{G};\mVD \cup \{ \mv_0 \})$ is an ideal in $H_0^1(\mathcal{G};\mVD)$ for every $\mv_0 \in \mVN$.

\begin{proposition}[Imposing Dirichlet conditions on additional vertices]\label{prop:additional-Dir}
Let $\mathcal{G}$ be a metric graph which satisfies \autoref{ass:graph} and let $\mv_0 \in \mVN$. Then
        
        \[p_t^{\Graph;\mVD \cup \{\mv_0\}}(x,y) \leq p_t^{\Graph;\mVD}(x,y)  \qquad \text{for all $x,y \in \Graph$ and all $t > 0$,}
        \] in particular, 
        \[
        \heatcont(\Graph;\mVD \cup \{ \mv_0 \}) \leq \heatcont(\Graph; \mVD) \qquad \text{for all $t>0$.}
        \]
\end{proposition}

{ 
\begin{corollary}\label{cor:lengthening-dir}
Let $\mathcal{G}$ be a metric graph which satisfies \autoref{ass:graph}, and let $\widetilde{\Graph}$ be a metric graph that arises from $\Graph$  lengthening by $\ell>0$ an edge one of whose endpoints lies in $\mVD$. Then 
\[
\heatcont(\widetilde{\Graph};\mVD)-\heatcont(\Graph;\mVD)\ge 
\frac{8 \ell}{\pi^2} \sum_{k=0}^\infty \e^{-t\big( \frac{\pi (2k+1)}{\ell} \big)^2} \frac{1}{(2k+1)^2}>0 \qquad \hbox{for all }t > 0.
\]
\end{corollary}
\begin{proof}
Let $\ell_\me$ the length of the relevant edge in $\Graph$, and $\widetilde{\ell}_\me>\ell_\me$ its lengthened version in $\widetilde{\Graph}$. Consider the heat content of the interval of length $\widetilde{\ell}_\me-\ell_\me=:\ell>0$ with Dirichlet conditions at both endpoints, which by \eqref{eq:qt-interval-elem} is given by
\[
\mathcal{Q}_t([0,\ell];\mVD) = \frac{8 \ell}{\pi^2} \sum_{k=0}^\infty \e^{-t\big( \frac{\pi (2k+1)}{\ell} \big)^2} \frac{1}{(2k+1)^2}, \qquad t > 0.
\]
Then ${\Graph}$ can be obtained from $\widetilde{\Graph}$ by first imposing a an additional Dirichlet condition on the edge $\me$ at distance $\widetilde{\ell}_\me-\ell_\me>0$ from the Dirichlet endpoint of $\widetilde{\Graph}$ (this raises the heat content by \autoref{prop:additional-Dir}), and then regarding $\widetilde{\Graph}$ as the disjoint union of $\Graph$ and an interval with two Dirichlet endpoints (whose heat content is larger than $\heatcont(\Graph;\mVD)$ by \autoref{rem:direct-sum}).
\end{proof}
}

\subsection{Operations not changing the volume}
We discuss the effect on the heat content of cutting the graph through a vertex of degree two. 
\begin{proposition}[Midpoint loop cut]\label{prop:loop-cut}
Let $\Graph$ satisfy \autoref{ass:graph} and let $\me \in \mE$ be a loop (i.e., both endpoints of $\me$ correspond to the same vertex). Moreover, let $\widetilde{\Graph}$ be the graph which arises from $\Graph$ by cutting the edge $\me$ through its midpoint (i.e., the point corresponding to $\frac{\ell_\me}{2}$) into two pendant edges both of length $\frac{\ell_\me}{2}$ (cf.\ \autoref{fig:lasso-edge-cut} below). Then
\[
\heatcont(\Graph;\mVD) = \heatcont(\widetilde{\Graph};\mVD) \qquad \text{for all $t>0$.}
\]
\end{proposition}
\begin{figure}[h]
        \begin{tikzpicture}[scale=0.6]
      \tikzset{enclosed/.style={draw, circle, inner sep=0pt, minimum size=.1cm, fill=black}, every loop/.style={}}
      \node[enclosed, label = {below left: $\mv$}] (Z) at (0,2) {};
      \node[enclosed, label = {left: $\Graph \qquad$}] (A) at (-4,2) {};
      \node[enclosed, white] (X) at (5.5,2) {};
      \node[enclosed, white] (Y) at (7.5,2) {};
      
	   \node[enclosed, label = {below left: $\widetilde{\mv}$}] (Z') at (13,2) {};
      \node[enclosed] (A') at (9,2) {};
      \node[enclosed, label = {above right: $\mv_1$}] (vs1) at (17,3) {};   
      \node[enclosed, label = {below right: $\mv_2$}] (vs2) at (17,1) {};   
      \node[enclosed, label = {right: $\qquad \widetilde{\Graph}$}, white] (invis) at (17,2) {}; 
      
      \draw[-] (Z) edge node[above] {} (A) node[midway, above] (edge1) {};
      \draw[->, ultra thick] (X) edge node[above] {} (Y) node[midway, above] (edge1) {};
      \draw[-] (Z) arc [start angle=-180, end angle=180,
                  x radius=2cm, 
                  y radius=10mm] node[below] [pos=0.25] {$\me$};
                  
      \node[enclosed, label = {above right: $\mv_0$}] (vs) at (4,2) {};
      \draw[-] (Z') edge node[above] {} (A') node[midway, above] (edge1) {};
      \draw[-] (Z') edge node[above] {} (vs1) node[midway, above] (edge1) {};
      \draw[-] (Z') edge node[above] {} (vs2) node[midway, above] (edge1) {};
      
      \draw [decorate,decoration={brace,amplitude=10pt},xshift=0pt,yshift=0pt, dotted]
    (13.075,2) -- (16.925,3) node [black,midway,xshift=-0.2cm,yshift=0.65cm] 
    {\small \rotatebox{18}{$\frac{\ell_\me}{2}$}};
    \draw [decorate,decoration={brace,amplitude=10pt},xshift=0pt,yshift=0pt, dotted]
    (16.925,1) -- (13.075,2) node [black,midway,xshift=-0.2cm,yshift=-0.65cm] 
    {\small \rotatebox{-12}{$\frac{\ell_\me}{2}$}};
     \end{tikzpicture}
     \caption{The lasso graph $\mathcal{\Graph}$ from \autoref{ex:lasso} (left) and a  $3$-star graph $\widetilde{\Graph}$ arising from a midpoint loop cut in the sense of \autoref{prop:loop-cut} through the edge $\me$ at $\mv_0$.} \label{fig:lasso-edge-cut}
     \end{figure}
     
\begin{proof}
Let $\me^+,\me^-$ denote the two pendant edges of $\widetilde{\Graph}$ arising through the midpoint cut through $\me$ and denote by $\mv^+,\mv^-$ as the corresponding endpoints of $\me^+,\me^-$ having degree $1$. Moreover, let $\mv$ be the vertex in $\Graph$ which connects the loop to a different edge of $\Graph$, and $\widetilde{\mv}$ be its corresponding counterpart in $\widetilde{\Graph}$.

Clearly, the first two addends appearing in \eqref{eq:heat-content-formula-bif-mug} are the same for $\Graph$ and $\widetilde{\Graph}$. Thus, it suffices to show that $\mathcal{L}_t(\Graph;\mVD) = \mathcal{L}_t(\widetilde{\Graph};\mVD)$ (cf.\ \eqref{eq:ltgvd}) for every $t>0$: to see this, for each directed path $\vec{p} \in \mathcal{P}_\mVD(\Graph)$ we have to assign a corresponding directed path $\vec{q} \in \mathcal{P}_\mVD(\widetilde{\Graph})$ with the same length as well as scattering coefficient, in a way that each path in $\mathcal{P}_\mVD(\widetilde{\Graph})$ can be constructed by a uniquely determined path in $\mathcal{P}_\mVD(\Graph)$, and vice versa. 

Therefore, each time a directed path $\vec{p} \in \mathcal{P}_\mVD(\Graph)$ enters the edge $\me$ coming from a \emph{different} edge, we assign two directions for the first transfer of the loop $\me$: for the clockwise and counter-clockwise traverse, respectively. Whenever $\vec{p}$ enters the loop $\me$ in clockwise direction, the corresponding path $\vec{q}$ shall traverse the edge $\me^+$ and reflect at $\mv^+$ back to $\widetilde{\mv}$ afterwards; otherwise if $\vec{p}$ enters the loop $\me$ in counter-clockwise direction it shall traverse $\me^-$ and reflect at $\mv^-$ back to $\widetilde{\mv}$ instead, cf.\ also \autoref{fig:lasso-edge-cut-contructing-q}.

\begin{figure}[h]
        \begin{tikzpicture}[scale=0.6]
      \tikzset{enclosed/.style={draw, circle, inner sep=0pt, minimum size=.1cm, fill=black}, every loop/.style={}}
      \node[enclosed, label = {below left: $\mv$}] (Z) at (0,2) {};
      \node[enclosed, white] (W1) at (3,3.5) {};
      \node[enclosed, white] (W2) at (3,0.5) {};
      \node[enclosed, label = {left: $\Graph \qquad$}] (A) at (-4,2) {};
      \node[enclosed, white] (X) at (5.5,2) {};
      \node[enclosed, white] (Y) at (7.5,2) {};
      
	   \node[enclosed, label = {below left: $\widetilde{\mv}$}] (Z') at (13,2) {};
      \node[enclosed] (A') at (9,2) {};
      \node[enclosed, label = {above right: $\mv^+$}] (vs1) at (17,3) {};   
      \node[enclosed, label = {below right: $\mv^-$}] (vs2) at (17,1) {};   
      \node[enclosed, label = {right: $\qquad \widetilde{\Graph}$}, white] (invis) at (17,2) {}; 
      
      \draw[-] (Z) edge node[above] {} (A) node[midway, above] (edge1) {};
      \draw[->, ultra thick] (X) edge node[above] {} (Y) node[midway, above] (edge1) {};
      \draw[-] (Z) arc [start angle=-180, end angle=180,
                  x radius=2cm, 
                  y radius=10mm] node[below] [pos=0.25] {$\me$};
                  
      \draw[-] (Z') edge node[above] {} (A') node[midway, above] (edge1) {};
      \draw[-] (Z') edge node[above] {$\me^+$} (vs1) node[midway, above] (edge1) {};
      \draw[-] (Z') edge node[below] {$\me^-$} (vs2) node[midway, above] (edge1) {};
      
      \draw[->, thick, blue] (Z) edge[bend left = 60] node[left] {\textcolor{black}{\small clockwise \:}} (W1) node[midway, above] (edge1) {};
      \draw[->, thick, OliveGreen] (Z) edge[bend right = 60] node[left] {\textcolor{black}{\small counterclockwise \:}} (W2) node[midway, above] (edge1) {};
      \draw[->, thick, blue] (Z') edge[bend left = 45] node[above] {} (vs1) node[midway, above] (edge1) {};
      \draw[->, thick, OliveGreen] (Z') edge[bend right = 45] node[above] {} (vs2) node[midway, above] (edge1) {};
      \draw[->, thick, blue] (vs1) edge[bend right = 75] node[above] {} (Z') node[midway, above] (edge1) {};
      \draw[->, thick, OliveGreen] (vs2) edge[bend left = 75] node[above] {} (Z') node[midway, above] (edge1) {};
     \end{tikzpicture}
     \caption{The lasso graph $\mathcal{\Graph}$  (left) and the $3$-star graph $\widetilde{\Graph}$ arising from a midpoint loop cut in the sense of \autoref{prop:loop-cut} through the edge $\me$ as in \autoref{fig:lasso-edge-cut}. The blue arrow (left) marks the clockwise whereas the green arrow (left) marks the counter-clockwise direction corresponding to reflections (blue and green) at $\mv^+$ and $\mv^-$ (right), respectively.} \label{fig:lasso-edge-cut-contructing-q}
     \end{figure}
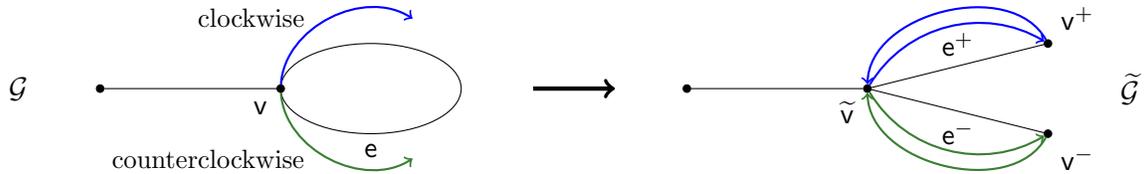

After this, $\vec{q}$ shall follow the same bonds in $\widetilde{\Graph}$ as $\vec{p}$ where each transfer (resp.\ reflection) at $\mv$ corresponds to a transfer (resp.\ reflection) of $\vec{q}$ at $\widetilde{\mv}$ (see \autoref{fig:lasso-edge-cut-contructing-q-2}) until $\vec{p}$ enters the loop $\me$ again: in this case, the above procedure shall be repeated to determine the upcoming bonds of $\vec{q}$. 
\begin{figure}[h]
        \begin{tikzpicture}[scale=0.6]
      \tikzset{enclosed/.style={draw, circle, inner sep=0pt, minimum size=.1cm, fill=black}, every loop/.style={}}
      \node[enclosed, label = {below left: $\mv$}] (Z) at (0,2) {};
      \node[enclosed, white] (W1) at (4.4,2.1) {};
      \node[enclosed, white, label = {right: $\vec{\me}_2$}] (W) at (4.35,2) {};
      \node[enclosed, white, label = {right: $\vec{\me}_3$}] (W') at (5.5,2) {};
      
      \node[enclosed, white] (W1') at (5.5,2.1) {};
      \node[enclosed, white] (W2) at (4.4,1.9) {};
      \node[enclosed, white] (W2') at (5.5,1.9) {};
      \node[enclosed] (A) at (-4,2) {};
      \node[enclosed, white] (A1) at (-1.9,3.75) {};
      \node[enclosed, white] (A2) at (-2.1,3.75) {};
      \node[enclosed, white, label = {above: $\vec{\me}_4$}] (A3) at (-2,3.75) {};
      \node[enclosed, white] (X) at (7.5,2) {};
      \node[enclosed, white] (Y) at (9.5,2) {};
      
	   \node[enclosed, label = {below left: $\widetilde{\mv}$}] (Z') at (15,2) {};
      \node[enclosed] (A') at (11,2) {};
      \node[enclosed, white] (A1') at (12.9,3.75) {};
      \node[enclosed, white] (A2') at (13.1,3.75) {};
      \node[enclosed, white, label = {above: $\vec{\me}_4$}] (A3') at (13,3.75) {};
      \node[enclosed, label = {above right: $\mv^+$}] (vs1) at (19,3) {};   
      \node[enclosed, label = {below right: $\mv^-$}] (vs2) at (19,1) {};   
      
      \draw[-] (Z) edge node[above] {} (A) node[midway, above] (edge1) {};
      \draw[->, ultra thick] (X) edge node[above] {} (Y) node[midway, above] (edge1) {};
      \draw[-] (Z) arc [start angle=-180, end angle=180,
                  x radius=2cm, 
                  y radius=10mm] node[above] [pos=0.25] {$\me$};
                  
      \draw[-] (Z') edge node[above] {} (A') node[midway, above] (edge1) {};
      \draw[-] (Z') edge node[below] {} (vs1) node[midway, above] (edge1) {};
      \draw[-] (Z') edge node[above] {} (vs2) node[midway, above] (edge1) {};

	\draw[->, thick, blue] (A) edge[bend left] node[above] {\textcolor{black}{$\vec{\me}_1$}} (Z) node[midway, above] (edge1) {};      
      \draw[-, thick, blue] (Z) edge[bend left = 90] node[above] {} (W2) node[midway, above] (edge1) {};
      \draw[-, thick, blue] (Z) edge[bend left = 90] node[above] {} (W2') node[midway, above] (edge1) {};
      \draw[->, thick, blue] (W1) edge[bend left = 90] node[above] {} (Z) node[midway, above] (edge1) {};
      \draw[->, thick, blue] (W1') edge[bend left = 90] node[above] {} (Z) node[midway, above] (edge1) {};
      \draw[-, thick, blue] (Z) edge[bend right = 40] node[above] {} (A2) node[midway, above] (edge1) {};
      \draw[->, thick, blue] (A1) edge[bend right = 40] node[above] {} (A) node[midway, above] (edge1) {};
      \draw[->, thick, blue] (Z') edge[bend left = 25] node[above right] {\tiny \textcolor{black}{$\vec{\me}_2^1$}} (vs1) node[midway, above] (edge1) {};
      \draw[->, thick, blue] (Z') edge[bend right = 25] node[below right] {\tiny \textcolor{black}{$\vec{\me}_3^1$}} (vs2) node[midway, above] (edge1) {};
      \draw[->, thick, blue] (vs1) edge[bend right = 75] node[above] {\textcolor{black}{$\vec{\me}_2^2$}} (Z') node[midway, above] (edge1) {};
      \draw[->, thick, blue] (vs2) edge[bend left = 75] node[below] {\textcolor{black}{$\vec{\me}_3^2$}} (Z') node[midway, above] (edge1) {};
      \draw[->, thick, blue] (A') edge[bend left] node[above] {\textcolor{black}{$\vec{\me}_1$}} (Z') node[midway, above] (edge1) {};
      \draw[-, thick, blue] (Z') edge[bend right = 40] node[above] {} (A1') node[midway, above] (edge1) {};
      \draw[->, thick, blue] (A2') edge[bend right = 40] node[above] {} (A') node[midway, above] (edge1) {};
     \end{tikzpicture}
     \caption{A directed path $\vec{p} = \big( \mv_-(\vec{p}), \vec{\me}_1,\vec{\me}_2,\vec{\me}_3,\vec{\me}_4,\mv_+(\vec{p}) \big)$ on the graph $\Graph$ (left) entering the loop $\me$ with the bond $\vec{\me}_2$ and continuing with a transfer through $\mv$ before leaving the loop $\me$ with $\vec{\me}_4$. The corresponding directed path $\vec{q}$ is thus given by $\vec{q}= \big( \mv_-(\vec{q}), \vec{\me}_1,\vec{\me}_2^1,\vec{\me}_2^2,\vec{\me}_3^1,\vec{\me}_3^2,\vec{\me}_4,\mv_+(\vec{q}) \big)$ where the bond $\vec{\me}_2$ (entering of loop $\me$ in \emph{clockwise} direction) is replaced by two bonds $\vec{\me}_2^1, \vec{\me}_2^2$ such that $\vec{q}$ reflects at $\mv^+$ and the consecutive bond $\vec{\me}_3$ (a transfer at $\mv$) is replaced by $\vec{\me}_3^1, \vec{\me}_3^2$ such that $\vec{q}$ reflects now at $\mv_-$.} \label{fig:lasso-edge-cut-contructing-q-2}
     \end{figure}
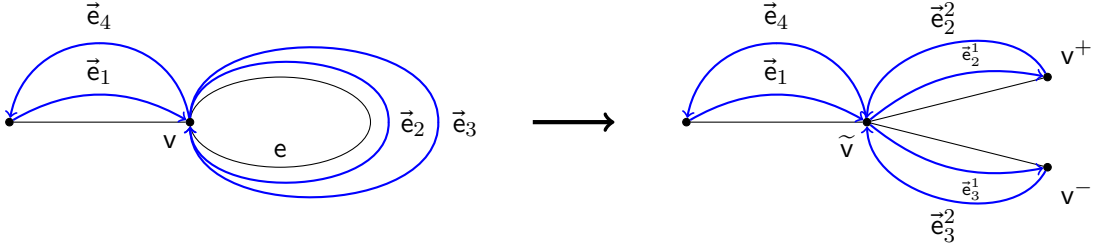
     
This construction indeed leaves the length and scattering coefficient invariant as $\deg_\Graph(\mv) = \deg_{\widetilde{\Graph}}(\widetilde{\mv})$ and each bond of $\vec{p}$ traversing the loop $\me$ (and thus contributes a length of $\ell_\me$ to $\ell(\vec{p})$) is replaced by two bonds traversing $\me^+$ (resp.\ $\me^-$) twice, thus leading to a contributing length of $2\ell_{\me^+} = \ell_\me$ (resp.\ $2\ell_{\me^+} = \ell_\me$) to $\ell(\vec{q})$.     
     
     If now $m_\me(\vec{p}) \in \mathbb{N}_0$ denotes the number of enters of the loop $\me$ for a directed path $\vec{p}$ and $\vec{q}_{\vec{p}, m(\vec{p})}$ the corresponding directed path in $\widetilde{\Graph}$ constructed as above, we have bijection between the sets 
\[
\mathcal{P}(\Graph) = \bigcup_{m=0}^\infty \{ \vec{p} \in \mathcal{P}_\mVD(\Graph) \: : \: m_\me(\vec{p}) = m \}
\]
and $\mathcal{P}_\mVD(\widetilde \Graph)$ -- mapping each $\vec{p}$ to $\vec{q}_{\vec{p};m(\vec{p})}$ -- such that $\ell(\vec{p}) = \ell(\vec{q}_{\vec{p};m(\vec{p})})$ and $\alpha(\vec{p}) = \alpha(\vec{q}_{\vec{p};m(\vec{p})})$ for every $\vec{p} \in \mathcal{P}_\mVD(\Graph)$. This yields
\begin{align*}
\mathfrak{L}_t(\Graph; \mVD) &= \sum_{\vec{p} \in \mathcal{P}_\mVD(\Graph)} \alpha(\vec{p}) H \bigg( \frac{\ell(\vec{p})}{2\sqrt{t}} \bigg) = \sum_{m=0}^\infty \sum_{\substack{\vec{p} \in \mathcal{P}_\mVD(\Graph) \\ m_\me(\vec p) = m}} \alpha(\vec{p}) H \bigg( \frac{\ell(\vec{p})}{2\sqrt{t}} \bigg) \\&= \sum_{m=0}^\infty \sum_{\substack{\vec{p} \in \mathcal{P}_\mVD(\Graph) \\ m_\me(\vec p) = m}} \alpha(\vec{q}_{\vec{p}, m(\vec{p})}) H \bigg( \frac{\ell(\vec{q}_{\vec{p}, m(\vec{p})})}{2\sqrt{t}} \bigg) = \sum_{\vec{q} \in \mathcal{P}_\mVD(\widetilde \Graph)} \alpha(\vec{q}) H \bigg( \frac{\ell(\vec{q})}{2\sqrt{t}} \bigg) = \mathfrak{L}_t(\widetilde \Graph; \mVD)
\end{align*}
for all $t>0$, finally proving the claim.
\end{proof}

\subsection{Operations increasing the volume}
Let us now discuss some surgery principles for the heat content involving certain geometric manipulations of the underlying graph.

We start by studying the operation of attaching pendant graphs in the sense of~\cite[Definition~3.9]{BerKenKur19}. 

We already know from \autoref{cor:lengthening-dir} that lengthening a pendant edge $\me$ raises the heat content, as long as one endpoint of $\me$ lies in $\mVD$. Indeed, the same is true if the degree one vertex lies in $\mVN$. The following shows this and more: it is an immediate consequence of a Feynman--Kac-type formula for metric graphs, see \cite[Proposition~3.2]{BifTau25} and we therefore omit the easy proof.

\begin{proposition}[Attaching graphs]\label{prop:attaching-surgery}
Let $\Graph$ be a metric graph which satisfies \autoref{ass:graph} and let $\widetilde{\Graph}$ be the graph which is formed by attaching a pendant graph $\mathcal{H}$ at a vertex $\mv_0 \in \mVN$, then $\heatcont(\Graph;\mVD) \leq \heatcont(\widetilde{\Graph};\mVD)$ for all $t >0$.
\end{proposition}

\begin{exa}\label{exa:larger-loop}
Let us present a simple application of \autoref{prop:attaching-surgery}. Let $\Graph$ be a lasso or, more generally, any metric graph that contains a pendant loop. If $\widetilde{\Graph}$ is a metric graph arising from $\Graph$ lengthening the loop, then 
\begin{equation}\label{eq:larger-loop}
\heatcont(\Graph;\mVD)\le \heatcont(\widetilde{\Graph};\mVD)\qquad\hbox{ for all }t>0.
\end{equation}
Indeed, let $\Graph^{*}$ and $\widetilde{\Graph}^{*}$ the metric graphs that arise from $\Graph$ and $\widetilde{\Graph}$, respectively, by a midpoint loop cut, thus generating two new edges $\me_1,\me_2$ and $\widetilde{\me}_1,\widetilde{\me}_2$, respectively: we know from \autoref{prop:loop-cut} that this operation does not change the heat content at any $t>0$, i.e., 
\[
\heatcont(\Graph;\mVD)=\heatcont(\Graph^*;\mVD)\quad\hbox{ and }
\heatcont(\widetilde{\Graph};\mVD)=\heatcont(\widetilde{\Graph}^*;\mVD)\qquad\hbox{for all }t>0.
\] 
Then $\widetilde{\Graph}^{*}$ arises from $\Graph^{*}$ by attaching a pendant interval at the Neumann endpoint of both $\me_1,\me_2$: by \autoref{prop:attaching-surgery}, \eqref{eq:larger-loop} follows immediately; see also \autoref{cor:heat-length-s} below for a comparable result.
\end{exa}

\begin{definition}[Mirroring at vertices]\label{defi:basic-surgery}
Let $\Graph$ be a metric graph which satisfies \autoref{ass:graph} and let $\mV_0 \subset \Graph \setminus \mVD$ be a finite subset. Consider the (disconnected) metric graph consisting of the disjoint union of $m$ copies of $\Graph$ -- let us denote them by $\Graph_1,\ldots,\Graph_m$ -- each of which contains a copy of elements of $\mV_0$ -- let us denote them by $\mv^{(1)},\ldots,\mv^{(m)}$ for every $\mv\in \mV_0$.
For each $\mv\in \mV_0$, glue the vertices  $\mv^{(1)},\ldots,\mv^{(m)}$ to form a new vertex $\widetilde{\mv}$: we denote by $\Graph^{(m)}(\mV_0)$ the metric graph thus arising, and we call it the \textit{$m$-fold mirrored} metric graph with respect to the reflection set $\mV_0$. Moreover, for any subset $\mW \subset \mV \setminus \mV_0$, we denote with $\mW^{(m)}$ the $m$ copies corresponding $\mW$ in the $m$-fold mirrored graph (cf.\ \autoref{fig:m-fold-mirroring} below).

\end{definition} 
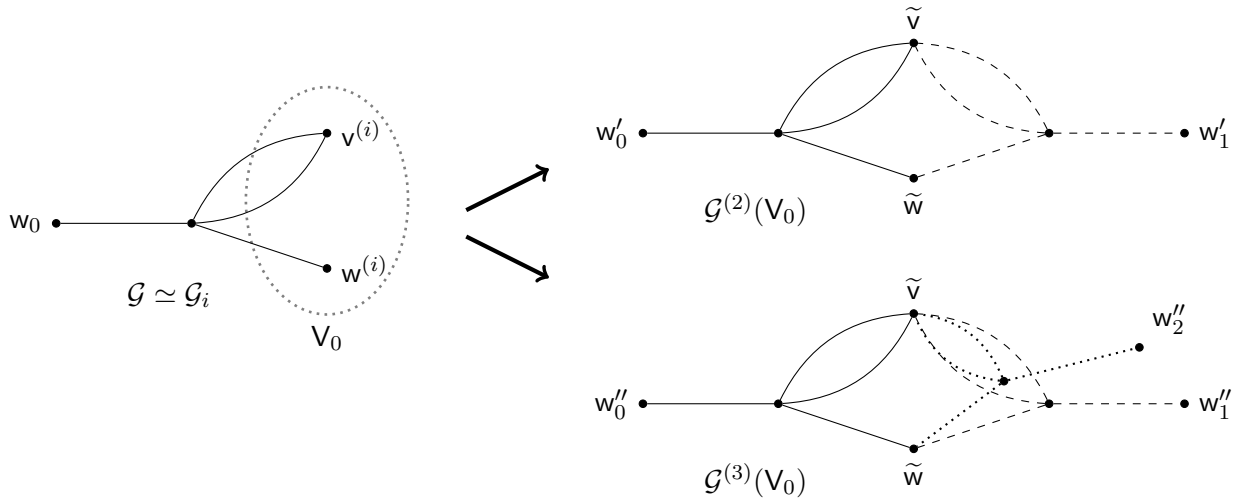
\begin{figure}[h]
\begin{tikzpicture}[scale=0.60]
      \tikzset{enclosed/.style={draw, circle, inner sep=0pt, minimum size=.10cm, fill=black}, every loop/.style={}, every fit/.style={ellipse,draw, dotted,inner sep=3.2pt,text width=1.3cm, line width=1pt}}

      \node[enclosed] (Z) at (-1,0) {};
      \node[enclosed, label = {right: $\mv^{(i)}$}] (A) at (2,2) {};
      \node[enclosed, label = {left: $\mw_0$}] (B) at (-4,0) {};
      \node[enclosed, label = {right: $\mw^{(i)}$}] (C) at (2,-1) {};
      \node[enclosed, white] (E) at (5,0.25) {};
      \node[enclosed, white] (F) at (7,1.25) {};
      \node[enclosed, white] (E') at (5,-0.25) {};
      \node[enclosed, white] (F') at (7,-1.25) {};
      
      \node[enclosed] (Z') at (12,2) {};
      \node[enclosed, label = {above: $\widetilde{\mv}$}] (A') at (15,4) {};
      \node[enclosed, label = {left: $\mw_0'$}] (B') at (9,2) {};
      \node[enclosed, label = {below: $\widetilde{\mw}$}] (C') at (15,1) {};
      \node[enclosed, label = {right: $\mw_1'$}] (B'refl) at (21,2) {};
      \node[enclosed] (Z'refl) at (18,2) {};
      
      \node[enclosed] (Z'') at (12,-4) {};
      \node[enclosed, label = {above: $\widetilde{\mv}$}] (A'') at (15,-2) {};
      \node[enclosed, label = {left: $\mw_0''$}] (B'') at (9,-4) {};
      \node[enclosed, label = {below: $\widetilde{\mw}$}] (C'') at (15,-5) {};
      \node[enclosed, label = {right: $\mw_1''$}] (B''refl) at (21,-4) {};
      \node[enclosed] (Z''refl) at (18,-4) {};
      \node[enclosed, label = {above right: $\mw_2''$}] (B''refl2) at (20,-2.75) {};
      \node[enclosed] (Z''refl2) at (17,-3.5) {};
     
      \node[enclosed, white, label={below: $\Graph \simeq \mathcal{G}_i$}] (G) at (-1.5,-1) {};
      \node[enclosed, white, label={below: $\Graph^{(2)}(\mV_0)$}] (G'') at (11.5,1) {};
      \node[enclosed, white, label={below: $\Graph^{(3)}(\mV_0)$}] (G'') at (11.5,-5) {};

\node [gray,fit=(A) (C),label=below:\textcolor{black}{$\mV_0$}] {};

      \draw (Z) edge[bend left] node[above] {} (A) node[midway, above] (edge1) {};
      \draw (Z) edge[bend right] node[above] {} (A) node[midway, above] (edge1) {};
      \draw (Z) edge node[above] {} (B) node[midway, above] (edge2) {};
      \draw (Z) edge node[above] {} (C) node[midway, above] (edge3) {};
      
      \draw (Z') edge[bend left] node[above] {} (A') node[midway, above] (edge1) {};
      \draw (Z') edge[bend right] node[above] {} (A') node[midway, above] (edge1) {};
      \draw (Z') edge node[above] {} (B') node[midway, above] (edge2) {};
      \draw (Z') edge node[above] {} (C') node[midway, above] (edge3) {};
      \draw[dashed] (Z'refl) edge[bend left] node[above] {} (A') node[midway, above] (edge1) {};
      \draw[dashed] (Z'refl) edge[bend right] node[above] {} (A') node[midway, above] (edge1) {};
      \draw[dashed] (Z'refl) edge node[above] {} (C') node[midway, above] (edge3) {};
      \draw[dashed] (Z'refl) edge node[above] {} (B'refl) node[midway, above] (edge2) {};
      
      \draw (Z'') edge[bend left] node[above] {} (A'') node[midway, above] (edge1) {};
      \draw (Z'') edge[bend right] node[above] {} (A'') node[midway, above] (edge1) {};
      \draw (Z'') edge node[above] {} (B'') node[midway, above] (edge2) {};
      \draw (Z'') edge node[above] {} (C'') node[midway, above] (edge3) {};
      \draw[dashed] (Z''refl) edge[bend left] node[above] {} (A'') node[midway, above] (edge1) {};
      \draw[dashed] (Z''refl) edge[bend right] node[above] {} (A'') node[midway, above] (edge1) {};
      \draw[dashed] (Z''refl) edge node[above] {} (C'') node[midway, above] (edge3) {};
      \draw[dashed] (Z''refl) edge node[above] {} (B''refl) node[midway, above] (edge2) {};
      \draw[thick, dotted] (Z''refl2) edge[bend left] node[above] {} (A'') node[midway, above] (edge1) {};
      \draw[thick, dotted] (Z''refl2) edge[bend right] node[above] {} (A'') node[midway, above] (edge1) {};
      \draw[thick, dotted] (Z''refl2) edge node[above] {} (C'') node[midway, above] (edge3) {};
      \draw[thick, dotted] (Z''refl2) edge node[above] {} (B''refl2) node[midway, above] (edge2) {};
 
      \draw[->, ultra thick] (E) edge node[above] {} (F) node[midway, above] (edge5) {};
      \draw[->, ultra thick] (E') edge node[above] {} (F') node[midway, above] (edge5') {};

     \end{tikzpicture}
     \caption{The corresponding $2$- (right above) and $3$-fold (right below) mirrored graphs $\Graph^{(2)}(\mV_0)$ and $\Graph^{(3)}(\mV_0)$ arising from $i=2,3$ copies of a graph $\Graph$ (left). For $\mW := \{ \mw_0 \}$ one has that $\mW^{(2)} = \{\mw_0',\mw_1' \}$ and $\mW^{(3)} = \{\mw_0'',\mw_1'',\mw_2'' \}$.}\label{fig:m-fold-mirroring}
     \end{figure}

\begin{theorem}\label{thm:mirroring-heat-content}
Let $\Graph$ be a metric graph which satisfies \autoref{ass:graph} and let $\mV_0 \subset \mV_\mathrm{N}$. For every $m \in \mathbb{N}$ there holds 
\begin{equation}\label{eq:heat-mirrored}
\heatcont(\Graph^{(m)}(\mV_0);(\mVD)^{(m)}) = m\heatcont(\Graph;\mVD)\qquad \hbox{for all }t>0.
\end{equation}
\end{theorem}
In particular, the left hand side of \eqref{eq:heat-mirrored} does not depend on the number of elements in $\mV_0$, since neither does the right hand side.

The proof will be based on on another notion which is based on the notations introduced in Section \ref{sec:comb-paths-g}: to this end, let $m \in \mathbb{N}$, let $\Graph$ be a graph, and $\mV_0 \subset \Graph \setminus \mVD$ be a finite subset. For a directed path $\vec{p} \in \mathcal{P}(\Graph)$ with $n:=\#_{\mV_0} \vec{p} \in \mathbb{N}_0$ we denote with 
\[
\vec{p}(m_1,\dots,m_j,\dots,m_{n+1}), \qquad \text{where $m_i \in \{1,\dots,m \}$ and $i \in \{1,\dots,n+1 \}$}, 
\]
the directed path in $\mathcal{P}(\Graph^{(m)}(\mV_0))$ that follows the same bonds as $\vec{p}$
\begin{itemize}
\item[•] between the $(j-1)$-th and the $j$-th hit of $\mV_0$, if $j=2,\dots,n$,
\item[•] before the first hit of $\mV_0$, if $j=1$,
\item[•] after the last hit of $\mV_0$, if $j=n+1$
\end{itemize}
 in the \emph{$k$-th} copy of $\Graph$ if $m_j=k$; see also \autoref{fig:expl-m-fold-mirror-2} and \autoref{fig:expl-m-fold-mirror-3} below.

\begin{figure}[h]
        \begin{tikzpicture}[scale=0.57]
      \tikzset{enclosed/.style={draw, circle, inner sep=0pt, minimum size=.1cm, fill=black}, every loop/.style={}}
      \node[enclosed] (Z) at (0,2) {};
      \node[enclosed] (A) at (-4,2) {};
      \node[enclosed, label = {right: $\mv_0$}] (W) at (4,2) {};
      \node[enclosed, white, label = {above: $\vec{\me}_4$}] (W'up) at (2,4.5) {};
      \node[enclosed, white] (W'upr) at (2.1,4.5) {};
      \node[enclosed, white] (W'upl) at (1.9,4.5) {};
      \node[enclosed, white, label = {below: $\vec{\me}_5$}] (W'low) at (2,-0.5) {};
      \node[enclosed, white] (W'lowr) at (2.1,-0.5) {};
      \node[enclosed, white] (W'lowl) at (1.9,-0.5) {};
      \node[enclosed, white] (X) at (5.75,2) {};
      \node[enclosed, white] (Y) at (7.75,2) {};
      
      \node[enclosed] (Z') at (12.75,2) {};
      \node[enclosed] (Z'refl) at (20.75,2) {};
      \node[enclosed] (A') at (8.75,2) {};
      \node[enclosed] (A'refl) at (24.75,2) {};
      \node[enclosed] (W') at (16.75,2) {};
      
      \draw[-] (Z) edge node[above] {} (A) node[midway, above] (edge1) {};
      \draw[-] (Z) arc [start angle=-180, end angle=180,
                  x radius=2cm, 
                  y radius=7mm] node[right] [pos=0.5] {};
     \draw[->, thick, blue] (A) edge[bend left] node[above] {\textcolor{black}{$\vec{\me}_1$}} (Z) node[midway, above] (edge1) {};
     \draw[-, thick, blue] (Z) edge[bend left = 50] node[below] {} (W'upr) node[midway, above] (edge1) {};
     \draw[->, thick, blue] (W'upl) edge[bend left = 50] node[below] {} (W) node[midway, above] (edge1) {};
      \draw[-, thick, blue] (W) edge[bend left = 50] node[above] {} (W'lowl) node[midway, above] (edge1) {};
       \draw[->, thick, blue] (W'lowr) edge[bend left = 50] node[above] {} (Z) node[midway, above] (edge1) {};
     \draw[->, thick, blue] (Z) edge[in=80, out=100] node[above] {\textcolor{black}{$\vec{\me}_2$}} (W) node[midway, above] (edge1) {};
     \draw[->, thick, blue] (W) edge[in=-100, out=-80] node[below] {\textcolor{black}{$\vec{\me}_3$}} (Z) node[midway, above] (edge1) {};
     \draw[->, thick, blue] (Z) edge[bend left] node[below] {\textcolor{black}{$\vec{\me}_6$}} (A) node[midway, above] (edge1) {};
     \draw[->, ultra thick] (X) edge node[below] {} (Y) node[midway, above] (edge1) {};
     
     \draw[-] (Z') edge node[above] {} (A') node[midway, above] (edge1) {};
     \draw[-] (Z'refl) edge node[above] {} (A'refl) node[midway, above] (edge1) {};
     \draw[-] (Z') arc [start angle=-180, end angle=180,
                  x radius=2cm, 
                  y radius=7mm] node[right] [pos=0.5] {};
      \draw[-] (Z'refl) arc [start angle=0, end angle=360,
                  x radius=2cm, 
                  y radius=7mm] node[right] [pos=0.5] {};
                  
\draw[->, thick, blue] (A') edge[bend left] node[above] {\textcolor{black}{$\vec{\me}_1$}} (Z') node[midway, above] (edge1) {};
     \draw[->, thick, blue] (Z') edge[in=80, out=100] node[above] {\textcolor{black}{$\vec{\me}_2$}} (W') node[midway, above] (edge1) {};
     \draw[->, thick, red] (W') edge[in=80, out=100] node[above] {\textcolor{black}{$\vec{\me}_3$}} (Z'refl) node[midway, above] (edge1) {};
     \draw[->, thick, red] (Z'refl) edge[in=-100, out=-80] node[below] {\textcolor{black}{$\vec{\me}_4$}} (W') node[midway, above] (edge1) {};
     \draw[->, thick, blue] (W') edge[in=-100, out=-80] node[below] {\textcolor{black}{$\vec{\me}_5$}} (Z') node[midway, above] (edge1) {};
     \draw[->, thick, blue] (Z') edge[bend left] node[below] {\textcolor{black}{$\vec{\me}_6$}} (A') node[midway, above] (edge1) {};
     
     \draw [decorate,decoration={brace,amplitude=10pt},xshift=0pt,yshift=0pt, thick]
    (16.7,-0.5) -- (8.8,-0.5) node [black,midway,xshift=0cm,yshift=-0.6cm] 
    {first copy of $\Graph$};
     \draw [decorate,decoration={brace,amplitude=10pt},xshift=0pt,yshift=0pt, thick]
    (24.7,-0.5) -- (16.8,-0.5) node [black,midway,xshift=0cm,yshift=-0.6cm] 
    {second copy of $\Graph$};
     \end{tikzpicture}
     \caption{A lasso graph $\Graph$ together with an vertex $\mv_0$ and a directed path $\vec{p} = \big(\mv_-(\vec{p}), \vec{\me}_1, \vec{\me}_2, \vec{\me}_3, \vec{\me}_4, \vec{\me}_5, \vec{\me}_6, \mv_+(\vec{p})\big) \in \mathcal{P}(\Graph)$ with $\#_{\mv_0} \vec{p} = 2)$ (left), and its corresponding $2$-fold mirrored graph (at $\mV_0 := \{ \mv_0 \}$) $\Graph^{(2)}(\mV_0)$ together with the path $\vec{p}(1,2,1)$ whose bonds \emph{before} the first and \emph{after} the second (i.e., the last) hit of $\mv_0$ belong to the first copy of $\Graph$ (blue) and the bonds \emph{between} the first and the second hit belong to the second copy of $\Graph$ (red).} \label{fig:expl-m-fold-mirror-2}
     \end{figure}
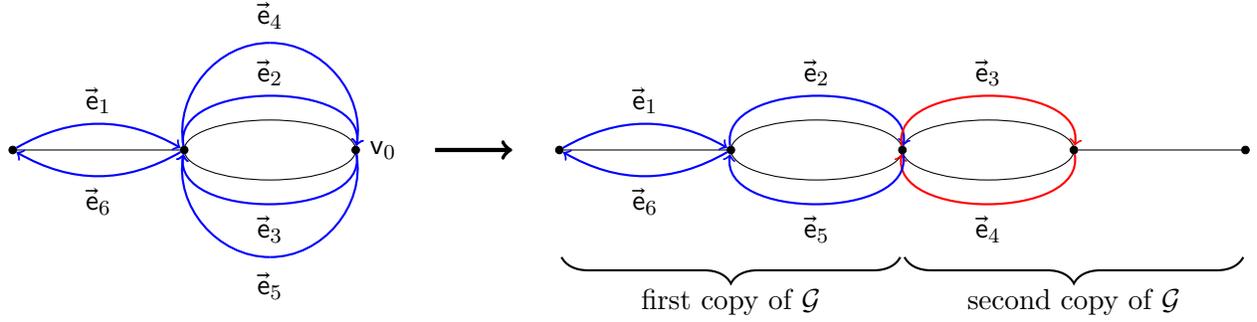
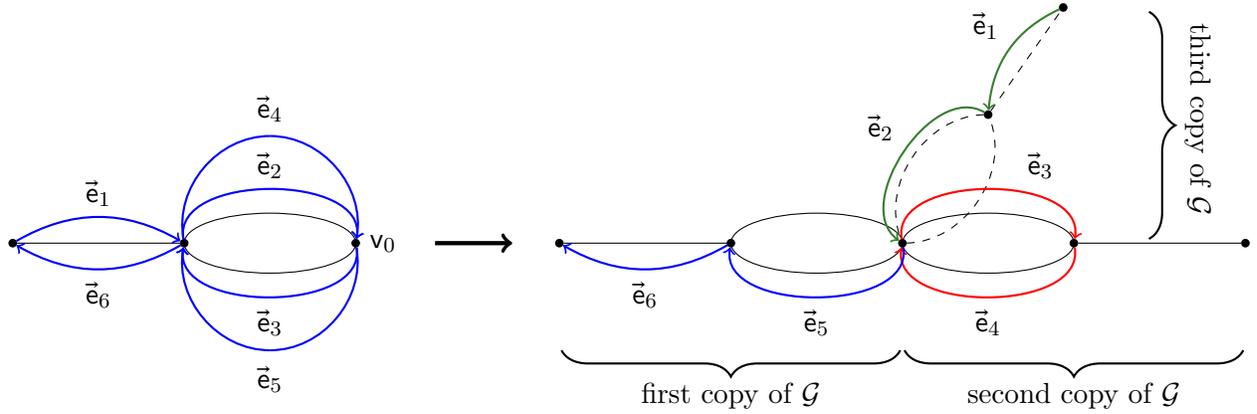
\begin{figure}[h]
        \begin{tikzpicture}[scale=0.57]
      \tikzset{enclosed/.style={draw, circle, inner sep=0pt, minimum size=.1cm, fill=black}, every loop/.style={}}
      \node[enclosed] (Z) at (0,2) {};
      \node[enclosed] (A) at (-4,2) {};
      \node[enclosed, label = {right: $\mv_0$}] (W) at (4,2) {};
      \node[enclosed, white, label = {above: $\vec{\me}_4$}] (W'up) at (2,4.5) {};
      \node[enclosed, white] (W'upr) at (2.1,4.5) {};
      \node[enclosed, white] (W'upl) at (1.9,4.5) {};
      \node[enclosed, white, label = {below: $\vec{\me}_5$}] (W'low) at (2,-0.5) {};
      \node[enclosed, white] (W'lowr) at (2.1,-0.5) {};
      \node[enclosed, white] (W'lowl) at (1.9,-0.5) {};
      \node[enclosed, white] (X) at (5.75,2) {};
      \node[enclosed, white] (Y) at (7.75,2) {};
      
      \node[enclosed] (Z') at (12.75,2) {};
      \node[enclosed] (Z'refl) at (20.75,2) {};
      \node[enclosed] (Z'refl2) at (18.75,5) {};
      \node[enclosed] (A') at (8.75,2) {};
      \node[enclosed] (A'refl) at (24.75,2) {};
      \node[enclosed] (A'refl2) at (20.5,7.5) {};
      \node[enclosed] (W') at (16.75,2) {};
      
      \draw[-,dashed] (W') edge[bend right = 55] node[above] {} (Z'refl2) node[midway, above] (edge1) {};   
      \draw[-,dashed] (Z'refl2) edge[bend right = 55] node[above] {} (W') node[midway, above] (edge1) {};     
       \draw[-,dashed] (Z'refl2) edge node[above] {} (A'refl2) node[midway, above] (edge1) {};  
      
      \draw[-] (Z) edge node[above] {} (A) node[midway, above] (edge1) {};
      \draw[-] (Z) arc [start angle=-180, end angle=180,
                  x radius=2cm, 
                  y radius=7mm] node[right] [pos=0.5] {};
     \draw[->, thick, blue] (A) edge[bend left] node[above] {\textcolor{black}{$\vec{\me}_1$}} (Z) node[midway, above] (edge1) {};
     \draw[-, thick, blue] (Z) edge[bend left = 50] node[below] {} (W'upr) node[midway, above] (edge1) {};
     \draw[->, thick, blue] (W'upl) edge[bend left = 50] node[below] {} (W) node[midway, above] (edge1) {};
      \draw[-, thick, blue] (W) edge[bend left = 50] node[above] {} (W'lowl) node[midway, above] (edge1) {};
       \draw[->, thick, blue] (W'lowr) edge[bend left = 50] node[above] {} (Z) node[midway, above] (edge1) {};
     \draw[->, thick, blue] (Z) edge[in=80, out=100] node[above] {\textcolor{black}{$\vec{\me}_2$}} (W) node[midway, above] (edge1) {};
     \draw[->, thick, blue] (W) edge[in=-100, out=-80] node[below] {\textcolor{black}{$\vec{\me}_3$}} (Z) node[midway, above] (edge1) {};
     \draw[->, thick, blue] (Z) edge[bend left] node[below] {\textcolor{black}{$\vec{\me}_6$}} (A) node[midway, above] (edge1) {};
     \draw[->, ultra thick] (X) edge node[below] {} (Y) node[midway, above] (edge1) {};
     
     \draw[-] (Z') edge node[above] {} (A') node[midway, above] (edge1) {};
     \draw[-] (Z'refl) edge node[above] {} (A'refl) node[midway, above] (edge1) {};
     \draw[-] (Z') arc [start angle=-180, end angle=180,
                  x radius=2cm, 
                  y radius=7mm] node[right] [pos=0.5] {};
      \draw[-] (Z'refl) arc [start angle=0, end angle=360,
                  x radius=2cm, 
                  y radius=7mm] node[right] [pos=0.5] {};
                  
\draw[->, thick, OliveGreen] (A'refl2) edge[bend right] node[above left] {\textcolor{black}{$\vec{\me}_1$}} (Z'refl2) node[midway, above] (edge1) {};
     \draw[->, thick, OliveGreen] (Z'refl2) edge[bend right = 87.5] node[above left] {\textcolor{black}{$\vec{\me}_2$}} (W') node[midway, above] (edge1) {};
     \draw[->, thick, red] (W') edge[in=80, out=100] node[above right] {\textcolor{black}{\quad $\vec{\me}_3$}} (Z'refl) node[midway, above] (edge1) {};
     \draw[->, thick, red] (Z'refl) edge[in=-100, out=-80] node[below] {\textcolor{black}{$\vec{\me}_4$}} (W') node[midway, above] (edge1) {};
     \draw[->, thick, blue] (W') edge[in=-100, out=-80] node[below] {\textcolor{black}{$\vec{\me}_5$}} (Z') node[midway, above] (edge1) {};
     \draw[->, thick, blue] (Z') edge[bend left] node[below] {\textcolor{black}{$\vec{\me}_6$}} (A') node[midway, above] (edge1) {};
     
     \draw [decorate,decoration={brace,amplitude=10pt},xshift=0pt,yshift=0pt, thick]
    (16.7,-0.5) -- (8.8,-0.5) node [black,midway,xshift=0cm,yshift=-0.6cm] 
    {first copy of $\Graph$};
     \draw [decorate,decoration={brace,amplitude=10pt},xshift=0pt,yshift=0pt, thick]
    (24.7,-0.5) -- (16.8,-0.5) node [black,midway,xshift=0cm,yshift=-0.6cm] 
    {second copy of $\Graph$};
    \draw [decorate,decoration={brace,amplitude=10pt},xshift=0pt,yshift=0pt, thick]
    (22.5,7.4) -- (22.5,2.1) node [black,midway,xshift=0.65cm,yshift=1.5cm, rotate=-90,anchor=west] 
    {third copy of $\Graph$};
     \end{tikzpicture}
     \caption{The lasso graph $\Graph$ from \autoref{fig:expl-m-fold-mirror-2} again together with the vertex $\mv_0$ and the directed path $\vec{p} = \big(\mv_-(\vec{p}), \vec{\me}_1, \vec{\me}_2, \vec{\me}_3, \vec{\me}_4, \vec{\me}_5, \vec{\me}_6, \mv_+(\vec{p})\big)$ with $\#_{\mv_0} \vec{p} = 2)$ (left), but now with its corresponding \emph{$3$-fold} mirrored graph (at $\mV_0 := \{ \mv_0 \}$) $\Graph^{(3)}(\mV_0)$ together with the path $\vec{p}(3,2,1)$ whose bonds \emph{before} the first hit belongs to the third copy (green), the bonds \emph{between} the first and the second hit of $\mv_0$ belong to the second copy (red), and the bonds after the second (i.e., the last) hit of $\mv_0$ belong to the first copy of $\Graph$ (blue).} \label{fig:expl-m-fold-mirror-3}
     \end{figure}     
     
     Now for a path $\vec{p} \in \mathcal{P}(\Graph)$ we introduce the sets
     \begin{align}\label{eq:decomposition-sets-m-fold-mirrored-graph}
\mathcal{P}_{\vec{p}}(\Graph^{(m)}(\mV_0)) := \Big\{ \vec{p}(m_1,m_2,\dots,m_{n+1}) \: : \: m_i \in \{ 1,\dots,m \}, \: i \in \{ 1,\dots,n+1 \} \Big\};
\end{align}
note that in the special case where $n=0$, $\mathcal{P}_{\vec{p}}(\Graph^{(m)}(\mV_0))$ consists of $m$ copies of $\vec{p}$, each of them lying in a different copy $\Graph_1,\dots,\Graph_m$ of $\Graph$.
Moreover, $\vec{p}(1,\dots,1)$ can be regarded as the embedded copy of $\vec p$ in $\Graph^{(m)}(\mV_0)$.

\begin{lemma}\label{lem:technical-lemma-attaching}
Let $m \in \mathbb{N}$, and let $\Graph$ satisfy \autoref{ass:graph}. Moreover, let $\mV_0 \subset \Graph \setminus \mVD$ and $\mW \subset \mV \setminus \mV_0$. Then the family $(\mathcal{P}_{\vec{p}}(\Graph^{(m)}(\mV_0)))_{\vec{p} \in \mathcal{P}_\mW(\Graph)}$ of subsets of $\mathcal{P}_{\mW^{(m)}}(\Graph^{(m)}(\mV_0))$ defined in \eqref{eq:decomposition-sets-m-fold-mirrored-graph} satisfies
\begin{itemize}
\item[(i)] $\mathcal{P}_{\vec{p}}(\Graph^{(m)}(\mV_0)) \cap \mathcal{P}_{\vec{q}}(\Graph^{(m)}(\mV_0)) = \emptyset$ for all $\vec{p}, \vec{q} \in \mathcal{P}_{\mW}(\Graph)$ with $\vec{p} \neq \vec{q}$,
\item[(ii)] $\bigcup\limits_{\vec{p} \in \mathcal{P}_\mW(\Graph)} \mathcal{P}_{\vec{p}}(\Graph^{(m)}(\mV_0)) = \mathcal{P}_{\mW^{(m)}}(\mathcal{G}^{(m)}(\mV_0))$,
\item[(iii)] $\ell(\vec{q})=\ell(\vec{p})$ for all $\vec{p} \in \mathcal{P}_\mW(\Graph)$ and all $\vec{q} \in \mathcal{P}_{\vec{p}}(\Graph^{(m)}(\mV_0))$, and
\item[(iv)] $\sum\limits_{\vec{q} \in \mathcal{P}_{\vec{p}}(\Graph^{(m)}(\mV_0))} \alpha_{\Graph^{(m)}(\mV_0)}(\vec{q}) = m\alpha_\Graph(\vec{p})$.
\end{itemize}
\end{lemma}
\begin{proof}
By construction, one immediately observes properties (i), (ii) and (iii); thus we have to verify the fourth condition: to this end, let 
\[
\vec{p} = \big( \mv_-(\vec{p}), \vec{\me}_1,\dots,\vec{\me}_k, \mv_+(\vec{p}) \big) \in \mathcal{P}_\mW(\Graph)
\]
(i.e. $\# \vec{p} =k$) with $n:= \#_{\mV_0} \vec
p$. {Also -- upon possibly subdividing edges so that $\mV_0$ only consists of vertices of $\Graph$ -- let 
\begin{equation}\label{eq:ekj}
\vec{\me}_{k(j)},\qquad k(j) \in \{1,\dots,\# \vec{p}-1\},\ j=1,\dots,n,
\end{equation}
denote the bond of $\vec{p}$ that hit the set $\mV_0$ (more precisely, $\partial^+(\vec{\me}_{k(j)}) \in \mV_0$) for the $j$-th time (note that $\vec{p}$ cannot end at $\mV_0$!): accordingly, $k(j)$ is the number of edges that the path $\vec{p}$ traverses before hitting $\mV_0$ for the $j$-th time. }

{Using the notation in~\autoref{defi:scatt-coeff},} let now $\beta_j := \beta(\vec{\me}_{k(j)}, \vec{\me}_{k(j)+1})$, $j=1,\dots,n$, be the contributing factor to the scattering coefficient $\alpha_\Graph(\vec{p})$ of $\vec{p}$ arising from $j$-th hit of $\mV_0$, and let
\begin{align*}
\widetilde{\alpha_\Graph}(\vec{p}):=\frac{\alpha_\Graph(\vec{p}) }{ \prod\limits_{j=1}^n \beta_j}.
\end{align*}
I.e., $\widetilde{\alpha_\Graph}(\vec{p})$ is the product consisting of the factors that contribute to the scattering coefficient $\alpha_\Graph(\vec{p})$ that do not arise from a point in $\mV_0$ (if $n=0$, the product $\prod_{j=1}^n \beta_j$ shall be interpreted as $1$): in other words, 
\begin{align*}
\widetilde{\alpha_\Graph}(\vec{p}) = \prod_{\substack{\ell=1 \\ \partial^+(\vec{\me}_\ell) \notin \mV_0}}^{k-1} \beta(\vec{\me}_\ell, \vec{\me}_{\ell+1}).
\end{align*}
 Now, any path in $\mathcal{P}_{\vec{p}}(\Graph^{(m)}(\mV_0))$ follows by construction the same bonds as $\vec{p}$ (possibly in a different copy of $\Graph$): therefore, the contributing factors to the respective scattering coefficients the arise from bonds that do not hit $\mV_0$ (as paths in $\Graph^{(m)}(\mV_0)$) are the same as for the original path $\vec{p}$. More precisely,
 \begin{align*}
 \widetilde{\alpha_\Graph}(\vec{p})= \prod_{\substack{ \ell=1 \\ \partial^+(\mf_\ell) \notin \mV_0}}^{\# \vec{q}-1} \beta\big(\vec{\mf}_\ell, \vec{\mf}_{\ell+1} \big)  \qquad  \text{for every $\vec{q} = \big( \mv_-(\vec{q}),\vec{\mf}_1,\dots, \vec{\mf}_{\# \vec{q}}, \mv_+(\vec{q})\big) \in \mathcal{P}_{\vec{p}}(\Graph^{(m)}(\mV_0))$}.
 \end{align*}
 Additionally, after each hit of $\mV_0$, any path in $\mathcal{P}_{\vec{p}}(\Graph^{(m)}(\mV_0))$ can run into $m$ different copies of $\Graph$. Furthermore, as each of them follows the same bonds as $\vec{p}$ possibly in a different copy, the corresponding contributing factor to the respective scattering coefficient arising from the $j$-th hit of $\mV_0$ does only depend on both copies of $\Graph$ the path crosses before and after the $j$-th hit, but not on the remaining itinerary of the path in $\Graph^{(m)}(\mV_0)$.
We can therefore write
\begin{align*}
\alpha_{\Graph^{(m)}(\mV_0)}(\vec{p}(m_1,m_2,\dots,m_{n+1})) = \widetilde{\alpha_\Graph}(\vec{p}) \prod_{j=1}^n \beta_j^{m_j,m_{j+1}}
\end{align*}
for all $m_i \in \{1,\dots,m\}$, $i \in \{1,\dots,n+1 \}$ with
\begin{align*}
\beta_j^{m_j,m_{j+1}} := \beta(\vec{\me}(m_1,m_2,\dots,m_{n+1})_{k(j)},\vec{\me}(m_1,m_2,\dots,m_{n+1})_{k(j)+1}) \qquad \text{for $j=1,\dots,n$},
\end{align*}
where $\vec{\me}(m_1,m_2,\dots,m_{n+1})_\ell$, $\ell=1,\dots,k$ denote the bonds of $\vec{p}(m_1,m_2,\dots,m_{n+1})$. (Note that the number of steps for a path in $\mathcal{P}_{\vec{p}}(\Graph^{(m)}(\mV_0))$ to hit $\mV_0$ is the same as for $\vec{p}$.) 
\begin{figure}[h]
        \begin{tikzpicture}[scale=0.6]
      \tikzset{enclosed/.style={draw, circle, inner sep=0pt, minimum size=.1cm, fill=black}, every loop/.style={}}
      \node[enclosed, label = {right: $\mv_0$}] (Z) at (0,2) {};
      \node[enclosed, label = {left: $\Graph \qquad$}] (A) at (-3,2) {};
      \node[enclosed] (B) at (-3,0) {};
      \node[enclosed] (C) at (-3,4) {};
      
      \node[enclosed] (Z') at (8.5,2) {};
      \node[enclosed] (A') at (5.5,2) {};
      \node[enclosed] (B') at (5.5,0) {};
      \node[enclosed] (C') at (5.5,4) {};
      \node[enclosed] (A'refl) at (9,5.25) {};
      \node[enclosed] (B'refl) at (11.75,2.5) {};
      \node[enclosed] (C'refl) at (10.5,3.875) {};
      \node[enclosed] (A'refl2) at (9,-1.25) {};
      \node[enclosed] (B'refl2) at (11.75,1.5) {};
      \node[enclosed] (C'refl2) at (10.5,0.125) {};
      \node[enclosed, white, label = {right: $\:\:\qquad\qquad \Graph^{(3)}(\{\mv_0\})$}] (E) at (12,2) {};
      
      \node[enclosed, white] (X) at (2,2) {};
      \node[enclosed, white] (Y) at (4,2) {};
      \draw[->, thick, blue] (C) edge[bend left] node[above right] {\textcolor{black}{$\vec{p}$}} (Z) node[midway, above] (edge1) {};
      \draw[->, thick, blue] (Z) edge[bend left] node[below left] {\textcolor{black}{$\beta_1 \:\:$}} (A) node[midway, above] (edge1) {};
      \draw[-] (Z) edge node[above] {} (A) node[midway, above] (edge1) {};
      \draw[-] (Z) edge node[above] {} (B) node[midway, above] (edge1) {};
      \draw[-] (Z) edge node[above] {} (C) node[midway, above] (edge1) {};
      \draw[->, ultra thick] (X) edge node[above] {} (Y) node[midway, above] (edge1) {};
      
      \draw[-] (Z') edge node[above] {} (A') node[midway, above] (edge1) {};
      \draw[-] (Z') edge node[above] {} (B') node[midway, above] (edge1) {};
      \draw[-] (Z') edge node[above] {} (C') node[midway, above] (edge1) {};
      
       \draw[-] (Z') edge node[above] {} (A'refl) node[midway, above] (edge1) {};
       \draw[-] (Z') edge node[above] {} (B'refl) node[midway, above] (edge1) {};
       \draw[-] (Z') edge node[above] {} (C'refl) node[midway, above] (edge1) {};
       
       \draw[-] (Z') edge node[above] {} (A'refl2) node[midway, above] (edge1) {};
       \draw[-] (Z') edge node[above] {} (B'refl2) node[midway, above] (edge1) {};
       \draw[-] (Z') edge node[above] {} (C'refl2) node[midway, above] (edge1) {};

       \draw [decorate,decoration={brace,amplitude=10pt},xshift=0pt,yshift=0pt, thick]
    (8.45,-2) -- (5.55,-2) node [black,midway,xshift=0cm,yshift=-0.6cm] 
    {\nth{1} copy};
       \draw [decorate,decoration={brace,amplitude=10pt},xshift=0pt,yshift=0pt, thick]
    (12.5,5.2) -- (12.5,2.55) node [black,midway,xshift=0.65cm,yshift=0.85cm, rotate=-90,anchor=west]
    {\nth{2} copy};
    \draw [decorate,decoration={brace,amplitude=10pt},xshift=0pt,yshift=0pt, thick]
    (12.5,1.45) -- (12.5,-1.2) node [black,midway,xshift=0.65cm,yshift=0.85cm, rotate=-90,anchor=west]
    {\nth{3} copy};
    
    \draw[->, thick, blue] (C') edge[bend left] node[above] {} (Z') node[midway, above] (edge1) {};
    \draw[->, thick, blue, dashed] (Z') edge[bend left] node[below left] {\tiny \textcolor{black}{$\beta_1^{1,1}$}} (A') node[midway, above] (edge1) {};
      \draw[->, thick, blue, dashed] (Z') edge[bend left] node[above] {\tiny \textcolor{black}{$\: \beta_1^{1,2}$}} (C'refl) node[midway, above] (edge1) {};
      
      \draw[->, thick, blue, dashed] (Z') edge[bend left] node[below right] {\tiny \textcolor{black}{$\quad \beta_1^{1,3}$}} (C'refl2) node[midway, above] (edge1) {};
     \end{tikzpicture}
     \caption{A $3$-star graph $\Graph$ together with a directed path $\vec{p}$ (left) with $\alpha_\Graph(\vec{p}) = \beta_1$ (i.e., $\widetilde{\alpha_\Graph}(\vec{p}) = 1$ in this case) and the corresponding $3$-fold mirrored graph $\Graph_3(\{\mv_0 \})$ together with corresponding paths $\vec{p}(1,1), \vec{p}(1,2), \vec{p}(1,3)$ (right) following the same bonds as $\vec{p}$ possibly changing the copy after hitting the vertex $\mv_0$. The induced contributing coefficients $\beta_1^{1,i}$, $i=1,2,3$ mark the respective copy that $\vec{p}$ visits after hitting $\mv_0$. A similar correspondence would arise considering respective copies of $\vec{p}$ lying in the second or third copy of $\Graph$ before hitting $\mv_0$, therefore leading to $9 = 3 \cdot 3$ possibilities of contributing coefficients at $\mv_0$.} \label{fig:3-star-to-9-star}
     \end{figure}
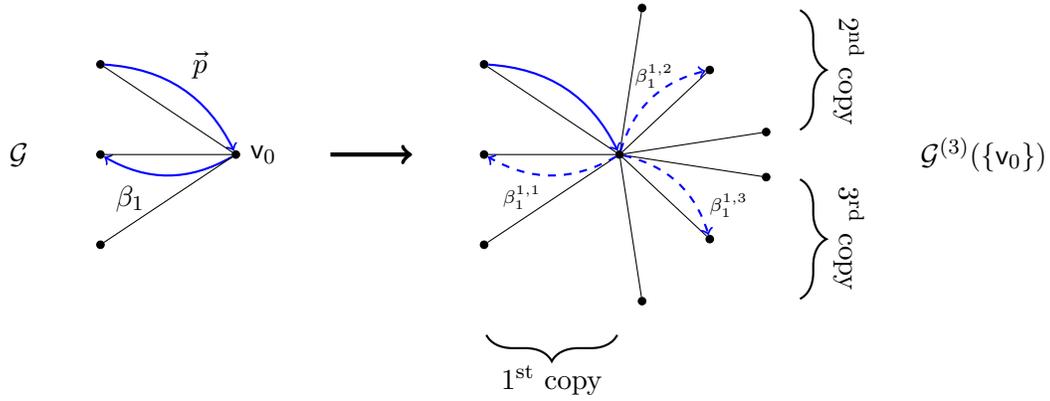

We claim
\begin{equation}\label{eq:sum-scattering-mirror}
\sum_{k=1}^m \beta_j^{m_{j},k} = \beta_j \qquad \hbox{for every }j=1,\dots,n:
\end{equation}
indeed, if $\vec{p}$ reflects at the $j$-th hit of $\mV_0$, that is, $\beta_j = \frac{2}{\deg_{\Graph}(\partial^+(\vec{\me}_{k(j)}))} - 1$,
then
\begin{align*}
\beta_j^{m_j,k} = \begin{cases} \frac{2}{\deg_{\Graph^{(m)}(\mV_0)}(\partial^+(\vec{\me}_{k(j)}))} - 1, & \text{if $m_j=k$,} \\ \frac{2}{\deg_{\Graph^{(m)}(\mV_0)}(\partial^+(\vec{\me}_{k(j)}))}, & \text{else,} \end{cases} 
\qquad \text{for $k =1,\dots,m$.}
\end{align*}
since the corresponding path $\vec{p}(m_1,\dots,m_n)$ also reflects at $\partial^+(\vec{\me}_{k(j)})$ whenever it stays in the \emph{same} copy of $\Graph$ before \emph{and} after the $j$-th hit of $\mV_0$, and transfers through $\partial^+(\vec{\me}_{k(j)})$, otherwise, (cf.\ \autoref{fig:3-star-to-9-star-part-2})

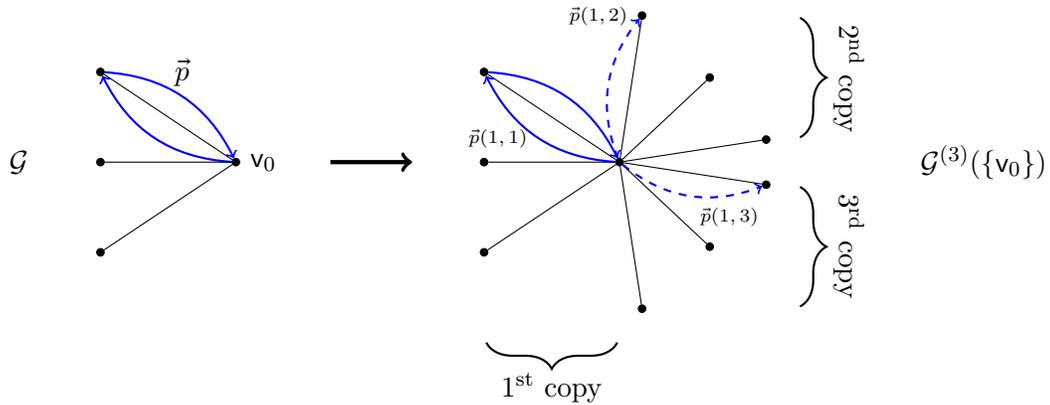
\begin{figure}[h]
        \begin{tikzpicture}[scale=0.6]
      \tikzset{enclosed/.style={draw, circle, inner sep=0pt, minimum size=.1cm, fill=black}, every loop/.style={}}
      \node[enclosed, label = {right: $\mv_0$}] (Z) at (0,2) {};
      \node[enclosed, label = {left: $\Graph \qquad$}] (A) at (-3,2) {};
      \node[enclosed] (B) at (-3,0) {};
      \node[enclosed] (C) at (-3,4) {};
      
      \node[enclosed] (Z') at (8.5,2) {};
      \node[enclosed] (A') at (5.5,2) {};
      \node[enclosed] (B') at (5.5,0) {};
      \node[enclosed] (C') at (5.5,4) {};
      \node[enclosed, label = {left: \tiny \textcolor{black}{$\vec{p}(1,2)$}}] (A'refl) at (9,5.25) {};
      \node[enclosed] (B'refl) at (11.75,2.5) {};
      \node[enclosed] (C'refl) at (10.5,3.875) {};
      \node[enclosed] (A'refl2) at (9,-1.25) {};
      \node[enclosed] (B'refl2) at (11.75,1.5) {};
      \node[enclosed] (C'refl2) at (10.5,0.125) {};
      \node[enclosed, white, label = {right: $\:\:\qquad\qquad \Graph^{(3)}(\{\mv_0\})$}] (E) at (12,2) {};
      
      \node[enclosed, white] (X) at (2,2) {};
      \node[enclosed, white] (Y) at (4,2) {};
      \draw[->, thick, blue] (C) edge[bend left] node[above] {\textcolor{black}{$\vec{p}$}} (Z) node[midway, above] (edge1) {};
      \draw[->, thick, blue] (Z) edge[bend left] node[below] {} (C) node[midway, above] (edge1) {};
      \draw[-] (Z) edge node[above] {} (A) node[midway, above] (edge1) {};
      \draw[-] (Z) edge node[above] {} (B) node[midway, above] (edge1) {};
      \draw[-] (Z) edge node[above] {} (C) node[midway, above] (edge1) {};
      \draw[->, ultra thick] (X) edge node[above] {} (Y) node[midway, above] (edge1) {};
      
      \draw[-] (Z') edge node[above] {} (A') node[midway, above] (edge1) {};
      \draw[-] (Z') edge node[above] {} (B') node[midway, above] (edge1) {};
      \draw[-] (Z') edge node[above] {} (C') node[midway, above] (edge1) {};
      
       \draw[-] (Z') edge node[above] {} (A'refl) node[midway, above] (edge1) {};
       \draw[-] (Z') edge node[above] {} (B'refl) node[midway, above] (edge1) {};
       \draw[-] (Z') edge node[above] {} (C'refl) node[midway, above] (edge1) {};
       
       \draw[-] (Z') edge node[above] {} (A'refl2) node[midway, above] (edge1) {};
       \draw[-] (Z') edge node[above] {} (B'refl2) node[midway, above] (edge1) {};
       \draw[-] (Z') edge node[above] {} (C'refl2) node[midway, above] (edge1) {};

       \draw [decorate,decoration={brace,amplitude=10pt},xshift=0pt,yshift=0pt, thick]
    (8.45,-2) -- (5.55,-2) node [black,midway,xshift=0cm,yshift=-0.6cm] 
    {\nth{1} copy};
       \draw [decorate,decoration={brace,amplitude=10pt},xshift=0pt,yshift=0pt, thick]
    (12.5,5.2) -- (12.5,2.55) node [black,midway,xshift=0.65cm,yshift=0.85cm, rotate=-90,anchor=west]
    {\nth{2} copy};
    \draw [decorate,decoration={brace,amplitude=10pt},xshift=0pt,yshift=0pt, thick]
    (12.5,1.45) -- (12.5,-1.2) node [black,midway,xshift=0.65cm,yshift=0.85cm, rotate=-90,anchor=west]
    {\nth{3} copy};
    
    \draw[->, thick, blue] (C') edge[bend left] node[above] {} (Z') node[midway, above] (edge1) {};
    \draw[->, thick, blue] (Z') edge[bend left] node[left] {\tiny \textcolor{black}{$\vec{p}(1,1)$}} (C') node[midway, above] (edge1) {};
      \draw[->, thick, blue, dashed] (Z') edge[bend left] node[above] {} (A'refl) node[midway, above] (edge1) {};
      
      \draw[->, thick, blue, dashed] (Z') edge[bend right] node[below right] {\tiny \textcolor{black}{$\vec{p}(1,3)$}} (B'refl2) node[midway, above] (edge1) {};
     \end{tikzpicture}
     \caption{The $3$-star graph $\Graph$ from \autoref{fig:3-star-to-9-star} together with a directed path $\vec{p}$ reflecting at $\mv_0$ (left) and the corresponding $3$-fold mirrored graph $\Graph_3(\{\mv_0 \})$ with corresponding paths $\vec{p}(1,1), \vec{p}(1,2), \vec{p}(1,3)$ (right).} \label{fig:3-star-to-9-star-part-2}
     \end{figure}
     
Now, 
as $\deg_{\Graph^{(m)}(\mV_0)}({ \partial^+(\vec{\me}_{k(j)})}) = m\deg_\Graph({ \partial^+(\vec{\me}_{k(j)})})$, it follows that
\begin{align*}
\sum_{k=1}^m \beta_j^{m_{j},k} &= \beta_j^{m_j,m_j} + \sum_{k=1, \, k \neq m_j}^m \beta_j^{m_{j},k} = \frac{2}{\deg_{\Graph^{(m)}(\mV_0)}(\partial^+(\vec{\me}_{k(j)}))} - 1 + \frac{2(m-1)}{\deg_{\Graph^{(m)}(\mV_0)}(\partial^+(\vec{\me}_{k(j)}))} \\&= \frac{2m}{\deg_{\Graph^{(m)}(\mV_0)}(\partial^+(\vec{\me}_{k(j)}))} - 1 = \frac{2}{\deg_{\Graph}(\partial^+(\vec{\me}_{k(j)}))} - 1 = \beta_j,
\end{align*}
showing \eqref{eq:sum-scattering-mirror} in this case. 

In the other case, where $\vec{p}$ transfers at the $j$-th hit of $\mV_0$, i.e., $\beta_j = \frac{2}{\deg_\Graph(\partial^+(\vec{\me}_{k(j)}))}$,
then -- regardless of the value of $m_j \in \{1,\dots,m\}$ -- the corresponding path $\vec{p}(m_1,\dots,m_{n+1})$ will also transfer at $\partial^+(\vec{\me}_{k(j)})$ as it stays at the same copy of $\Graph$ before and after the $j$-th hit; or it enters the other copy of $\Graph$ (and thus can not reflect) at $\mv_j$ (cf.\ \autoref{fig:3-star-to-9-star-part-3}).

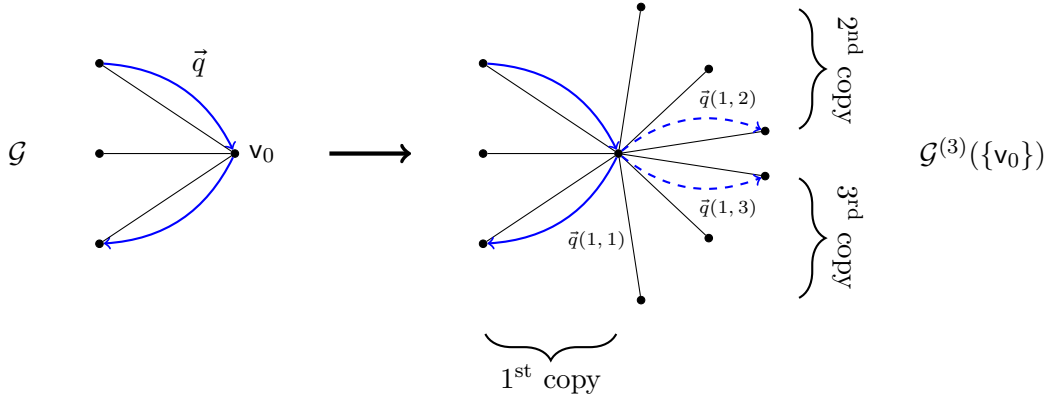
\begin{figure}[h]
        \begin{tikzpicture}[scale=0.6]
      \tikzset{enclosed/.style={draw, circle, inner sep=0pt, minimum size=.1cm, fill=black}, every loop/.style={}}
      \node[enclosed, label = {right: $\mv_0$}] (Z) at (0,2) {};
      \node[enclosed, label = {left: $\Graph \qquad$}] (A) at (-3,2) {};
      \node[enclosed] (B) at (-3,0) {};
      \node[enclosed] (C) at (-3,4) {};
      
      \node[enclosed] (Z') at (8.5,2) {};
      \node[enclosed] (A') at (5.5,2) {};
      \node[enclosed] (B') at (5.5,0) {};
      \node[enclosed] (C') at (5.5,4) {};
      \node[enclosed,] (A'refl) at (9,5.25) {};
      \node[enclosed] (B'refl) at (11.75,2.5) {};
      \node[enclosed] (C'refl) at (10.5,3.875) {};
      \node[enclosed] (A'refl2) at (9,-1.25) {};
      \node[enclosed] (B'refl2) at (11.75,1.5) {};
      \node[enclosed] (C'refl2) at (10.5,0.125) {};
      \node[enclosed, white, label = {right: $\:\:\qquad\qquad \Graph^{(3)}(\{\mv_0\})$}] (E) at (12,2) {};
      
      \node[enclosed, white] (X) at (2,2) {};
      \node[enclosed, white] (Y) at (4,2) {};
      \draw[->, thick, blue] (C) edge[bend left] node[above right] {\textcolor{black}{$\vec{q}$}} (Z) node[midway, above] (edge1) {};
      \draw[->, thick, blue] (Z) edge[bend left] node[below] {} (B) node[midway, above] (edge1) {};
      \draw[-] (Z) edge node[above] {} (A) node[midway, above] (edge1) {};
      \draw[-] (Z) edge node[above] {} (B) node[midway, above] (edge1) {};
      \draw[-] (Z) edge node[above] {} (C) node[midway, above] (edge1) {};
      \draw[->, ultra thick] (X) edge node[above] {} (Y) node[midway, above] (edge1) {};
      
      \draw[-] (Z') edge node[above] {} (A') node[midway, above] (edge1) {};
      \draw[-] (Z') edge node[above] {} (B') node[midway, above] (edge1) {};
      \draw[-] (Z') edge node[above] {} (C') node[midway, above] (edge1) {};
      
       \draw[-] (Z') edge node[above] {} (A'refl) node[midway, above] (edge1) {};
       \draw[-] (Z') edge node[above] {} (B'refl) node[midway, above] (edge1) {};
       \draw[-] (Z') edge node[above] {} (C'refl) node[midway, above] (edge1) {};
       
       \draw[-] (Z') edge node[above] {} (A'refl2) node[midway, above] (edge1) {};
       \draw[-] (Z') edge node[above] {} (B'refl2) node[midway, above] (edge1) {};
       \draw[-] (Z') edge node[above] {} (C'refl2) node[midway, above] (edge1) {};

       \draw [decorate,decoration={brace,amplitude=10pt},xshift=0pt,yshift=0pt, thick]
    (8.45,-2) -- (5.55,-2) node [black,midway,xshift=0cm,yshift=-0.6cm] 
    {\nth{1} copy};
       \draw [decorate,decoration={brace,amplitude=10pt},xshift=0pt,yshift=0pt, thick]
    (12.5,5.2) -- (12.5,2.55) node [black,midway,xshift=0.65cm,yshift=0.85cm, rotate=-90,anchor=west]
    {\nth{2} copy};
    \draw [decorate,decoration={brace,amplitude=10pt},xshift=0pt,yshift=0pt, thick]
    (12.5,1.45) -- (12.5,-1.2) node [black,midway,xshift=0.65cm,yshift=0.85cm, rotate=-90,anchor=west]
    {\nth{3} copy};
    
    \draw[->, thick, blue] (C') edge[bend left] node[above] {} (Z') node[midway, above] (edge1) {};
    \draw[->, thick, blue] (Z') edge[bend left] node[below right] {\hspace{-0.1cm}\tiny \textcolor{black}{$\vec{q}(1,1)$}} (B') node[midway, above] (edge1) {};
      \draw[->, thick, blue, dashed] (Z') edge[bend left] node[above right] {\tiny \textcolor{black}{$\vec{q}(1,2)$}} (B'refl) node[midway, above] (edge1) {};
      
      \draw[->, thick, blue, dashed] (Z') edge[bend right] node[below right] {\tiny \textcolor{black}{$\vec{q}(1,3)$}} (B'refl2) node[midway, above] (edge1) {};
     \end{tikzpicture}
     \caption{The $3$-star graph $\Graph$ from \autoref{fig:3-star-to-9-star} together with another directed path $\vec{q}$ transferring through $\mv_0$ (left) and the corresponding $3$-fold mirrored graph $\Graph_3(\{\mv_0 \})$ with corresponding paths $\vec{q}(1,1), \vec{q}(1,2), \vec{q}(1,3)$ (right), all having the same contributing coefficient at $\mv_0$.} \label{fig:3-star-to-9-star-part-3}
     \end{figure}
     
Consequently, $\beta_j^{m_{j},k} = \frac{2}{\deg_{\Graph^{(m)}(\mV_0)}(\partial^+(\vec{\me}_{k(j)}))}$ for every $k=1,\dots,m$, and thus
\[
\sum_{k=1}^m \beta_j^{m_{j},k} = \frac{2m}{\deg_{\Graph^{(m)}(\mV_0)}(\partial^+(\vec{\me}_{k(j)}))} = \frac{2}{\deg_\Graph(\partial^+(\vec{\me}_{k(j)}))} = \beta_j,
\]
yielding \eqref{eq:sum-scattering-mirror} also in this case. Therefore, using \eqref{eq:sum-scattering-mirror}, we deduce
\begin{align*}
\widetilde{\alpha_\Graph}(\vec{p})^{-1} \Bigg( \sum\limits_{\vec{q} \in \mathcal{P}_{\vec{p}}(\Graph^{(m)}(\mV_0))} \alpha_{\Graph^{(m)}(\mV_0)}(\vec{q})\Bigg) &= \widetilde{\alpha_\Graph}(\vec{p})^{-1} \Bigg( \sum\limits_{m_1=1}^{{m}} \dots \sum\limits_{m_n=1}^{{m}} \sum\limits_{m_{n+1} = 1}^{{m}} \alpha_{\Graph^{(m)}(\mV_0)}(m_1,\dots,m_n,m_{n+1}) \Bigg) \\&= \sum_{m_1=1}^{{m}} \dots \sum_{m_n=1}^{{m}} \sum_{m_{n+1} = 1}^{{m}} \bigg(\prod_{j=1}^n \beta_j^{m_j,m_{j+1}} \bigg) \\&= \sum_{m_1=1}^{{m}} \dots \sum_{m_n=1}^{{m}} \sum_{m_{n+1} = 1}^{{m}} \Bigg( \bigg(\prod_{j=1}^{n-1} \beta_j^{m_j,m_{j+1}} \bigg)\beta_n^{m_n,m_{n+1}} \Bigg) \\&= \sum_{m_1=1}^{{m}} \dots \sum_{m_n=1}^{{m}} \Bigg( \bigg(\prod_{j=1}^{n-1} \beta_j^{m_j,m_{j+1}} \bigg)\underbrace{ \sum_{m_{n+1} = 1}^m \beta_n^{m_n,m_{n+1}}}_{\overset{\eqref{eq:sum-scattering-mirror}}= \beta_n} \Bigg) \\&= \Bigg( \sum_{m_1=1}^{{m}} \dots \sum_{m_n=1}^{{m}} \bigg(\prod_{j=1}^{n-1} \beta_j^{m_j,m_{j+1}} \bigg) \Bigg)\beta_n,
\end{align*}
and, inductively, we eventually reach at
\begin{align}\label{eq:inductive-argument-reflecting}
\begin{aligned}
\widetilde{\alpha_\Graph}(\vec{p})^{-1} \Bigg(\sum\limits_{\vec{q} \in \mathcal{P}_{\vec{p}}(\Graph^{(m)}(\mV_0))} \alpha_{\Graph^{(m)}(\mV_0)}(\vec{q}) \Bigg) &= \bigg(\sum_{m_1=1}^{m} \underbrace{\sum_{m_2=1}^{m} \beta_1^{m_1,m_2}}_{\overset{\eqref{eq:sum-scattering-mirror}}= \beta_1} \bigg) \prod_{j=2}^n \beta_j 
\\&= \bigg( \sum_{m_1=1}^{m} \beta_1 \bigg) \prod_{j=2}^n \beta_j = m \, \prod_{j=1}^n \beta_j .
\end{aligned}
\end{align}
Hence, \eqref{eq:inductive-argument-reflecting} finally implies that
\begin{align*}
\sum_{\vec{q} \in \mathcal{P}_{\vec{p}}(\Graph^{(m)}(\mV_0))} \alpha_{\Graph^{(m)}(\mV_0)}(\vec{q}) = {m}\, \widetilde{\alpha_\Graph}(\vec{p}) \prod_{j=1}^n \beta_j = m \, \alpha_\Graph(\vec{p}),
\end{align*}
finishing the proof.
\end{proof}
\begin{proof}[Proof of \autoref{thm:mirroring-heat-content}]
    
    Clearly $\vert \Graph^{(m)}(\mV_0) \vert = m\vert \Graph \vert$ as well as $\# (\mVD)^{(m)} = m \# \mVD$. Therefore, according to \autoref{thm:heat-content-formula-bif-mug}, it remains to show that
    \begin{align}\label{eq:doubling-long-term-significant part}
    \mathfrak{L}_t(\Graph^{(m)}(\mV_0); (\mVD)^{(m)}) = m \mathfrak{L}_t(\Graph;\mVD) \qquad \text{for all $t>0$},
    \end{align}
    where we have adopted the notation in~\eqref{eq:ltgvd}. Indeed considering the corresponding family of subsets $(\mathcal{P}_{\vec{p}}(\Graph^{(m)}(\mV_0)))_{\vec{p} \in \mathcal{P}_\mVD(\Graph)}$ from \eqref{eq:decomposition-sets-m-fold-mirrored-graph} and \autoref{lem:technical-lemma-attaching}, we observe
  \begin{equation}\label{eq:final-proof-m}
  \begin{aligned}
 \mathfrak{L}_t(\Graph^{(m)}(\mV_0); (\mVD)^{(m)}) &= 4\sqrt{t}\sum_{\vec{p} \in \mathcal{P}_{(\mVD)^{(m)}}(\Graph^{(m)}(\mV_0))} \alpha(\vec{p}) H\bigg( \frac{\ell(\vec{p})}{2\sqrt{t}} \bigg) \\
 &= 4\sqrt{t}\sum_{\vec{p} \in \mathcal{P}_{\mVD}(\Graph)}\left( \sum_{\vec{q} \in \mathcal{P}_{\vec{p}}(\Graph^{(m)}(\mV_0))} \alpha(\vec{q}) H\bigg( \frac{\ell(\vec{q})}{2\sqrt{t}} \bigg) \right)\\&= 4\sqrt{t}\sum_{\vec{p} \in \mathcal{P}_{\mVD}(\Graph)} \Bigg(\sum_{\vec{q} \in \mathcal{P}_{\vec{p}}(\Graph^{(m)}(\mV_0))} \alpha(\vec{q})\Bigg) H\bigg( \frac{\ell(\vec{p})}{2\sqrt{t}} \bigg) \\&=  4m \sqrt{t} \sum_{\vec{p} \in \mathcal{P}_{\mVD}(\Graph)} \alpha(\vec{p}) H\bigg( \frac{\ell(\vec{p})}{2\sqrt{t}} \bigg) = m \mathfrak{L}_t(\Graph;\mVD)
  \end{aligned}
  \qquad \text{for all $t>0$,}
  \end{equation}
  which finishes the proof.
\end{proof}

\begin{exa}\label{exa:advan-sec-surg}
Let $m \in \mathbb{N}$. 

(i) If $\Graph$ is an equilateral star graph on $2m$ edges, each of length $\ell$, with $m$ outer Dirichlet, and $m$ outer Neumann conditions conditions, then taking into account \autoref{thm:mirroring-heat-content} the computations in Section~\ref{sec:three-examp}  yield
\[
\heatcont(\Graph;\mV_{\mathrm{D},\Graph}) = m\heatcont\bigg( \bigg[0,\frac{\vert \Graph \vert}{m} \bigg]; \{0\} \bigg) 
= \frac{{16m\ell }}{\pi^2} \sum_{k=0}^\infty \e^{-t\big( \frac{\pi (2k+1) }{4\ell} \big)^2} \frac{1}{(2k+1)^2}, \qquad t>0.
\]

(ii) {Following~\cite{BerKenKur17}, an $m$-regular pumpkin chain consists of $m$ copies of an interval $[0,\ell]$ such that all $m$ copies of the points at distance $r_j\in [0,\ell]$ from the endpoint 0 are glued together, with $0=r_0<r_1<\ldots<r_\ell=\ell$ and $j=0,1,\ldots,h$, see \autoref{fig:pumch}.}
Let now $\Graph$ be an $m$-regular pumpkin chain with Dirichlet condition at one or both ends: then
\[
\heatcont(\Graph;\mV_{\mathrm{D},\Graph}) = m\heatcont\bigg( \bigg[0,\frac{\vert \Graph \vert}{m} \bigg]; \mV_{\mathrm{D},\Graph} \bigg) = \frac{8 {m\ell}}{\pi^2} \sum_{k=0}^\infty \e^{-t\big( \frac{\pi (2k+1)(\#\mV_{\mathrm{D},\Graph})}{2{\ell}} \big)^2} \frac{1}{(2k+1)^2}, \qquad t>0.
\]
We stress that the heat content only depends on $\#\mV_{\mathrm{D},\Graph}$, ${\ell}$ and $m$, but not on the number $h$ of pumpkins $\Graph$ consists of.
\end{exa}

\begin{figure}[h]
\begin{tikzpicture}
\coordinate (a) at (0,0);
\coordinate (b) at (2,0);
\coordinate (c) at (6,0);
\coordinate (d) at (7,0);
\coordinate (e) at (9,0);
\draw (a) -- (b);
\draw[bend left]  (a) edge (b);
\draw[bend right]  (a) edge (b);
\draw (b) -- (c);
\draw[bend left]  (b) edge (c);
\draw[bend right]  (b) edge (c);
\draw (c) -- (d);
\draw[bend left]  (c) edge (d);
\draw[bend right]  (c) edge (d);
\draw (d) -- (e);
\draw[bend left]  (d) edge (e);
\draw[bend right]  (d) edge (e);
		\draw[fill=white] (a) circle (1.75pt);
		\draw[fill] (b) circle (1.75pt);
		\draw[fill] (c) circle (1.75pt);
		\draw[fill] (d) circle (1.75pt);
		\draw[fill] (e) circle (1.75pt);
\end{tikzpicture}
\caption{A $3$-regular pumpkin chain with $h=4$ and $\#\mV_{\mathrm{D}}=1$.}\label{fig:pumch}
\end{figure}

        \subsection{A Hadamard-type formula}
        Given a graph $\Graph$ as in \autoref{ass:graph} and fix an edge $\me_0 \in \mE$. We consider its \emph{length-perturbed} graph $\Graph_s$ for some $s > 0$ which is defined by the same topology as $\Graph$ but with edge lengths $(\ell_{\me,s})_{\me \in \mE}$ defined via
        \begin{align*}
        \ell_{\me,s} := \begin{cases} s + \ell_{\me}, & \text{if $\me = \me_0$}, \\ \ell_\me, & \text{if $\me \neq \me_0$,} \end{cases} \qquad \me \in \mE.
        \end{align*}
        Using now the combinatorial expansion in \eqref{eq:heat-content-formula-bif-mug}, we immediately reach at the following Hadamard-type formula for the heat content with respect to variation of the underlying edge lengths.
        \begin{theorem}\label{theo:hadamard}
        Let $\Graph$ a metric graph as in \autoref{ass:graph} and $\Graph_s$, $s>0$ some length-perturbation. Then, one has
        \begin{align}\label{eq:hadamard-heat-content}
        \frac{\mathrm{d}}{\mathrm{d}s}\bigg\vert_{s=0} \heatcont(\Graph_s;\mVD) &= 1- 2 \sum_{\vec p \in \mathcal{P}_\mVD(\Graph)} \alpha(\vec p) (\#_{\me_0} \vec{p}) \erfc\bigg(\frac{\ell( \vec p)}{2\sqrt{t}} \bigg), \qquad t > 0.
        \end{align}
        \end{theorem}
        Recall that $\#_{\me_0} \vec{p}$ denotes the number of time a directed path $\vec{p} \in \mathcal{P}_{\mVD}(\Graph)$ crosses $\me_0$.
        \begin{proof}
We are going to derive \eqref{eq:hadamard-heat-content} from \autoref{thm:heat-content-formula-bif-mug} and the fact that $H'(x) = -\erfc(x)$, noting that, for fixed $t>0$, the series appearing in the expansion for $\heatcont(\Graph_s)$ (with the derivative with respect to $s$ taken inside the sum) is uniformly convergent with respect to $s>0$. One can thus interchange the order of summation and derivation and reach at
        \begin{align}\label{eq:calc-hadamard}
        \begin{aligned}
        \frac{\mathrm{d}}{\mathrm{d}s} \heatcont(\Graph_s;\mVD) &= \frac{\mathrm{d}}{\mathrm{d}s} \Bigg( \vert \Graph_s \vert - \frac{2\sqrt{t}}{\sqrt{\pi}}\# \mVD + 4\sqrt{t} \sum_{\vec p \in \mathcal{P}_{\mVD}(\Graph_s)} \alpha(\vec p) H \bigg(\frac{\ell( \vec p)}{2\sqrt{t}} \bigg) \Bigg) \\&= \frac{\mathrm{d}}{\mathrm{d}s} \Bigg( \vert \Graph_s \vert + 4\sqrt{t} \sum_{\vec p \in \mathcal{P}_{\mVD}(\Graph)} \alpha(\vec p) H \bigg(\frac{\ell( \vec p) + (\#_{\me_0} \vec{p}) s}{2\sqrt{t}} \bigg) \Bigg) \\&= \frac{\mathrm{d}}{\mathrm{d}s} \Big( \vert \Graph \vert + s \Big) + 4\sqrt{t} \sum_{\vec{p} \in \mathcal{P}_{\mVD}(\Graph)} \alpha(\vec p) \, \Bigg( \frac{\mathrm{d}}{\mathrm{d}s}  H \bigg(\frac{\ell(\vec p)+ (\#_{\me_0} \vec{p})s}{2\sqrt{t}} \bigg) \Bigg) \\&= 1 - 2 \sum_{\vec p \in \mathcal{P}_{\mVD}(\Graph)} \alpha(\vec p) (\#_{\me_0} \vec{p}) \erfc \bigg(\frac{\ell(\vec p)+(\#_{\me_0} \vec{p})s}{2\sqrt{t}} \bigg)
        \end{aligned}
        \end{align}
        for all $s>0$, where we used the fact that $\alpha_{\Graph_s}(\vec{p}) = \alpha_{\Graph}(\vec{p})$ for all $\vec{p} \in \mathcal{P}_\mVD(\Graph_s)$ and $\mathcal{P}_{\mVD}(\Graph_s) \simeq \mathcal{P}_\mVD(\Graph)$ identifying each directed path $\vec{p} \in \mathcal{P}_{\mVD}(\Graph_s)$ with its canonical counterpart in $\mathcal{P}_\mVD(\Graph)$ following the same bonds in $\Graph$ as $\vec{p}$ in $\Graph_s$, with the difference that each bond running through $\me_0$ has length $\ell_{\me_0}$ rather than $s$. Moreover, one has
        \begin{align*}
        \frac{\mathrm{d}}{\mathrm{d}s} H\bigg( \frac{\ell(\vec p)+(\#_{\me_0} \vec{p})s}{2\sqrt{t}} \bigg) = -\frac{\#_{\me_0} \vec{p}}{2\sqrt{t}}\, \erfc\bigg( \frac{\ell(\vec p) + (\#_{\me_0} \vec{p})s}{2\sqrt{t}} \bigg) \quad \text{for $\vec p \in \mathcal{P}_\mVD(\Graph)$, $t>0$}.
        \end{align*}
        Plugging $s = 0$ into \eqref{eq:calc-hadamard} (taking into account that $\Graph_{0} = \Graph$) then yields \eqref{eq:hadamard-heat-content}.   
        \end{proof}
        As a consequence we deduce the following.
        \begin{corollary}[Short time behaviour under edge lengthening]\label{cor:heat-length-s}
        Let $\Graph$ be a metric graph as in \autoref{ass:graph} and $\widetilde{\Graph}$ the graph that arises by lengthening any of the edges. Then there exists $t_0 >0$ such that
        \[
        \heatcont(\Graph;\mVD) < \heatcont(\widetilde{\Graph};\mVD) \qquad \text{for all $0<t\leq t_0$.}
        \]
        \end{corollary}
        \begin{proof}
        It suffices to show that the sum in the right-hand side of \eqref{eq:hadamard-heat-content} converges to $0$ as $t \rightarrow 0^+$, because then the derivative with respect to edge lengthening will be strictly positive for sufficiently for all $t>0$ small enough: to this end, let again $0 < t < \frac{\ell_{\min}^2}{2\log d_{\max}}$. Now as 
        \begin{align}\label{eq:estimate-length-number-of-transfer}
        \frac{\#_{\me_0} \vec{p}}{\ell(\vec{p})} \leq \frac{\#_{\me_0} \vec{p}}{(\# \vec{p}) \ell_{\min}} \leq \frac{1}{\ell_{\min}} \qquad \text{for all $\vec{p} \in \mathcal{P}_{\mVD}(\Graph)$},
        \end{align}
        since any path $\vec{p} \in \mathcal{P}_{\mVD}(\Graph)$ can run through the edge $\me_0$ at most $\# p$ times, by \eqref{eq:estimate-erfc} and \eqref{eq:sim-borharjon} one finds
        \begin{align}\label{eq:helping-estimate-derivative}
        \begin{aligned}
        \bigg\vert \sum_{\vec p \in \mathcal{P}_\mVD(\Graph)} \alpha(\vec p) (\#_{\me_0} \vec{p}) \erfc\bigg(\frac{\ell( \vec p)}{2\sqrt{t}} \bigg) \bigg\vert 
        &\leq \frac{2\sqrt{t}}{\sqrt{\pi}}\sum_{\vec p \in \mathcal{P}_\mVD(\Graph)} \frac{\#_{\me_0} \vec{p}}{\ell(\vec{p})} \e^{-\frac{\ell( \vec p)^2}{4t}}
         \\&\leq \frac{4\sqrt{t}}{\sqrt{\pi}} \frac{\e^{\frac{\ell_{\min}^2}{4t}}}{\ell_{\min} \, d_{\max}}\sum_{n=1}^\infty \Big( d_{\max} \e^{-\frac{\ell_{\min}^2}{2t}} \Big)^{n} \\&= \frac{4\sqrt{t}}{\sqrt{\pi}} \frac{\e^{\frac{\ell_{\min}^2}{4t}}}{\ell_{\min}} \frac{\e^{-\frac{\ell_{\min}^2}{2t}}}{1-d_{\max}\e^{-\frac{\ell_{\min}^2}{2t}}} \\&= \frac{4\sqrt{t}}{\sqrt{\pi} \ell_{\min}} \frac{\e^{-\frac{\ell_{\min}^2}{4t}}}{1-d_{\max}\e^{-\frac{\ell_{\min}^2}{2t}}},
        \end{aligned}
\end{align}        
where the second estimate holds by \eqref{eq:estimate-length-number-of-transfer} in combination with \eqref{eq:estimate-comb-n-dmax}. Hence, \eqref{eq:helping-estimate-derivative} indeed implies that
\begin{align*}
\lim_{t \rightarrow 0^+} \sum_{\vec p \in \mathcal{P}_\mVD(\Graph)} \alpha(\vec p) (\#_{\me_0} \vec{p}) \erfc\bigg(\frac{\ell( \vec p)}{2\sqrt{t}} \bigg) &= 0,
\end{align*}
whence the claim follows.
        \end{proof}

If each edge on the underlying graph is scaled by a \emph{uniform} factor, more can be said.
\begin{proposition}[Graph scaling]
Let $\Graph$ be a metric graph satisfying \autoref{ass:graph} and $s\Graph$ the graph which arises through scaling each edge length of $\Graph$ by a constant factor $s > 0$. Then:
\begin{align}\label{eq:scalingheatkernel}
    p_{t}^{s\mathcal{G};\mVD}(x,y) =  \frac{1}{s} p_{\frac{t}{ s^2}}^{\mathcal{G};\mVD}\Big(\frac{x}{s},\frac{y}{s} \Big) \qquad \text{for every $x,y \in \Graph$ and every $t>0$,}
    \end{align}
    and in particular,
    \begin{align}\label{eq:scalingheatcontent}
    \mathcal{Q}_t(s\mathcal{G};\mVD) = s \mathcal{Q}_{\frac{t}{s^2}}(\mathcal{G};\mVD) \qquad \text{for all $t>0$.}
    \end{align}
\end{proposition}
\begin{proof}
The proof is an immediate consequence of the path sum formula for the heat kernel: indeed, one observes
    \begin{align*}
        p_{\frac{t}{s^2}}^{\mathcal{G};\mVD}\Big(\frac{x}{s},\frac{y}{s}\Big) &= \frac{1}{\sqrt{4 \pi t/s^2}} \sum_{p \in \mathcal{P}_{\mathcal{G}}(\frac{x}{s},\frac{y}{s})} \alpha_{\mathcal{G}}(p)\e^{-\frac{\ell(p)^2}{4t/s^2}} = \frac{1}{\sqrt{4 \pi t}} \sum_{p \in \mathcal{P}_{\mathcal{G}}(\frac{x}{s},\frac{y}{s})} \alpha_{\mathcal{G}}(p)\e^{-\frac{(s\ell(p))^2}{4t/s^2}} \\&= s\frac{1}{\sqrt{4 \pi t}} \sum_{p \in \mathcal{P}_{s\mathcal{G}}(x,y)} \alpha_{s\mathcal{G}}(p)\e^{-\frac{\ell(p)^2}{4t}} = s p_t^{s\mathcal{G};\mVD}(x,y),
    \end{align*}
    which yields \eqref{eq:scalingheatkernel} and thus eventually \eqref{eq:scalingheatcontent}.
\end{proof}

Note that one can also use the heat content formula from \autoref{thm:heat-content-formula-bif-mug} to infer \eqref{eq:scalingheatcontent} directly, without proving \eqref{eq:scalingheatkernel}.

\bibliographystyle{plain}
\bibliography{literatur}

\end{document}